\documentclass[10pt,reqno]{amsart}
\usepackage{amssymb}
\usepackage{amsmath, amssymb, amsthm}
\usepackage{mathrsfs}
\numberwithin{equation}{section}
\usepackage{color}
\usepackage{hyperref}
\newtheorem{theorem}{Theorem}[section]

\newtheorem{lemma}[theorem]{Lemma}

\newtheorem{corollary}[theorem]{Corollary}
\newtheorem{prop}{Proposition}[section]
\newcommand{\dif}{\mathrm{d}}
\newcommand{\ddiv}{\mathrm{div}}

\begin{document}
\title[Stability for three-dimensional relaxed CNS]{Asymptotic stability of planar viscous shock wave to three-dimensional relaxed compressible Navier-Stokes equations}
\author{Renyong Guan and Yuxi Hu}
 \thanks{\noindent  Renyong Guan,   Department of Mathematics, China University of Mining and Technology, Beijing, 100083, P.R. China, renyguan@163.com\\
\indent  Yuxi Hu, Department of Mathematics, China University of Mining and Technology, Beijing, 100083, P.R. China, yxhu86@163.com\\
 }
\begin{abstract}
This paper establishes the nonlinear time-asymptotic stability of shifted planar viscous shock waves for the three-dimensional relaxed compressible Navier-Stokes equations, in which a modified Maxwell-type model replaces the classical Newtonian constitutive relation. Under the assumptions of sufficiently small shock strength and initial perturbations, we prove that planar viscous shock waves are nonlinearly stable. The main steps of our analysis are as follows. First, using the relative entropy method together with the framework of $a$-contraction with shifts, we derive energy estimates for the weighted relative entropy of perturbations. We then successively obtain high-order and dissipation estimates via direct energy arguments, which provide the required a priori bounds. Combining these estimates with a local existence result, we establish the global asymptotic stability of the shifted planar viscous shock wave. Finally, we show that as the relaxation parameter tends to zero, solutions of the relaxed system converge globally in time to those of the classical Navier-Stokes system.
 \\[2em]

\end{abstract}
\maketitle
\section{Introduction}
The three-dimensional (3-D) isentropic compressible Navier-Stokes equations are governed by the fundamental principles of mass conservation and momentum balance as follows
\begin{equation}\label{1.1}
\begin{cases}
\partial_t\rho+\ddiv_x(\rho u)=0,\\
\partial_t (\rho u)+\ddiv_x(\rho u\otimes u)+\nabla_x p(\rho)=\ddiv_x \Pi,
\end{cases}
\end{equation}
where $(t, x)\in (0, +\infty)\times \Omega$ and  $\Omega$ is  an open subset of $\mathbb R^3$. Here, $\rho>0$, $u=(u_1, u_2, u_3)^T$, $p(\rho)$, $\Pi$ represent fluid density, velocity, pressure and stress tensor, respectively. In particular, the pressure $p(\rho)$ is assumed to satisfy the usual $\gamma$-law, $p(\rho)=A \rho^\gamma$ where $\gamma>1$ is the adiabatic index and $A$ is any positive constant. Without loss of generality, we set $A=1$ hereafter.

To close the system \eqref{1.1}, the stress $\Pi$ must be specified. For simple fluids, $\Pi$ is governed by the Newton constitutive relation
\begin{equation}\label{1.2}
\Pi=\mu\left(\nabla_x u+(\nabla_x u)^T-\frac{2}{3}\ddiv_x u \mathrm{I}_3\right)+\lambda\ddiv_x u\mathrm{I}_3,
\end{equation}
where the constants $\mu$ and $\lambda$ are shear and bulk viscosity coefficients, respectively. $\mathrm{I}_3$ denotes the $3\times 3$ identity matrix. The system $\eqref{1.1}$-$\eqref{1.2}$ is called classical compressible Navier-Stokes equations. For complex fluids, the Newtonian constitutive equation \eqref{1.2} is not valid.  Specially, for viscoelastic fluid, Maxwell \cite{MAX} introduced the following constitutive relation
\begin{equation}\nonumber
\tau\dot\Pi+\Pi=2\mu Du,
\end{equation}
 where $\dot\Pi=\Pi_t+u\cdot\nabla\Pi$ denotes the material derivative and $D u=\frac{1}{2}\left(\nabla_x u+(\nabla_x u)^T\right)$ the symmetric part of $\nabla u$. The relaxation parameter $\tau$ accounts for time lag in the stress tensor's response to the velocity gradient. Notably, even simple fluids, water for example, also exhibit a finite {\it time lag} (1 ps to 1 ns), see \cite{GM, FS}. Moreover, it was demonstrated that even for a simple fluid, this time lag cannot be neglected, as demonstrated in experiments measuring high-frequency (20 GHz) vibrations of nano-scale mechanical devices immersed in water-glycerol mixtures by Pelton et al. \cite{MP}.

On the other hand, from a mathematical perspective, Yong \cite{YWA} (see also Freist\"uhler \cite{FRE1, FRE2} for a Galilei-invariant model and M\"uller and Ruggeri\cite{muller1998}  within the framework of {\it{Rational Extended  Thermodynamics}}) rewrites the stress tensor $\Pi$ as $\Pi_1 + \Pi_2 \mathrm{I}_3$ and separately relaxes the components $\Pi_1$ and $\Pi_2$. Surprisingly, a physical explanation for this model has been provided by Chakraborty and Sarder \cite{DJE} through experiments on nano-scale mechanical devices vibrating in simple fluids. Their results demonstrate that this model best represents linear viscoelastic flow.

Motivated by the above works, we consider in this paper the following system:
\begin{equation}\label{x1.6}
\begin{cases}
\rho_t+\ddiv_x(\rho u)=0,\\
 \rho(\partial_t u+ u\cdot\nabla_x u) +\nabla_x p(\rho)=\ddiv_x \Pi_1+\nabla_x\Pi_2,\\
\tau \rho \left(\partial_t\Pi_1+ u\cdot\nabla_x \Pi_1\right)+\Pi_1=
   \mu\left(\nabla_x u+(\nabla_x u)^T-\frac{2}{3}\ddiv_x u \mathrm{I}_3\right),\\
\tau \rho \left(\partial_t\Pi_2+ u\cdot\nabla_x \Pi_2\right)+\Pi_2=\lambda\ddiv_x u,
\end{cases}
\end{equation}
for $(t, x)\in \mathbb R^+\times \Omega$ with $\Omega=\mathbb R\times \mathbb T^2$ and $\mathbb T^2=(\mathbb R/\mathbb Z)^2$. The initial data are given by
\begin{align} \label{1.7}
(\rho, u, \Pi_1, \Pi_2)(0,x)=(\rho_0, u_0, \Pi_{10}, \Pi_{20})(x)
\rightarrow(\rho_{\pm}, u_{\pm}, \mathrm O_{3\times3}, 0) \quad (x_1\rightarrow\pm\infty),
\end{align}
where $\Pi_{10}$ is symmetric and traceless, $\rho_\pm>0$, $u_\pm=(u_{1\pm}, 0, 0)$ and $\mathrm O_{3\times3}$ is the zero matrix of order three. 

As noted in \cite{MNS, GD, WY, WWS, GHSR, GHSS}, the large-time behavior of solutions to system \eqref{1.1}-\eqref{1.2} or \eqref{x1.6} with initial data \eqref{1.7} is closely
related to the planar Riemann problem of the associated 3-D Euler system
\begin{equation}\label{1.8}
\begin{cases}
\partial_t\rho+\ddiv_x(\rho u)=0,\\
\partial_t(\rho u)+\ddiv_x(\rho u\otimes u)+\nabla_x p(\rho)=0,
\end{cases}
\end{equation}
with the Riemann initial data
\begin{equation}\label{1.9}
(\rho,u)(t=0,x)=
\begin{cases}
(\rho_-,u_-),\quad x_1<0,\\
(\rho_+,u_+),\quad x_1>0.
\end{cases}
\end{equation}
Note that the planar shock wave solutions of \eqref{1.8}-\eqref{1.9} satisfy the following one-dimensional (1-D)  Riemann problem of Euler system
\begin{equation}\label{1.10}
\begin{cases}
\partial_t\rho+\partial_{x_1}(\rho u_1)=0,\\
\partial_t(\rho u_1)+\partial_{x_1}(\rho u_1^2+ p(\rho))=0,\\
(\rho,u_1)(t=0,x_1)=
\begin{cases}
(\rho_-,u_{1-}),\quad x_1<0,\\
(\rho_+,u_{1+}),\quad x_1>0.
\end{cases}
\end{cases}
\end{equation}
By considering 2-shock wave (the 1-shock wave case is analogous), the states $(\rho_\pm, u_{1\pm})$ satisfy the Rankine-Hugonint (R-H) condition
\begin{equation}\label{1.11rh}
\begin{cases}
-\sigma(\rho_+-\rho_-)+(\rho_+u_{1+}-\rho_-u_{1-})=0,\\
-\sigma(\rho_+u_{1+}-\rho_-u_{1-})+(\rho_+u_{1+}^2-\rho_-u_{1-}^2)+(p(\rho_+)-p(\rho_-))=0,
\end{cases}
\end{equation}
and the Lax entropy condition
\begin{equation}\label{1.12}
\rho_->\rho_+, \qquad u_{1-}>u_{1+},
\end{equation}
where $\sigma$ is shock speed.

When $\tau=0$, system \eqref{x1.6} reduces to the classical 3-D compressible Navier-Stokes equations, whose asymptotic behavior, especially in 1-D, has been extensively studied. Specifically, for shock profile initial data, by using anti-derivative method, Matsumura and Nishihara \cite{MNS}, and Goodman \cite{GD} first established the stability of traveling waves under small initial disturbances with zero mass condition. Subsequent works by Liu \cite{LTP}, Szepessy and Xin \cite{SXIN}, Liu and Zeng \cite{LTPZ} extended these results by introducing constant shifts to remove the zero mass condition. Huang and Matsumura \cite{HM} further demonstrated the asymptotic stability of composite wave consisting of two viscous shocks for the 1-D full compressible Navier-Stokes equations, provided the shock strengths are small and of the same order.

However, the method of anti-derivative proves to be incompatible when dealing with composite waves comprising both viscous shock and rarefaction. In this context, a new approach based on the relative entropy method and the theory of $a$-contraction with shifts has provided a powerful framework to address this issue. This theory was initially introduced by Bresch and Desjardins \cite{BD} and further developed in a series works including \cite{KV1, KV3, KV6, KV8, KV10, KV11, KV2, V27, V28}. Recently, Kang, Vasseur and Wang \cite{WY} successfully applied this framework to resolve the stability problem for composite waves of viscous shock and rarefaction. Subsequently, Han, Kang and Kim \cite{SMJ} employed similar techniques to establish the stability of composite waves of two viscous shocks with independently small strengths. Furthermore, Wang and Wang \cite{WWS} extended the relative entropy and the $a$-contraction with shifts theory to prove the stability of planar viscous shock wave for 3-D compressible Navier-Stokes equations. Building upon the framework established in these work, the present paper aims to investigate the nonlinear asymptotic stability of planar viscous shock wave for the 3-D relaxed compressible Navier-Stokes equations \eqref{x1.6}. For other related results, we refer the reader to \cite{MNR1, MNR2, LW, LWW}.

When $\tau>0$, for the one dimensional version of system \eqref{x1.6}, Hu-Wang \cite{ZWH} and Hu-Wang \cite{XFH} analyzed the linear stability of the viscous shock wave and the nonlinear stability of rarefaction waves, respectively.  Freist\"uhler \cite{FRE1} obtained the nonlinear stability of viscous shock waves under shock profile initial data. Recently, Guan and Hu \cite{GHSR, GHSS} demonstrated the time-asymptotic stability of two composite waves including rarefaction and viscous shock as well as two viscous shocks by using the methods of $a$-contraction with shift and relative entropy methods.  

By introducing the specific volume variable $v=\rho^{-1}$,  we reformulate the problem \eqref{x1.6}-\eqref{1.7} as
\begin{equation}\label{1.6}
\begin{cases}
\rho(\partial_t v+u\cdot\nabla_x v)=\ddiv_x u,\\
 \rho(\partial_t u+ u\cdot\nabla_x u) +\nabla_x p(\rho)=\ddiv_x \Pi_1+\nabla_x\Pi_2,\\
\tau \rho \left(\partial_t\Pi_1+ u\cdot\nabla_x \Pi_1\right)+\Pi_1=
   \mu\left(\nabla_x u+(\nabla_x u)^T-\frac{2}{3}\ddiv_x u \mathrm I_3\right),\\
\tau \rho \left(\partial_t\Pi_2+ u\cdot\nabla_x \Pi_2\right)+\Pi_2=\lambda\ddiv_x u,
\end{cases}
\end{equation}
with initial conditions
\begin{align} \label{1.13}
(v, u, \Pi_1, \Pi_2)(0,x)=(v_0, u_0, \Pi_{10}, \Pi_{20})(x)
\rightarrow(v_{\pm}, u_{\pm}, \mathrm O_{3\times3}, 0) \quad (x_1\rightarrow\pm\infty),
\end{align}
where $v_0(x)=\frac{1}{\rho_0(x)},v_\pm=\frac{1}{\rho_\pm}>0$.

This paper is devoted to the study of both the nonlinear asymptotic stability of planar viscous shock waves and the relaxation limit for system \eqref{1.6}, subject to the initial data \eqref{1.13}. We proceed in several steps. First, we derive the planar viscous shock solution (Lemma \ref{pvsw}) and introduce the shift and weight functions (\eqref{shiftx}, \eqref{weighta}). Next, we perform a weighted relative entropy estimate using standard energy methods. A key difficulty arises here: the dissipation in the relaxed system is insufficient to control $(v-v^s)$. To overcome this, we augment the analysis by artificially adding the first-order dissipation of $(p(v)-p(v^s))$, which enables the application of the multidimensional Poincar\'e inequality (\cite{WWS}, Lemma \ref{poinbds}) to secure the necessary control. Crucially, the added term can finally be bounded by synthesizing the entropy estimates with the first-order energy and dissipation estimates. To close the energy estimates, we further establish the required $H^3$ bounds. This allows us to conclude the global stability of the planar viscous shock wave. We emphasize that all the preceding estimates are independent of the relaxation parameter $\tau$. Consequently, by a standard compactness argument, we demonstrate that as $\tau$ tends to zero, the solutions of the relaxed system converge globally to those of the classical Navier-Stokes equations.


Our main theorem are states as follows:

\begin{theorem}\label{th1.1}
Assume the constant states $(v_\pm, u_{1\pm}) \in \mathbb{R}^+ \times \mathbb{R}$ satisfy \eqref{1.11rh}-\eqref{1.12}, and the relaxation parameter $\tau$ satisfies
\begin{equation}\label{sccs}
\tau\leq \min\{\inf\limits_{z\in[v_-,v_+]}\frac{\frac{4\mu}{3}+\lambda}{2|\sigma_\ast^2+p^{\prime}(z)|},
1\},
\end{equation}
 where $\sigma_\ast=\sqrt{\frac{p(v_+)-p(v_-)}{v_--v_+}}$. Let $(v^s, u^s, \Pi_1^s, \Pi_2^s)$ be the planar 2-viscous shock solutions of \eqref{1.6}-\eqref{1.13}. Then, there exist constants $\delta_0,\varepsilon_0>0$ such that for any initial data $(v_0, u_0, \Pi_{10}, \Pi_{20})$ satisfying
\begin{equation}\label{csz}
\sum\limits_{\pm}\left(\|(v_0-v_{\pm},u_0-u_{\pm})\|_{L^2(\mathbb{R}_\pm\times\mathbb{T}^2)}\right)
+\|\nabla_x(v_0,u_0)\|_{H^2(\Omega)}
+\sqrt{\tau}\|(\Pi_{10}, \Pi_{20})\|_{H^3(\Omega)}<\varepsilon_0,
\end{equation}
where $\mathbb{R}_+:=-\mathbb{R}_-=(0,+\infty)$ and $|v_--v_+|\leq\delta_0$, the initial value problem \eqref{1.6}-\eqref{1.13} admits a unique global-in-time solution $(v, u, \Pi_1, \Pi_2)\in C^1((0, +\infty)\times \Omega)$. Moreover, there exist an absolutely continuous shift $X(t)$ such that
\begin{equation}\label{lxx1}
\begin{aligned}
&v(t,x)-v^s(x_1-\sigma t-X(t))\in C(0,+\infty;H^3(\Omega)),\\
&u(t,x)-u^s(x_1-\sigma t-X(t))\in C(0,+\infty;H^3(\Omega)),\\
&\Pi_1(t,x)-\Pi_1^s(x_1-\sigma t-X(t))\in C(0,+\infty;H^3(\Omega),\\
&\Pi_2(t,x)-\Pi_2^s(x_1-\sigma t-X(t))\in C(0,+\infty;H^3(\Omega),
\end{aligned}
\end{equation}
and the following uniform estimates hold
\begin{equation}\label{lxx2}
\begin{aligned}
&\sup\limits_{t\in[0,+\infty)}\|(v-v^s, u-u^s, \sqrt\tau(\Pi_1-\Pi_1^s), 
\sqrt\tau(\Pi_2-\Pi_2^s))\|^2_{H^3(\Omega)}\\
&\qquad+\int_0^{+\infty}\left(\|\left(\nabla_x(v-v^s),\nabla_x (u-u^s)\right)\|^2_{H^2(\Omega)}+\|(\Pi_1-\Pi_1^s, \Pi_2-\Pi_2^s)\|^2_{H^3(\Omega)}\right)\dif t\\
&\leq C_0\left(\|(v_0-v^s(\cdot), u_0-u^s(\cdot))\|^2_{H^3(\Omega)}
+\tau\|(\Pi_{10}-\Pi_1^s(\cdot), \Pi_{20}-\Pi_2^s(\cdot))\|^2_{H^3(\Omega)}\right),
\end{aligned}
\end{equation}
where $C_0$ is a universal constant independent of $\tau$. Additionally, the time-asymptotic stability holds
\begin{equation}\label{jjwdx}
\begin{aligned}
&\lim\limits_{t\rightarrow+\infty}\sup\limits_{x\in\Omega}|(v, u)(t,x)-(v^s, u^s)(x_1-\sigma t-X(t))|=0,\\
&\lim\limits_{t\rightarrow+\infty}\sup\limits_{x\in\Omega}\sqrt{\tau}|(\Pi_1, \Pi_2)(t,x)-(\Pi_1^s, \Pi_2^s)(x_1-\sigma t-X(t))|=0,\\
\end{aligned}
\end{equation}
with
\begin{equation}\label{1.11}
\lim\limits_{t\rightarrow+\infty}|\dot{X}(t)|=0.
\end{equation}
\end{theorem}

Furthermore, based on the uniform estimates \eqref{lxx2}, we have the following global convergence theorem.
\begin{theorem}\label{th1.2}
Let $(v^{\tau}, u^{\tau}, \Pi_1^{\tau}, \Pi_2^{\tau})$ be the global solutions obtained in Theorem \ref{th1.1}. Then, there exists functions $(v^0 ,u^0)\in L^{\infty}\left((0, +\infty);H^3(\Omega)\right)$ and $(\Pi_1^0, \Pi_2^0)\in L^2\left((0, +\infty);H^3(\Omega)\right)$, such that, as $\tau\rightarrow0$
\begin{align*}
(v^{\tau}, u^{\tau})\rightharpoonup (v^0, u^0) \qquad weak-*\quad in \quad L^{\infty}\left((0, +\infty);H^3(\Omega)\right),\\
(\Pi_1^{\tau}, \Pi_2^{\tau})\rightharpoonup (\Pi_1^0, \Pi_2^0) \qquad weakly- \quad in \quad L^2\left((0, +\infty);H^3(\Omega)\right),
\end{align*}
where $(v^0, u^0)$ is the solution to the classical 3-D isentropic compressible Navier-Stokes equations \eqref{1.2} with initial value $(v_0, u_0)$. Moreover,
\[
\Pi_1^0=\mu\left(\nabla_x u^0+(\nabla_x u^0)^T-\frac{2}{3}\ddiv_x u^0\mathrm I_3\right),\quad
\Pi_1^0=\lambda \ddiv_x u^0.
\]
\end{theorem}

This paper is organized as follows. In  Section 2, we  introduce basic concept, including planar 2-viscous shock waves, the shift $X(t)$, and the weighted function $a(t, x_1)$. The a priori estimates (Proposition \ref{p1}) are present in Section 3,  which immediately yield Theorem \ref{th1.1}.  The  proof of  the a priori estimates (Proposition \ref{p1}) are given in Section 4. Finally, Section 5 proves the global-in-time convergence of solutions for the relaxed system \eqref{1.6} to those of the classical system \eqref{1.2}, thus finishing the proof of Theorem \ref{th1.2}.

\textbf{Notations:}  $L^p(\Omega)$ and $W^{k,p}(\Omega)$  ($1\le p \le\infty$) denote the usual Lebesgue and Sobolev spaces over $\Omega$ with the norm $\|\cdot \|_{L^p}$ and $\|\cdot\|_{W^{s,p}}$, respectively. Note that, when $k=0$, $W^{0,p}=L^p$. For $p=2$, $W^{k, 2}$ are abbreviated to $H^k$ as usual.
Let $T$ and $B$ be a positive constant and a Banach space, respectively. $C^i(0,T; B)(i \ge 0 )$ denotes the space of $B$-valued $i$-times continuously differentiable functions on $[0,T]$, and $L^p(0,T; B)$ denotes the space of $B$-valued $L^p$-functions on $[0,T]$. The corresponding space $B$-valued functions on $[0,\infty)$ are defined in an analogous manner. For $n\times n$ matrices $A=(a_{ij}), B=(b_{ij})$, we denote $A:B=\sum_{i=1}^{n}\sum_{j=1}^{n}a_{ij}b_{ij}$.

\section{Preliminaries}

\subsection{Planar viscous shock wave}

Firstly, we demonstrate the existence of planar viscous shock wave solutions for system \eqref{1.6}. Let $\xi=(\xi_1, \xi_2, \xi_3)$ with $\xi_1=x_1-\sigma t$ and $\xi_i=x_i$ $(i=2,3)$. Correspondingly, the planar 2-viscous shock wave profile is given by $(\rho^s, u^s, \Pi_1^s, \Pi_2^s)(\xi_1)$, where the velocity field takes the form $u^s(\xi_1):=(u_1^s(\xi_1), 0, 0)^T$ and $\Pi^s_1(\xi_1)=(\Pi^s_{ij}(\xi_1))_{3\times3}$.

Next, assume the functions $(\rho^s, u^s, \Pi_1^s, \Pi_2^s)(\xi_1)$ satisfies the far-field boundary conditions
\begin{equation}\label{2.1}
(\rho^s, u^s, \Pi_1^s, \Pi_2^s)(\xi_1)\rightarrow (\rho_\pm, u_\pm, O_{3\times3}, 0),\qquad (\xi_1\rightarrow\pm\infty).
\end{equation}
 Substituting $(\rho^s, u^s, \Pi_1^s, \Pi_2^s)(\xi_1)$ into $\eqref{x1.6}_1$ and \eqref{1.6} yields the following ordinary differential equations
\begin{equation}\label{2.2}
-\sigma\rho^s_{\xi_1}+(\rho^s u_1^s)_{\xi_1}=0,
\end{equation}
\begin{equation}\label{2.3}
\begin{cases}
\rho^s\left(-\sigma v^s_{\xi_1}+ u_1^sv^s_{\xi_1}\right)=(u_1^s)_{\xi_1},\\
\rho^s\left(-\sigma( u_1^s)_{\xi_1}+u_1^s(u_1^s)_{\xi_1}\right)+p(v^s)_{\xi_1}=
                            (\Pi_{11}^s)_{\xi_1}+(\Pi_2^s)_{\xi_1},\\
\tau \rho^s \left(-\sigma(\Pi^s_{1})_{\xi_1}+ u_1^s(\Pi^s_{1})_{\xi_1}\right)+\Pi^s_{1}=
                          \text{diag}\{ \frac{4\mu}{3}(u_1^s)_{\xi_1}, -\frac{2\mu}{3}(u_1^s)_{\xi_1}, -\frac{2\mu}{3}(u_1^s)_{\xi_1}\},\\
\tau \rho^s \left(-\sigma(\Pi_2^s)_{\xi_1}+ u_1^s(\Pi^s_2)_{\xi_1}\right)+\Pi^s_2=\lambda(u_1^s)_{\xi_1},
\end{cases}
\end{equation}
 Integrating \eqref{2.2} from $(-\infty, \xi_1)$, we have
\begin{equation}\label{2.5}
-\sigma\rho^s+\rho^s u_1^s=-\sigma\rho_-+\rho_- {u_1}_-=:-\sigma_{\ast}.
\end{equation}
On the other hand, a direct calculation in $\eqref{2.3}_3$ gives
\begin{equation}\label{2.4}
\Pi_{22}^s(\xi_1)=\Pi_{33}^s(\xi_1)=-\frac{1}{2}\Pi^s_{11}(\xi_1),\quad \Pi^s_{ij}(\xi_1)=0 (i\neq j).
\end{equation}
Therefore, the system \eqref{2.3} can be rewritten as
\begin{equation}\label{2.6}
\begin{cases}
-\sigma_{\ast}v^s_{\xi_1}=(u_1^s)_{\xi_1},\\
-\sigma_{\ast}( u_1^s)_{\xi_1}+p(v^s)_{\xi_1}=
                            (\Pi_{11}^s)_{\xi_1}+(\Pi_2^s)_{\xi_1},\\
-\tau\sigma_{\ast} (\Pi^s_{11})_{\xi_1}+\Pi^s_{11}=\frac{4\mu}{3}(u_1^s)_{\xi_1},\\
-\tau\sigma_{\ast} (\Pi_2^s)_{\xi_1}+\Pi^s_2=\lambda(u_1^s)_{\xi_1},
\end{cases}
\end{equation}
with the far field condition \eqref{2.4}. Further integrating the equations $\eqref{2.6}_1$ and $\eqref{2.6}_2$ with respect to $\xi_1$, it holds
\begin{equation}\label{2.7}
\begin{cases}
\sigma_\ast v^s+u_1^s=\sigma_\ast v_-+u_{1-}=\sigma_\ast v_++u_{1+},\\
\Pi^s_{11}+\Pi^s_2=-\sigma_\ast(u_1^s-u_{1-})+(p(v^s)-p(v_-)).
\end{cases}
\end{equation}
In particularly, we can get
\begin{equation}\label{2.8}
\begin{cases}
\sigma_\ast v_-+u_{1-}=\sigma_\ast v_++u_{1+},\\
-\sigma_\ast(u_{1+}-u_{1-})+(p(v_+)-p(v_-))=0.
\end{cases}
\end{equation}
Thus, we imply that $\sigma_\ast=\sqrt{\frac{p(v_-)-p(v_+)}{v_+-v_-}}>0$ for 2-shock wave.
Subsequently, combining $\eqref{2.6}_3$ and $\eqref{2.6}_4$, we obtain that
\begin{equation}\label{2.9}
-\tau\sigma_\ast\Big((\Pi^s_{11})_{\xi_1}+(\Pi^s_2)_{\xi_1}\Big)+(\Pi^s_{11}+\Pi^s_2)
=(\frac{4\mu}{3}+\lambda)(u_1^s)_{\xi_1},
\end{equation}
then, substituting $\eqref{2.6}_1$, $\eqref{2.6}_2$ and $\eqref{2.7}_2$ into $\eqref{2.9}$, we derive that
\begin{equation}\label{2.10}
v^s_{\xi_1}=\frac{h(v^s)}{\sigma_\ast(\frac{4}{3}\mu+\lambda)+\tau\sigma_\ast h^{\prime}(v^s)},
\end{equation}
where $h(v^s)=\sigma_\ast^2(v_--v^s)+(p(v_-)-p(v^s))$.

The equations for the planar viscous shock \eqref{2.6} are similar to those for a 1-D traveling wave. The difference is that the planar shock has two equation ($\Pi_1$ and $\Pi_2$) instead of one. Despite this difference, solving the respective equations yields results analogous to the 1-D case. Then, we have the following lemma, see \cite{GHSR, GHSS}.

\begin{lemma}\label{pvsw}
Under the assumption \eqref{sccs} for $\tau$, given any states $(v_-,u_{1-})$, $(v_+,u_{1+})$, and $\sigma_\ast>0$ satisfying R-H condition \eqref{1.11rh} and Lax  entropy condition \eqref{1.12}, there exists a positive constant $C$ independent of $\tau$ such that the following is true. The planar 2-viscous shock wave solutions $(v^s,u_{1}^s,\Pi_{11}^s,\Pi_2^s)(\xi_1)$ connecting $(v_-, u_{1-}, O_{3\times3}, 0)$ and $(v_+, u_{1+}, O_{3\times3}, 0)$ exists uniquely and satisfy
\begin{align*}
v^s_{\xi_1}>0,\quad
(u_1^s)_{\xi_1}=-\sigma_\ast v^s_{\xi_1}<0,
\end{align*}
and
\begin{align*}
&|v^s(\xi_1)-v_-|\leq C\delta e^{-C\delta|\xi_1|},\quad
|u_1^s(\xi_1)-u_{1-}|\leq C\delta e^{-C\delta|\xi_1|},\qquad  \forall \xi_1<0,\\
&|v^s(\xi_1)-v_+|\leq C\delta e^{-C\delta|\xi_1|},\quad
|u_1^s(\xi_1)-u_{1+}|\leq C\delta e^{-C\delta|\xi_1|},\qquad  \forall \xi_1>0,\\
&|(v^s_{\xi_1},u^s_{1\xi_1})|\le C\delta^2e^{-C\delta|\xi_1|},\quad
|\partial_{\xi_1}^{k}(v^s, u_1^s)| \le C\delta v^s_{\xi_1},\quad\quad\qquad\quad
\forall \xi_1\in\mathbb{R},\\
&|(\Pi^s_{11}, \Pi^s_2)|\le Cv^s_{\xi_1},\qquad\qquad
|\partial^{k-1}_{\xi_1}(\Pi^s_{11}, \Pi^s_2)|\le C\delta v^s_{\xi_1},
\qquad\quad\quad\forall\xi_1\in\mathbb{R},
\end{align*}
for $k=2,3,4,...$, where $\delta$ denote the strength of the shock wave as $\delta:=|p(v_-)-p(v_+)|\sim|v_--v_+|\sim|u_{1-}-u_{1+}|$.
\end{lemma}

\subsection{Construction of shift}

We define the shift function $X(t)$ as the solution to the following ODE system
\begin{equation}\label{shiftx}
\begin{cases}
\begin{aligned}
&\dot{X}(t)=-\frac{M}{\delta}
\left[\int_{\mathbb{T}^2}\int_{\mathbb{R}}\frac{a^{-X}(\xi_1)}{\sigma_\ast}
\rho(u^s_1)_{\xi_1}^{-X}(p(v)-p((v^s)^{-X})\dif \xi_1\dif\xi^{\prime}\right.\\
&\qquad\qquad\qquad\left.
-\int_{\mathbb{T}^2}\int_{\mathbb{R}}a^{-X}(\xi_1)\rho p((v^s)^{-X})_{\xi_1}\left(v-(v^s)^{-X}\right)\dif \xi_1\dif\xi^{\prime}\right],\\
&X(0)=0,
\end{aligned}
\end{cases}
\end{equation}
where $M$ is the specific constant chosen as $M:=\frac{9(\gamma+1)\sigma_-^3v_-^2}{16\gamma p(v_-)}$ and $f^{-X}(\xi_1)$ denotes the shifted function $f^{-X}(\xi_1):=f(\xi_1-X(t))$. The shifted weight function $a$ is defined by
\begin{align}\label{weighta}
a(t,x_1):=a(x_1-\sigma_1t-X_1(t)),
\end{align}
where the base weight function has the form
\[
a(\xi_1)=1+\frac{\nu(p(v_-)-p(v^s))}{\delta},
\]
with the parameter $\nu$ satisfying
\[
\delta\ll\nu\le C\sqrt{\delta}.
\]
On the other hand, we known that $1<a<1+C\nu$ and
\begin{equation}\label{weight1}
a^{\prime}(\xi_1)=-\frac{\nu}{\delta}p(v^s)_{\xi_1},\qquad
|a^{\prime}|\le C\frac{\nu}{\delta}|v^s_{\xi_1}|.
\end{equation}

The definitions of the shift function $X(t)$ and the weight function $a(\xi_1)$ are analogous to those in \cite{GHSR, GHSS}, with the only distinction being the coordinate system: there are defined herein in Eulerian coordinates, as opposed to the Lagrangian coordinates used in the 1-D case. Then, we have the following lemma.
\begin{lemma}\label{le2.2}
For any $c_1, c_2>0$, there exists a constant $C>0$ such that for any $T>0$ and any function $v\in L^{\infty}((0,T )\times \Omega)$ satisfying
\begin{equation}\label{2.18}
c_1<v(t,x)<c_2,\quad \forall(t,x)\in[0, T]\times\Omega,
\end{equation}
the ODE \eqref{shiftx} has a unique absolutely continuous solution $X(t)$ on $[0, T]$. Moreover,
\begin{equation}\label{2.19}
|X(t)|\le C t, \qquad \forall \, 0\le t \le T.
\end{equation}
\end{lemma}

\section{Local solution and a priori estimates}

Firstly, by changing of variable $(t, x)\rightarrow(t,\xi)$ we rewrite the system \eqref{1.6} into following system

\begin{equation}\label{3.1}
\begin{cases}
\rho\left(\partial_tv-\sigma\partial_{\xi_1}v+u\cdot\nabla_{\xi} v\right)=\ddiv_{\xi} u,\\
\rho\left(\partial_tu-\sigma\partial_{\xi_1}u+u\cdot\nabla_{\xi} u\right)+\nabla_{\xi}p(v)=\ddiv_{\xi}\Pi_1+\nabla_{\xi}\Pi_2,\\
\tau\rho\left(\partial_t\Pi_1-\sigma\partial_{\xi_1}\Pi_1+u\cdot\nabla_{\xi} \Pi_1\right)+\Pi_1
   =\mu\left(\nabla_{\xi} u+(\nabla_{\xi} u)^T-\frac{2}{3}\ddiv_{\xi} u \mathrm I_3\right),\\
\tau\rho\left(\partial_t\Pi_2-\sigma\partial_{\xi_1}\Pi_2+u\cdot\nabla_{\xi}\Pi_2\right)+\Pi_2
=\lambda\ddiv_{\xi} u.
\end{cases}
\end{equation}

Then,  following a method similar to that in \cite{hufes2017}, we establish the following local existence theory. For related results and alternative approaches, we refer the reader to \cite{JR, WY, MV20, TA}.

\begin{theorem}\label{th3.1}
Let $\underline{v}$ and $\underline{u}_1$ be smooth monotone functions satisfying
\[
\underline{v}(x_1)=v_{\pm} \quad \underline{u}_1(x_1)=u_{1\pm} \quad for \quad \pm x_1\geq1.
\]
For any constants $M_0,M_1,\kappa_1,\kappa_2,\kappa_3,\kappa_4$ with $M_1>M_0>0$ and $\kappa_1>\kappa_2>\kappa_3>\kappa_4>0$, there exists a constant $T_0>0$ such that if
\[
\begin{aligned}
&\|v_0-\underline{v}\|_{H^3}+\|u_0-\underline{u}\|_{H^3}
+\sqrt{\tau}\|(\Pi_{10}, \Pi_{20})\|_{H^3}\leq M_0,\\
&0<\kappa_3\leq v_0(x)\leq\kappa_2,\qquad \forall x\in\Omega,
\end{aligned}
\]
where $\underline{u}=(\underline{u}_1, 0, 0)^T$, then the system \eqref{3.1} with initial condition \eqref{1.13} admits a unique solution $(v, u, \Pi_1, \Pi_2)$ on $[0,T_0]$ with
\[
\begin{aligned}
(v-\underline{v}, u-\underline{u}, \Pi_1, \Pi_2)\in C([0,T_0];H^3),
\end{aligned}
\]
and
\[
\|v-\underline{v}\|_{L^{\infty}(0,T_0;H^3)}+\|u-\underline{u}\|_{L^{\infty}(0,T_0;H^3)}
+\sqrt{\tau}\|(\Pi_1, \Pi_2)\|_{L^{\infty}(0,T_0;H^2)}\leq M_1.
\]
Moreover,
\[
\kappa_4\leq v(t,x)\leq\kappa_1,\qquad \forall(t,x)\in[0,T_0]\times\Omega.
\]
\end{theorem}

To obtain the global solution, we establish the following a priori estimates.

\begin{prop}\label{p1}
Let $(v, u, \Pi_1, \Pi_2)$ be local solutions given by Theorem \ref{th3.1} on $[0,T]$ for some $T>0$ and $( v^s,  u^s, \Pi^s_1, \Pi^s_{2})$ be defined in \eqref{2.3}. Then there exist positive constants $\delta_0,\varepsilon_1$ such that if the shock waves strength satisfy $\delta<\delta_0$ and
\begin{equation}\label{3.4}
\sup_{0\le t\le T} \|(v-(v^s)^{-X}, u-(u^s)^{-X}, \sqrt{\tau}(\Pi_1-(\Pi^s_1)^{-X}), \sqrt{\tau}(\Pi_2-(\Pi^s_2)^{-X}))\|_{H^3}\leq\varepsilon_1,
\end{equation}
then the following estimates hold
\begin{equation}\label{3.5}
\begin{aligned}
&\sup\limits_{t\in[0, T]}\|(v-(v^s)^{-X}, u-(u^s)^{-X}, \sqrt{\tau}(\Pi_1-(\Pi^s_1)^{-X}), \sqrt{\tau}(\Pi_2-(\Pi^s_2)^{-X}))\|_{H^3}\\
&\quad+\int_0^{T}\|\sqrt{|(v^s)^{-X}_{\xi_1}|}|v-(v^s)^{-X}|\|_{L^2}^2\dif t
+\int_0^{T}\|\left(\nabla_\xi(v-(v^s)^{-X}),\nabla_\xi (u-(u^s)^{-X})\right)\|^2_{H^2}\dif t
\\
&\quad+\delta\int_0^{T}|\dot X(t)|^2\dif t
+\int_0^{T}\|(\Pi_1-(\Pi_1^s)^{-X}, \Pi_2-(\Pi_2^s)^{-X})\|^2_{H^3}\dif t\\
&\leq C_0\left(\|(v_0-v^s, u_0-u^s, \sqrt{\tau}(\Pi_{10}-\Pi_1^s), \sqrt{\tau}(\Pi_{20}-\Pi_2^s))\|^2_{H^3}\right),
\end{aligned}
\end{equation}
where $C_0$ independent of $T$ and $\tau$.

In addition, 
\begin{equation}\label{3.6}
|\dot{X}(t)|\leq C_0\|(v-(v^s)^{-X})(t,\cdot)\|_{L^{\infty}},\qquad\forall t\leq T.
\end{equation}
\end{prop}

We will give the proof of Proposition \ref{p1} in Section 4.

Now, by using Theorem \ref{th3.1} and Proposition \ref{p1}, we are able to prove Theorem \ref{th1.1}.

{\bf{Proof of Theorem \ref{th1.1}}}:  Firstly, by the construction of $\underline v, \underline {u}_1$, there exists $C_\ast>0$ such that
\begin{align}\label{3.7}
\sum\limits_{\pm}\left(\|(\underline{v}-v_{\pm},\underline{u}_1-u_{1\pm})\|
_{H^3(\mathbb{R}_{\pm}\times\mathbb{T}^2)}\right)
<C_\ast \delta.
\end{align}
Then, using Lemma \ref{pvsw} and \eqref{3.7}, we can derive that for some $C_{\ast\ast}>0$,
\begin{equation}\label{3.8}
\begin{aligned}
&\|(\underline{v}-v^s,\underline{u}-u^s)\|_{H^3}
+\sqrt{\tau}\|(\Pi_1^s, \Pi_2^s)\|_{H^3}\\
&\le
\sum\limits_{\pm}\left(\|(\underline{v}-v_{\pm},\underline{u}-u_{\pm})\|
_{H^3(\mathbb{R}_{\pm}\times\mathbb{T}^2)}\right)
+\sum\limits_{\pm}\left(\|( v^s_0-v_{\pm}, u^s_0-u_{\pm}, \sqrt{\tau}\Pi^s_{10}, \sqrt{\tau}\Pi^s_{20})\|_{H^3(\mathbb{R}_{\pm}\times\mathbb{T}^2)}\right)
\\
&<C_{\ast\ast}\sqrt{\delta}.
\end{aligned}
\end{equation}
Noting that $\delta\in(0, \delta_0)$ and using the smallness of $\delta$, we have
\begin{equation}\label{3.9}
\frac{\varepsilon_1}{2(C_0+1)}-C_{\ast\ast}\sqrt{\delta}
-C_{\ast}\delta>0.
\end{equation}
Then, based on the above positive constants, we define $\varepsilon_0$ as follows
\[
\varepsilon_0:=\varepsilon_\ast-C_\ast\delta,\qquad
\varepsilon_\ast:=\frac{\varepsilon_1}{2(C_0+1)}-C_{\ast\ast}\sqrt{\delta},
\]
where we can choose $\varepsilon_0$ independent of $\delta_0$. For example, by choosing $\delta$ sufficiently small, we can take $\varepsilon_0=\frac{\varepsilon_1}{4 (C_0+1)}$.

Now, by using the assumption \eqref{1.10} in Theorem \ref{th1.1} and \eqref{3.7}, we can get
\begin{equation}\label{3.10}
\begin{aligned}
&\|(v_0-\underline{v},u_0-\underline{u})\|_{H^3}
+\sqrt{\tau}\|(\Pi_{10}, \Pi_{20})\|_{H^3}\\
&\le\sum\limits_{\pm}\left(\|(v_0-v_{\pm}, u_0-u_{\pm}, \sqrt{\tau}\Pi_{10}, \sqrt{\tau}\Pi_{10})\|
_{H^3(\mathbb{R}_{\pm}\times\mathbb{T}^2)}\right)
+\sum\limits_{\pm}\left(\|(\underline{v}-v_{\pm},\underline{u}-u_{\pm})\|
_{H^3(\mathbb{R}_{\pm}\times\mathbb{T}^2)}\right)
\\
&<\varepsilon_0+C_{\ast}\delta=\varepsilon_\ast.
\end{aligned}
\end{equation}
Especially, by applying Sobolev's embedding theorem, we obtain
\[
\|v_0-\underline{v}\|_{L^\infty}\le C\varepsilon_\ast,
\]
and subsequently, using the smallness of $\varepsilon_\ast$, we can derive that
\begin{equation}\label{3.11}
\frac{v_-}{2}<v_0(x)<2v_+,\qquad \forall x\in\Omega.
\end{equation}
On the other hand, we have
\begin{equation}\label{3.12}
0<\varepsilon_\ast<\frac{\varepsilon_1}{2}.
\end{equation}
Therefore, according Theorem \ref{th3.1} with \eqref{3.10}, \eqref{3.11} and \eqref{3.12}, we obtain that the system \eqref{1.6} with initial condition \eqref{1.13} has a unique solution $(v, u, \Pi_1, \Pi_2)$ on $[0,T_0]$ such that
\begin{equation}\nonumber
\begin{aligned}
(v-\underline{v}, u-\underline{u}, \Pi_1, \Pi_2)\in C([0,T_0];H^3),
\end{aligned}
\end{equation}
\begin{equation}\label{3.14}
\|v-\underline{v}\|_{L^{\infty}(0,T_0;H^3)}+\|u-\underline{u}\|_{L^{\infty}(0,T_0;H^3)}
+\sqrt{\tau}\|(\Pi_1, \Pi_2)\|_{L^{\infty}(0,T_0;H^3)}\leq \frac{\varepsilon_1}{2},
\end{equation}
and
\begin{equation}\nonumber
\frac{v_-}{3}\leq v(t,x)\leq 3v_+,\qquad \forall(t,x)\in[0,T_0]\times\Omega.
\end{equation}
In addition, using Lemma \ref{pvsw}, we have
\begin{align*}
&\int_0^\infty |v^s(x_1-\sigma t-X(t))-v_+|^2\dif x_1\\
&\le \int_0^\infty |v^s(x_1-\sigma t)-v_+|^2\dif x_1
+\int_{-X(t)}^0 |v^s(x_1-\sigma t)-v_+|^2\dif x_1\\
&\le C\delta(1+|X(t)|).
\end{align*}
Similarly, we can get
$$
\|(v^s)^{-X}-v_-\|_{L^2(\mathbb{R}_-)}^2\le C\delta(1+|X(t)|),
$$
and
$$
 \|(u^s)^{-X}-u_-\|_{L^2(\mathbb{R}_-)}^2\le C\delta(1+|X(t)|), \, \|(u^s)^{-X}-u_m\|_{L^2(\mathbb{R}_+)}^2\le C\delta(1+|X(t)|).
$$
Hence, combining the above estimates and using \eqref{3.7}, one get
\begin{align*}
&\|(\underline{v}-(v^s)^{-X},\underline{u}-(u^s)^{-X})\|_{H^3}
+\sqrt{\tau}\|((\Pi_1^s)^{-X}, (\Pi_2^s)^{-X})\|_{H^3}\\
&\le
\sum\limits_{\pm}\left(\|((v^s)^{-X}-v_{\pm},(u^s)^{-X}-u_{\pm},
\sqrt{\tau}(\Pi_1^s)^{-X}, \sqrt{\tau}(\Pi_2^s)^{-X})\|_{H^3(\mathbb{R}_{\pm}\times\mathbb{T}^2)}\right)
\\
&\quad+\sum\limits_{\pm}\left(\|(\underline{v}-v_{\pm},\underline{u}-u_{\pm})\|
_{H^3(\mathbb{R}_{\pm}\times\mathbb{T}^2)}\right)\\
&<C\sqrt{\delta}(1+\sqrt{|X(t)|}).
\end{align*}
Then, using Lemma \ref{le2.2}, we can get
\begin{equation}\nonumber
\|(\underline{v}-(v^s)^{-X},\underline{u}-(u^s)^{-X})\|_{H^3}
+\sqrt{\tau}\|((\Pi_1^s)^{-X}, (\Pi_2^s)^{-X})\|_{H^3}
<C\sqrt{\delta_0}(1+\sqrt{t}).
\end{equation}
Next, choosing $0<T_1<T_0$ small enough such that $C\sqrt{\delta_0}(1+\sqrt{t})<\frac{\varepsilon_1}{2}$, we have
\begin{equation}\label{3.17}
\|(\underline{v}-(v^s)^{-X},\underline{u}-(u^s)^{-X})\|_{L^\infty(0, T_1; H^3)}
+\sqrt{\tau}\|((\Pi_1^s)^{-X}, (\Pi_2^s)^{-X})\|_{L^\infty(0, T_1; H^3)}<\frac{\varepsilon_1}{2}.
\end{equation}
Combining \eqref{3.14} and \eqref{3.17}, we obtain that
\begin{equation}\label{3.18}
\|(v-(v^s)^{-X},u-(u^s)^{-X})\|_{L^\infty(0, T_1; H^3)}
+\sqrt{\tau}\|((\Pi_1^s)^{-X}, (\Pi_2^s)^{-X})\|_{L^\infty(0, T_1; H^3)}
<\varepsilon_1.
\end{equation}
Additionally, since the shifts $X(t)$ is Lipschitz continuous and using \eqref{3.18}, we can get
\begin{equation}\nonumber
\begin{aligned}
(v-(v^s)^{-X}, u- (u^s)^{-X}, \Pi_1- (\Pi_1^s)^{-X}, \Pi_2- (\Pi_2^s)^{-X})\in C([0,T_1];H^3),
\end{aligned}
\end{equation}
Now, we consider the maximal existence time defined as follows
\[
T_m:=\sup\left\{t>0\Big|\sup\limits_{[0, t]}
 \|(v-(v^s)^{-X}, u-(u^s)^{-X}, \sqrt{\tau}(\Pi_1-(\Pi^s_1)^{-X}), \sqrt{\tau}(\Pi_2-(\Pi^s_2)^{-X}))\|_{H^3}\le\varepsilon_1\right\}.
\]
If $T_m<\infty$, according the continuation argument, we can obtain that
\begin{equation}\label{3.20}
\sup\limits_{[0, T_m]} \|(v-(v^s)^{-X}, u-(u^s)^{-X}, \sqrt{\tau}(\Pi_1-(\Pi^s_1)^{-X}), \sqrt{\tau}(\Pi_2-(\Pi^s_2)^{-X}))\|_{H^3}=\varepsilon_1.
\end{equation}
However, by \eqref{3.8}, \eqref{3.9} and \eqref{3.10}, we can get
\begin{equation}\nonumber
\|(v_0-v^s_0,u_0-u^s_0)\|_{H^3}
+\sqrt\tau\|(\Pi_{10}-\Pi_{10}^s, \Pi_{20}-\Pi_{20}^s)\|_{H^3}
\le \frac{\varepsilon_1}{2(C_0+1)}.
\end{equation}
Next, based on the Proposition \ref{p1} and smallness of $\delta_0$, it can be deduced that
\begin{equation}\nonumber
\begin{aligned}
&\sup\limits_{[0, T_m]} \|(v-(v^s)^{-X}, u-(u^s)^{-X}, \sqrt{\tau}(\Pi_1-(\Pi^s_1)^{-X}), \sqrt{\tau}(\Pi_2-(\Pi^s_2)^{-X}))\|_{H^3}\\
&\le C_0\frac{\varepsilon_1}{2(C_0+1)}
\le \frac{\varepsilon_1}{2},
\end{aligned}
\end{equation}
which contradicts \eqref{3.20}.

Hence, we conclude that $T_m=\infty$ and together with Proposition \ref{p1}, we can derive that
\begin{equation}\label{3.23}
\begin{aligned}
&\sup\limits_{t\in[0, \infty)}\|(v-(v^s)^{-X}, u-(u^s)^{-X}, \sqrt{\tau}(\Pi_1-(\Pi^s_1)^{-X}), \sqrt{\tau}(\Pi_2-(\Pi^s_2)^{-X}))\|_{H^3}\\
&\quad+\int_0^{\infty}\|\sqrt{|(v^s)^{-X}_{\xi_1}|}|v-(v^s)^{-X}|\|_{L^2}^2\dif t
+\int_0^{\infty}\|(\Pi_1-(\Pi_1^s)^{-X}, \Pi_2-(\Pi_2^s)^{-X})\|^2_{H^3}\dif t
\\
&\quad+\delta\int_0^{\infty}|\dot X(t)|^2\dif t
+\int_0^{\infty}\|\left(\nabla_\xi(v-(v^s)^{-X}),\nabla_\xi (u-(u^s)^{-X})\right)\|^2_{H^2}\dif t
\\
&\leq C_0\left(\|(v_0-v^s, u_0-u^s,
\sqrt{\tau}(\Pi_{10}-\Pi_{1}^s), \sqrt{\tau}(\Pi_{20}-\Pi_{2}^s))\|^2_{H^3}\right).
\end{aligned}
\end{equation}

In addition, using system \eqref{sys1} below and \eqref{3.23}, we can get
$$
\int_0^{\infty}\left\|\nabla_{x}(v-(v^s)^{-X}, u-(u^s)^{-X}, \Pi_1-(\Pi_1^s)^{-X}, \Pi_2-(\Pi_2^s)^{-X})\right\|_{L^2}^2\dif t \le C
$$
and
$$\int_0^{\infty} \Big|\frac{\dif}{\dif t}\left\|\nabla_{x}(v-(v^s)^{-X}, u-(u^s)^{-X}, \sqrt\tau(\Pi_1-(\Pi_1^s)^{-X}), \sqrt\tau(\Pi_2-(\Pi_2^s)^{-X}))\right\|_{L^2}^2 \Big |\dif t \le C.$$
Thus, combining the above estimates and using \eqref{sccs}, we get
\begin{equation}\label{3.24}
\lim\limits_{t\rightarrow\infty}\left\|\nabla_{x}(v-(v^s)^{-X}, u-(u^s)^{-X}, \sqrt\tau(\Pi_1-(\Pi_1^s)^{-X}), \sqrt\tau(\Pi_2-(\Pi_2^s)^{-X}))\right\|_{L^2}^2=0.
\end{equation}
On the other hand, using \eqref{3.23} again, we obtain that
\begin{equation}\label{3.25}
\left\|(v-(v^s)^{-X}, u-(u^s)^{-X}, \Pi_1-(\Pi_1^s)^{-X}, \Pi_2-(\Pi_2^s)^{-X})\right\|_{L^2}^2\le C,
\end{equation}
and
\begin{equation}\label{3.26}
\left\|\nabla_{xx}(v-(v^s)^{-X}, u-(u^s)^{-X}, \Pi_1-(\Pi_1^s)^{-X}, \Pi_2-(\Pi_2^s)^{-X})\right\|_{L^2}^2\le C.
\end{equation}
Therefore, combining \eqref{3.24}-\eqref{3.26} and applying \eqref{sccs} and the Gagliardo-Nirenberg interpolation inequality defined in Lemma \ref{czbds1}, we derive that
\begin{align*}
\lim\limits_{t\rightarrow\infty}\left\|(v-(v^s)^{-X}, u-(u^s)^{-X}, \sqrt\tau(\Pi_1-(\Pi_1^s)^{-X}), \sqrt\tau(\Pi_2-(\Pi_2^s)^{-X}))\right\|_{L^\infty}=0.
\end{align*}
Furthermore, using the definition of shift $X(t)$ (see \eqref{shiftx}), we have
\begin{equation}\nonumber
|\dot{X}(t)|\leq C\left\|v-(v^s)^{-X}\right\|_{L^{\infty}}\rightarrow0\quad as \quad t\rightarrow\infty.
\end{equation}

Thus, the proof of Theorem \ref{th1.1} is completed.


\section{Proof of proposition \ref{p1}}
In this section, we establish the a priori estimates  of local solutions and thus give a proof of Proposition \ref{p1}.

First of all, combining systems \eqref{3.1} and \eqref{2.3}, we derive the perturbation equations for error term $\left(v-(v^s)^{-X}, u-(u^s)^{-X},\Pi_1-(\Pi_1^s)^{-X}, \Pi_2-(\Pi_2^s)^{-X}\right)$ as follows
\begin{equation}\label{sys1}
\begin{cases}
\begin{aligned}
&\rho\partial_t\left(v-(v^s)^{-X}\right)-\sigma\rho\partial_{\xi_1}\left(v-(v^s)^{-X}\right)
        +\rho u\cdot\nabla_{\xi}\left(v-(v^s)^{-X}\right)-\rho\dot{X}(t)\partial_{\xi_1}(v^s)^{-X}
        +F\partial_{\xi_1}(v^s)^{-X}\\
 &\qquad=\ddiv_{\xi}\left(u-(u^s)^{-X}\right),\\
&\rho\partial_t\left(u-(u^s)^{-X}\right)-\sigma\rho\partial_{\xi_1}\left(u-(u^s)^{-X}\right)
        +\rho u\cdot\nabla_{\xi}\left(u-(u^s)^{-X}\right)-\rho\dot{X}(t)\partial_{\xi_1}(u^s)^{-X} +F\partial_{\xi_1}(u^s)^{-X}\\
&\qquad  +\nabla_{\xi}\left(p(v)-p((v^s)^{-X})\right)
               =\ddiv_{\xi}\left(\Pi_1-(\Pi^s)_1^{-X}\right)+\nabla_{\xi}\left(\Pi_2-(\Pi_2^s)^{-X}\right),\\
&\tau\rho\partial_t\left(\Pi_1-(\Pi_1^s)^{-X}\right)-\tau\sigma\rho\partial_{\xi_1}\left(\Pi_1-(\Pi_1^s)^{-X}\right)
+\tau\rho u\cdot\nabla_{\xi} \left(\Pi_1-(\Pi_1^s)^{-X}\right)-\tau\rho\dot{X}(t)\partial_{\xi_1}(\Pi_1^s)^{-X}\\
&\qquad +\tau F\partial_{\xi_1}(\Pi_1^s)^{-X}+\Pi_1-(\Pi_1^s)^{-X}\\
 &  =\mu\left(\nabla_{\xi}\left(u-(u^s)^{-X}\right)+\left(\nabla_{\xi}\left(u-(u^s)^{-X}\right)\right)^T
   -\frac{2}{3}\ddiv_{\xi}\left(u-(u^s)^{-X}\right) \mathrm{I}_3\right),\\
&\tau\rho\partial_t\left(\Pi_2-(\Pi_2^s)^{-X}\right)-\tau\sigma\rho\partial_{\xi_1}\left(\Pi_2-(\Pi_2^s)^{-X}\right)
+\tau\rho u\cdot\nabla_{\xi} \left(\Pi_2-(\Pi_2^s)^{-X}\right)-\tau\rho\dot{X}(t)\partial_{\xi_1}(\Pi_2^s)^{-X}\\
&+\tau F\partial_{\xi_1}(\Pi_2^s)^{-X}+\Pi_2-(\Pi_2^s)^{-X}
=\lambda\ddiv_{\xi} \left(u-(u^s)^{-X}\right),
\end{aligned}
\end{cases}
\end{equation}
where
\begin{equation}\label{wcxf}
\begin{aligned}
F&=-\sigma\left(\rho-(\rho^s)^{-X}\right)+\rho u_1-(\rho^su_1^s)^{-X}\\
 &=-\frac{\sigma_\ast}{(\rho^s)^{-X}}\left(\rho-(\rho^s)^{-X}\right)+\rho\left(u_1-(u_1^s)^{-X}\right)\\
 &=\frac{\sigma_\ast}{v}\left(v-(v^s)^{-X}\right)+\frac{1}{v}\left(u_1-(u_1^s)^{-X}\right),
\end{aligned}
\end{equation}
by using \eqref{2.5}.

To establish the a priori estimates for Proposition \ref{p1}, we proceed by proving the lower-order, higher-order and dissipation estimates in sequence. Prior to demonstrating the lower-order estimates, it is necessary to develop several preparatory lemmas. Among these, we first state a 3-D Poincar\'e type inequality, the proof if which is provided in \cite{WWS}.

\begin{lemma}\label{poinbds}
For any $f$: $[0, 1]\times\mathbb{T}^2\rightarrow\mathbb{R}$ satisfying
\[
\int_{\mathbb{T}^2}\int_0^1\left(y_1(1-y_1)|\partial_{y_1}f|^2
  +\frac{|\nabla_{y^{\prime}}f|^2}{y_1(1-y_1)}\right)\dif y_1\dif y^{\prime}<\infty,
\]
it holds
\begin{equation}\label{czbds}
\int_{\mathbb{T}^2}\int_0^1|f-\bar{f}|^2\dif y_1\dif y^{\prime}\le
\frac{1}{2}\int_{\mathbb{T}^2}\int_0^1y_1(1-y_1)|\partial_{y_1}f|^2\dif y_1\dif y^{\prime}
+\frac{1}{16\pi^2}\int_{\mathbb{T}^2}\int_0^1\frac{|\nabla_{y^{\prime}}f|^2}{y_1(1-y_1)}\dif y_1\dif y^{\prime},
\end{equation}
where $\bar{f}=\int_{\mathbb{T}^2}\int_0^1f\dif y_1\dif y^{\prime}$ and $y^{\prime}=(y_2, y_3)$.
\end{lemma}

Next, we state a 3-D Gagliardo-Nirenberg interpolation inequality. The proof can be found in \cite{LWWR, WW1}.
\begin{lemma}\label{czbds1}
It holds for $g(x)\in H^2(\Omega)$ with $x=(x_1, x_2, x_3)\in\Omega$ that
\begin{equation}
\|g\|_{L^{\infty}}\le \sqrt{2}\|g\|_{L^{2}}^{\frac{1}{2}}\|\partial_{x_1}g\|_{L^{2}}^{\frac{1}{2}}
+C\|\nabla_xg\|_{L^{2}}^{\frac{1}{2}}\|\nabla_x^2g\|_{L^{2}}^{\frac{1}{2}},
\end{equation}
where $C>0$ is a positive constant.
\end{lemma}

\subsection{Estimates for weighted relative entropy}

Now, let relative entropy
$$\eta(t, \xi)=H(v|(v^s)^{-X})+\frac{|u-(u^s)^{-X}|^2}{2}
+\frac{\tau|\Pi_1-(\Pi_1^s)^{-X}|^2}{4\mu}+\frac{\tau|\Pi_2-(\Pi_2^s)^{-X}|^2}{2\lambda},$$ where $H(v|(v^s)^{-X})=H(v)-H(v^s)-H^{\prime}(v^s)(v-v^s)$ and $H(v)=\frac{v^{1-\gamma}}{\gamma-1}$.  In the following lemma, we estimates the shifted relative entropy $\eta(t, \xi)$ weighted by $a(\xi_1)$.

\begin{lemma}\label{le4.5}
We have
\begin{equation}\label{4.5}
\frac{\dif}{\dif t}\int_{\mathbb{T}^2}\int_{\mathbb{R}}a^{-X}\rho\eta(t,\xi)\dif \xi_1\dif \xi^{\prime}=\dot{X}(t)Y(t)+J^{bad}(t)-J^{good}(t),
\end{equation}
where
\begin{align*}
Y(t):&=-\int_{\mathbb{T}^2}\int_{\mathbb{R}}a^{-X}_{\xi_1}\rho\eta(t,\xi)\dif \xi_1\dif \xi^{\prime}
-\int_{\mathbb{T}^2}\int_{\mathbb{R}}a^{-X}\rho p^{\prime}((v^s)^{-X})(v^s)^{-X}_{\xi_1}(v-(v^s)^{-X})\dif \xi_1\dif \xi^{\prime}\\
&+\int_{\mathbb{T}^2}\int_{\mathbb{R}}a^{-X}\rho \left(u-(u^s)^{-X}\right)(u^s)^{-X}_{\xi_1}\dif \xi_1\dif \xi^{\prime}\\
 &+\int_{\mathbb{T}^2}\int_{\mathbb{R}}a^{-X} \frac{\tau\rho}{2\mu}\left(\Pi_1-(\Pi_1^s)^{-X}\right):(\Pi_1^s)^{-X}_{\xi_1}\dif \xi_1\dif \xi^{\prime}\\
&+\int_{\mathbb{T}^2}\int_{\mathbb{R}}a^{-X} \frac{\tau\rho}{\lambda}\left(\Pi_2-(\Pi_2^s)^{-X}\right)(\Pi_2^s)^{-X}_{\xi_1}\dif \xi_1\dif \xi^{\prime},
\end{align*}
\begin{align*}
J^{bad}(t):=&\frac{1}{2\sigma_\ast}\int_{\mathbb{T}^2}\int_{\mathbb{R}}a^{-X}_{\xi_1}
|p(v)-p((v^s)^{-X})|^2\dif \xi_1\dif \xi^{\prime}
+\sigma_\ast\int_{\mathbb{T}^2}\int_{\mathbb{R}}a^{-X}p(v|(v^s)^{-X})(v^s)^{-X}_{\xi_1}\dif \xi_1\dif \xi^{\prime}\\
&+\frac{\delta}{\nu}\int_{\mathbb{T}^2}\int_{\mathbb{R}}\frac{a^{-X}}{\sigma_\ast v}a^{-X}_{\xi_1}\left(u_1-(u_1^s)^{-X}\right)^2\dif \xi_1\dif \xi^{\prime}\\
&-\int_{\mathbb{T}^2}\int_{\mathbb{R}}a^{-X}_{\xi_1}\left(u_1-(u_1^s)^{-X}\right)\left(\Pi_2-(\Pi_2^s)^{-X}\right)\dif \xi_1\dif \xi^{\prime}\\
&-\int_{\mathbb{T}^2}\int_{\mathbb{R}}a^{-X}_{\xi_1}\left(\left(u_1-(u^s_1)^{-X}\right)
\left(\Pi_{11}-(\Pi_{11}^s)^{-X}\right)+u_2\Pi_{21}+u_3\Pi_{31}
\right)\dif \xi_1\dif \xi^{\prime}\\
&+\int_{\mathbb{T}^2}\int_{\mathbb{R}}
a^{-X}\frac{1}{v}\left((\Pi_{11}^s)_{\xi_1}+(\Pi_{2}^s)_{\xi_1}\right)(v-(v^s)^{-X})\left(u_1-(u_1^s)^{-X}\right)
\dif \xi_1\dif \xi^{\prime}\\
&+\int_{\mathbb{T}^2}\int_{\mathbb{R}}a^{-X}\frac{1}{v\sigma_\ast}
\left((\Pi_{11}^s)_{\xi_1}+(\Pi_{2}^s)_{\xi_1}\right)\left(u_1-(u_1^s)^{-X}\right)^2
\dif \xi_1\dif \xi^{\prime}
\\
&+\int_{\mathbb{T}^2}\int_{\mathbb{R}}a^{-X}_{\xi_1}F\eta(t,\xi)\dif \xi_1\dif \xi^{\prime}
-\int_{\mathbb{T}^2}\int_{\mathbb{R}}a^{-X}\frac{\tau}{2\mu}F\left(\Pi_1-(\Pi_1^s)^{-X}\right) : (\Pi_1^s)^{-X}_{\xi_1}\dif \xi_1\dif \xi^{\prime}\\
&-\int_{\mathbb{T}^2}\int_{\mathbb{R}}a^{-X}\frac{\tau}{\lambda}F\left(\Pi_2-(\Pi_2^s)^{-X}\right)
(\Pi_2^s)^{-X}_{\xi_1}\dif \xi_1\dif \xi^{\prime}
:=\sum\limits_{i=1}^{10}B_i(t),
\end{align*}
and
\begin{align*}
J^{good}(t)&:=\sigma_\ast\int_{\mathbb{T}^2}\int_{\mathbb{R}}a^{-X}_{\xi_1} H(v|(v^s)^{-X})\dif \xi_1\dif \xi^{\prime}
+\sigma_\ast\int_{\mathbb{T}^2}\int_{\mathbb{R}}a^{-X}_{\xi_1} \frac{u_2^2+u_3^2}{2}\dif \xi_1\dif \xi^{\prime}\\
&+\frac{\sigma_\ast}{2}\int_{\mathbb{T}^2}\int_{\mathbb{R}}a^{-X}_{\xi_1}
\Big|u_1-(u_1^s)^{-X}-\frac{p(v)-p((v^s)^{-X})}{\sigma_\ast}\Big|^2\dif \xi_1\dif \xi^{\prime}\\
&-\int_{\mathbb{T}^2}\int_{\mathbb{R}}a^{-X}\frac{\sigma_\ast}{v}p^{\prime}((v^s)^{-X})(v^s)^{-X}_{\xi_1}(v-(v^s)^{-X})^2\dif \xi_1\dif \xi^{\prime}\\
&+\sigma_\ast\int_{\mathbb{T}^2}\int_{\mathbb{R}}a^{-X}_{\xi_1} \frac{\tau|\Pi_1-(\Pi_1^s)^{-X}|^2}{4\mu}\dif \xi_1\dif \xi^{\prime}
+\int_{\mathbb{T}^2}\int_{\mathbb{R}}a^{-X}\frac{|\Pi_1-(\Pi_1^s)^{-X}|^2}{2\mu}\dif \xi_1\dif \xi^{\prime}\\
&+\sigma_\ast\int_{\mathbb{T}^2}\int_{\mathbb{R}}a^{-X}_{\xi_1} \frac{\tau|\Pi_2-(\Pi_2^s)^{-X}|^2}{2\lambda}\dif \xi_1\dif \xi^{\prime}
+\int_{\mathbb{T}^2}\int_{\mathbb{R}}a^{-X}\frac{|\Pi_2-(\Pi_2^s)^{-X}|^2}{\lambda}\dif \xi_1\dif \xi^{\prime}\\
&:=\sum\limits_{i=1}^{8}G_i(t).
\end{align*}
\end{lemma}

\begin{proof}
Firstly, we have
\begin{align*}
&\frac{\dif}{\dif t}\int_{\mathbb{T}^2}\int_{\mathbb{R}}a^{-X}\rho\eta(t,\xi)\dif \xi_1\dif \xi^{\prime}\\
&=
\frac{\dif}{\dif t}\int_{\mathbb{T}^2}\int_{\mathbb{R}}a^{-X}\rho H(v|(v^s)^{-X})\dif \xi_1\dif \xi^{\prime}
+\frac{\dif}{\dif t}\int_{\mathbb{T}^2}\int_{\mathbb{R}}a^{-X}\rho\frac{|u-(u^s)^{-X}|^2}{2}
\dif \xi_1\dif \xi^{\prime}\\
&+\frac{\dif}{\dif t}\int_{\mathbb{T}^2}\int_{\mathbb{R}}a^{-X}\rho\frac{\tau|\Pi_1-(\Pi_1^s)^{-X}|^2}{4\mu}\dif \xi_1\dif \xi^{\prime}
+\frac{\dif}{\dif t}\int_{\mathbb{T}^2}\int_{\mathbb{R}}a^{-X}\rho\frac{\tau|\Pi_2-(\Pi_2^s)^{-X}|^2}{2\lambda}\dif \xi_1\dif \xi^{\prime}:=\sum\limits^4\limits_{i=1}J_i.
\end{align*}
The first term decomposes as
\begin{align*}
J_1
=&-\dot{X}(t)\int_{\mathbb{T}^2}\int_{\mathbb{R}}a^{-X}_{\xi_1}\rho H(v|(v^s)^{-X})\dif \xi_1\dif \xi^{\prime}
+\int_{\mathbb{T}^2}\int_{\mathbb{R}}a^{-X}\rho_t H(v|(v^s)^{-X})\dif \xi_1\dif \xi^{\prime}\\
&+\int_{\mathbb{T}^2}\int_{\mathbb{R}}a^{-X}\rho
\left(-p(v)v_t+p((v^s)^{-X})v_t+p^{\prime}((v^s)^{-X})(v^s)^{-X}_t(v-(v^s)^{-X})\right)\dif \xi_1\dif \xi^{\prime}.
\end{align*}
By using $\eqref{3.1}_1$, we derive
\begin{align*}
&\int_{\mathbb{T}^2}\int_{\mathbb{R}}a^{-X}\rho_t H(v|(v^s)^{-X})\dif \xi_1\dif \xi^{\prime}\\
=&\int_{\mathbb{T}^2}\int_{\mathbb{R}}a^{-X}\left(\sigma\rho_{\xi_1}-\ddiv_\xi(\rho u)\right) H(v|(v^s)^{-X})\dif \xi_1\dif \xi^{\prime}\\
=&-\sigma\int_{\mathbb{T}^2}\int_{\mathbb{R}}a^{-X}_{\xi_1}\rho H(v|(v^s)^{-X})\dif \xi_1\dif \xi^{\prime}
+\int_{\mathbb{T}^2}\int_{\mathbb{R}}a^{-X}_{\xi_1}\rho u_1 H(v|(v^s)^{-X})\dif \xi_1\dif \xi^{\prime}\\
&-\sigma\int_{\mathbb{T}^2}\int_{\mathbb{R}}a^{-X}\rho
\left(-p(v)v_{\xi_1}+p((v^s)^{-X})v_{\xi_1}+p^{\prime}((v^s)^{-X})(v^s)^{-X}_{\xi_1}(v-(v^s)^{-X})\right)\dif \xi_1\dif \xi^{\prime}\\
&+\int_{\mathbb{T}^2}\int_{\mathbb{R}}a^{-X}\rho u\cdot
\left(-p(v)\nabla_\xi v+p((v^s)^{-X})\nabla_\xi v+p^{\prime}((v^s)^{-X})(v-(v^s)^{-X})\nabla_\xi (v^s)^{-X}\right)\dif \xi_1\dif \xi^{\prime},
\end{align*}
and
\begin{align*}
&\int_{\mathbb{T}^2}\int_{\mathbb{R}}a^{-X}\rho v_t\left(-p(v)+p((v^s)^{-X})\right)\dif \xi_1\dif \xi^{\prime}\\
&=\int_{\mathbb{T}^2}\int_{\mathbb{R}}a^{-X}\left(-p(v)+p((v^s)^{-X})\right)\rho\left(\sigma v_{\xi_1}-u\cdot\nabla_\xi v\right)\dif \xi_1\dif \xi^{\prime}\\
&\quad+\int_{\mathbb{T}^2}\int_{\mathbb{R}}a^{-X}\left(-p(v)+p((v^s)^{-X})\right)\ddiv_\xi u\dif \xi_1\dif \xi^{\prime},
\end{align*}
where
\begin{align*}
&\int_{\mathbb{T}^2}\int_{\mathbb{R}}a^{-X}\left(-p(v)+p((v^s)^{-X})\right)\ddiv_\xi u\dif \xi_1\dif \xi^{\prime}\\
&=\int_{\mathbb{T}^2}\int_{\mathbb{R}}a^{-X}\left(-p(v)+p((v^s)^{-X})\right)\ddiv_\xi (u-(u^s)^{-X})\dif \xi_1\dif \xi^{\prime}\\
&\quad+\underbrace{\int_{\mathbb{T}^2}\int_{\mathbb{R}}a^{-X}\left(-p(v)+p((v^s)^{-X})\right)(u_1^s)^{-X}_{\xi_1}\dif \xi_1\dif \xi^{\prime}}_{=:I_1}.
\end{align*}
Then, noting that $p(v|(v^s)^{-X})=p(v)-p((v^s)^{-X})-p^{\prime}((v^s)^{-X})(v-(v^s)^{-X})$, we obtain
\begin{align*}
	I_1&=
	-\sigma_\ast\int_{\mathbb{T}^2}\int_{\mathbb{R}}a^{-X}\left(-p(v)+p((v^s)^{-X})\right)v^s_{\xi_1}\dif \xi_1\dif \xi^{\prime}\\
	&=\sigma_\ast\int_{\mathbb{T}^2}\int_{\mathbb{R}}a^{-X}p(v|(v^s)^{-X})(v^s)^{-X}_{\xi_1}\dif \xi_1\dif \xi^{\prime}
	+\sigma_\ast\int_{\mathbb{T}^2}\int_{\mathbb{R}}a^{-X}p^{\prime}((v^s)^{-X})(v-(v^s)^{-X})(v^s)^{-X}_{\xi_1}\dif \xi_1\dif \xi^{\prime}.
\end{align*}
In addition, using \eqref{2.6}, we get
\begin{align*}
&\int_{\mathbb{T}^2}\int_{\mathbb{R}}a^{-X}\rho p^{\prime}((v^s)^{-X})(v^s)^{-X}_t(v-(v^s)^{-X})\dif \xi_1\dif \xi^{\prime}\\
&=-\dot{X}(t)\int_{\mathbb{T}^2}\int_{\mathbb{R}}a^{-X}\rho p^{\prime}((v^s)^{-X})(v^s)^{-X}_{\xi_1}(v-(v^s)^{-X})\dif \xi_1\dif \xi^{\prime}.
\end{align*}
Thus, we conclude that
\begin{equation}\label{0j1}
\begin{aligned}
J_1
&=-\dot{X}(t)\int_{\mathbb{T}^2}\int_{\mathbb{R}}a^{-X}_{\xi_1}\rho H(v|(v^s)^{-X})\dif \xi_1\dif \xi^{\prime}
-\dot{X}(t)\int_{\mathbb{T}^2}\int_{\mathbb{R}}a^{-X}\rho p^{\prime}((v^s)^{-X})(v^s)^{-X}_{\xi_1}(v-(v^s)^{-X})\dif \xi_1\dif \xi^{\prime}\\
&\quad+\int_{\mathbb{T}^2}\int_{\mathbb{R}}a^{-X}_{\xi_1}(-\sigma_\ast+F) H(v|(v^s)^{-X})\dif \xi_1\dif \xi^{\prime}
+\sigma_\ast\int_{\mathbb{T}^2}\int_{\mathbb{R}}a^{-X}p(v|(v^s)^{-X})(v^s)^{-X}_{\xi_1}\dif \xi_1\dif \xi^{\prime}\\
&\quad+\int_{\mathbb{T}^2}\int_{\mathbb{R}}a^{-X}Fp^{\prime}((v^s)^{-X})(v-(v^s)^{-X})(v^s)^{-X}_{\xi_1}\dif \xi_1\dif \xi^{\prime}\\
&\quad\underbrace{-\int_{\mathbb{T}^2}\int_{\mathbb{R}}a^{-X}\left(p(v)-p((v^s)^{-X})\right)\ddiv_\xi (u-(u^s)^{-X})\dif \xi_1\dif \xi^{\prime}}\limits_{=:I_2},
\end{aligned}
\end{equation}
where
\begin{equation*}
	-\sigma\rho-\rho u_1
	=-\sigma(\rho^s)^{-X}+(\rho^su_1^s)^{-X}-\sigma(\rho-(\rho^s)^{-X})+(\rho u_1-(\rho^su_1^s)^{-X})
	=-\sigma_\ast+F.
\end{equation*}

For term $J_2$, we have
\begin{align*}
J_2
=&-\dot{X}(t)\int_{\mathbb{T}^2}\int_{\mathbb{R}}a^{-X}_{\xi_1}\rho \frac{|u-(u^s)^{-X}|^2}{2}\dif \xi_1\dif \xi^{\prime}
+\int_{\mathbb{T}^2}\int_{\mathbb{R}}a^{-X}\rho_t \frac{|u-(u^s)^{-X}|^2}{2}\dif \xi_1\dif \xi^{\prime}\\
&+\int_{\mathbb{T}^2}\int_{\mathbb{R}}a^{-X}\rho
\left(u-(u^s)^{-X}\right)\cdot\left(u-(u^s)^{-X}\right)_t\dif \xi_1\dif \xi^{\prime}.
\end{align*}
Following a similar procedure to that used for estimating the second term of $J_1$, we obtain
\begin{align*}
&\int_{\mathbb{T}^2}\int_{\mathbb{R}}a^{-X}\rho_t \frac{|u-(u^s)^{-X}|^2}{2}\dif \xi_1\dif \xi^{\prime}\\
&=\int_{\mathbb{T}^2}\int_{\mathbb{R}}a^{-X}_{\xi_1}\left(-\sigma\rho+\rho u_1\right) \frac{|u-(u^s)^{-X}|^2}{2}\dif \xi_1\dif \xi^{\prime}
-\sigma\int_{\mathbb{T}^2}\int_{\mathbb{R}}a^{-X}\rho \left(\frac{|u-(u^s)^{-X}|^2}{2}\right)_{\xi_1}\dif \xi_1\dif \xi^{\prime}\\
&\quad+\int_{\mathbb{T}^2}\int_{\mathbb{R}}a^{-X}\rho u\cdot\nabla_\xi \frac{|u-(u^s)^{-X}|^2}{2}\dif \xi_1\dif \xi^{\prime}.
\end{align*}
Subsequently, by using $\eqref{sys1}_2$, we derive that
\begin{align*}
&\int_{\mathbb{T}^2}\int_{\mathbb{R}}a^{-X}\rho
\left(u-(u^s)^{-X}\right)\left(u-(u^s)^{-X}\right)_t\dif \xi_1\dif \xi^{\prime}\\
&=\int_{\mathbb{T}^2}\int_{\mathbb{R}}a^{-X}
\left(u-(u^s)^{-X}\right)\cdot
\left(\sigma\rho\partial_{\xi_1}\left(u-(u^s)^{-X}\right)_{\xi_1}
 -\rho u\cdot\nabla_\xi\left(u-(u^s)^{-X}\right)+\rho\dot{X}(t)\partial_{\xi_1}(u^s)^{-X} \right.\\
        &\qquad\quad\left.-F\partial_{\xi_1}(u^s)^{-X}-\nabla_\xi\left(p(v)-p((v^s)^{-X})\right)
        +\ddiv_\xi\left(\Pi_1-(\Pi^s)^{-X}\right)+\nabla_\xi\left(\Pi_2-(\Pi_2^s)^{-X}\right)\right)\dif \xi_1\dif \xi^{\prime}.
\end{align*}
Therefore, we conclude that
\begin{equation}\label{0j2}
\begin{aligned}
J_2=&-\dot{X}(t)\int_{\mathbb{T}^2}\int_{\mathbb{R}}a^{-X}_{\xi_1}\rho \frac{|u-(u^s)^{-X}|^2}{2}\dif \xi_1\dif \xi^{\prime}
+\dot{X}(t)\int_{\mathbb{T}^2}\int_{\mathbb{R}}a^{-X}\rho \left(u-(u^s)^{-X}\right)\cdot\partial_{\xi_1}(u^s)^{-X}\dif \xi_1\dif \xi^{\prime}\\
&+\int_{\mathbb{T}^2}\int_{\mathbb{R}}a^{-X}_{\xi_1}(-\sigma_\ast+F) \frac{|u-(u^s)^{-X}|^2}{2}\dif \xi_1\dif \xi^{\prime}
-\int_{\mathbb{T}^2}\int_{\mathbb{R}}a^{-X}F\left(u-(u_1^s)^{-X}\right)(u_1^s)^{-X}_{\xi_1}\dif \xi_1\dif \xi^{\prime}\\
&+I_3+I_4+I_5,
\end{aligned}
\end{equation}
where $I_3:=-\int_{\mathbb{T}^2}\int_{\mathbb{R}}a^{-X}\left(u-(u^s)^{-X}\right)\cdot\nabla_\xi\left(p(v)-p((v^s)^{-X})\right)\dif \xi_1\dif \xi^{\prime}$, $I_4:=\int_{\mathbb{T}^2}\int_{\mathbb{R}}a^{-X}\left(u-(u^s)^{-X}\right)\cdot\ddiv_\xi\left(\Pi_1-(\Pi^s_1)^{-X}\right)\dif \xi_1\dif \xi^{\prime}$ and $I_5=\int_{\mathbb{T}^2}\int_{\mathbb{R}}a^{-X}\left(u-(u^s)^{-X}\right)\cdot\nabla_\xi\left(\Pi_2-(\Pi_2^s)^{-X}\right)\dif \xi_1\dif \xi^{\prime}$.

For $J_3$ and $J_4$, following the same procedure as for $J_2$, it holds
\begin{equation}\label{0j3}
\begin{aligned}
J_3=&-\dot{X}(t)\int_{\mathbb{T}^2}\int_{\mathbb{R}}a^{-X}_{\xi_1}\rho \frac{\tau|\Pi_1-(\Pi_1^s)^{-X}|^2}{4\mu}\dif \xi_1\dif \xi^{\prime}
+\dot{X}(t)\int_{\mathbb{T}^2}\int_{\mathbb{R}}a^{-X}\rho \frac{\tau}{2\mu}\left(\Pi_1-(\Pi_1^s)^{-X}\right):(\Pi_1^s)^{-X}_{\xi_1}\dif \xi_1\dif \xi^{\prime}\\
&+\int_{\mathbb{T}^2}\int_{\mathbb{R}}a^{-X}_{\xi_1}(-\sigma_\ast+F) \frac{\tau|\Pi_1-(\Pi_1^s)^{-X}|^2}{4\mu}\dif \xi_1\dif \xi^{\prime}
-\int_{\mathbb{T}^2}\int_{\mathbb{R}}a^{-X}\frac{\tau}{2\mu}F\left(\Pi_1-(\Pi_1^s)^{-X}\right):(\Pi_1^s)^{-X}_{\xi_1}\dif \xi_1\dif \xi^{\prime}\\
&+\underbrace{\int_{\mathbb{T}^2}\int_{\mathbb{R}}a^{-X}\frac{\left(\Pi_1-(\Pi_1^s)^{-X}\right)}{2}
:\left(\nabla_\xi\left(u-(u^s)^{-X}\right)
+\left(\nabla_\xi\left(u-(u^s)^{-X}\right)\right)^T\right)\dif \xi_1\dif \xi^{\prime}}\limits_{=:I_6}\\
&-\int_{\mathbb{T}^2}\int_{\mathbb{R}}a^{-X}\frac{|\Pi_1-(\Pi_1^s)^{-X}|^2}{2\mu}\dif \xi_1\dif \xi^{\prime},
\end{aligned}
\end{equation}
and
\begin{equation}\label{0j4}
\begin{aligned}
J_4=&-\dot{X}(t)\int_{\mathbb{T}^2}\int_{\mathbb{R}}a^{-X}_{\xi_1}\rho \frac{\tau|\Pi_2-(\Pi_2^s)^{-X}|^2}{2\lambda}\dif \xi_1\dif \xi^{\prime}
+\dot{X}(t)\int_{\mathbb{T}^2}\int_{\mathbb{R}}a^{-X}\rho \frac{\tau}{\lambda}\left(\Pi_2-(\Pi_2^s)^{-X}\right)(\Pi_2^s)^{-X}_{\xi_1}\dif \xi_1\dif \xi^{\prime}\\
&+\int_{\mathbb{T}^2}\int_{\mathbb{R}}a^{-X}_{\xi_1}(-\sigma_\ast+F) \frac{\tau|\Pi_2-(\Pi_2^s)^{-X}|^2}{2\lambda}\dif \xi_1\dif \xi^{\prime}
-\int_{\mathbb{T}^2}\int_{\mathbb{R}}a^{-X}\frac{\tau}{\lambda}F\left(\Pi_2-(\Pi_2^s)^{-X}\right)(\Pi_2^s)^{-X}_{\xi_1}\dif \xi_1\dif \xi^{\prime}\\
&-\int_{\mathbb{T}^2}\int_{\mathbb{R}}a^{-X}\frac{|\Pi_2-(\Pi_2^s)^{-X}|^2}{\lambda}\dif \xi_1\dif \xi^{\prime}
+\underbrace{\int_{\mathbb{T}^2}\int_{\mathbb{R}}a^{-X}\left(\Pi_2-(\Pi_2^s)^{-X}\right)
\ddiv_\xi\left(u-(u^s)^{-X}\right))\dif \xi_1\dif \xi^{\prime}}\limits_{=:I_7}.
\end{aligned}
\end{equation}

A simple calculation shows that
\begin{align*}
I_2+I_3=\int_{\mathbb{T}^2}\int_{\mathbb{R}}a^{-X}_{\xi_1}\left(p(v)-p((v^s)^{-X})\right)(u_1-(u_1^s)^{-X})\dif \xi_1\dif \xi^{\prime},
\end{align*}
\begin{align*}
I_4+I_6=-\int_{\mathbb{T}^2}\int_{\mathbb{R}}a^{-X}_{\xi_1}\left(\left(u_1-(u^s_1)^{-X}\right)
\left(\Pi_{11}-(\Pi_{11}^s)^{-X}\right)+u_2\Pi_{21}+u_3\Pi_{31}
\right)\dif \xi_1\dif \xi^{\prime},
\end{align*}
\begin{align*}
I_5+I_7=-\int_{\mathbb{T}^2}\int_{\mathbb{R}}a^{-X}_{\xi_1}\left(u_1-(u_1^s)^{-X}\right)\left(\Pi_2-(\Pi_2^s)^{-X}\right)\dif \xi_1\dif \xi^{\prime},
\end{align*}
\begin{align*}
&-\sigma_\ast a^{-X}_{\xi_1}\frac{|u-(u^s)^{-X}|^2}{2}+a^{-X}_{\xi_1}\left(p(v)-p((v^s)^{-X})\right)(u_1-(u_1^s)^{-X})\\
&=-\frac{\sigma_{\ast}}{2}a^{-X}_{\xi_1}\Big|u_1-(u_1^s)^{-X}-\frac{p(v)-p((v^s)^{-X})}{\sigma_\ast}\Big|^2
+a^{-X}_{\xi_1}\frac{|p(v)-p((v^s)^{-X})|^2}{2\sigma_\ast}-\sigma_\ast a^{-X}_{\xi_1}\frac{u_2^2+u_3^2}{2},
\end{align*}
and
\begin{align*}
&a^{-X}Fp^{\prime}((v^s)^{-X})(v-(v^s)^{-X})(v^s)^{-X}_{\xi_1}+
a^{-X}F\left(u_1-(u_1^s)^{-X}\right)(u_1^s)^{-X}_{\xi_1} \\
&=a^{-X}\frac{\sigma_\ast}{v}p^{\prime}((v^s)^{-X})(v^s)^{-X}_{\xi_1}(v-(v^s)^{-X})^2
+a^{-X}\frac{1}{v}\left((\Pi_{11}^s)_{\xi_1}+(\Pi_{2}^s)_{\xi_1}\right)(v-(v^s)^{-X})\left(u_1-(u_1^s)^{-X}\right)\\
&\quad+\frac{\delta}{\nu}\frac{a^{-X}}{\sigma_\ast v}a^{-X}_{\xi_1}\left(u_1-(u_1^s)^{-X}\right)^2
+a^{-X}\frac{1}{v\sigma_\ast}\left((\Pi_{11}^s)_{\xi_1}+(\Pi_{2}^s)_{\xi_1}\right)\left(u_1-(u_1^s)^{-X}\right)^2.
\end{align*}

Combining the above estimates and \eqref{0j1}-\eqref{0j4}, the proof of this lemma is completed.
\end{proof}

To address certain terms via the $a$-contraction method, we rewrite the function $Y(t)$ as follows
\[
Y(t):=\sum\limits_{i=1}\limits^{9}Y_i(t),
\]
where
\[
Y_1(t):=\int_{\mathbb{T}^2}\int_{\mathbb{R}}\frac{a^{-X}}{\sigma_\ast}\rho(u_1^s)^{-X}_{\xi_1}
\left(p(v)-p((v^s)^{-X})\right)\dif \xi_1\dif \xi^{\prime},
\]
\[
Y_2(t):=-\int_{\mathbb{T}^2}\int_{\mathbb{R}}a^{-X}\rho\left(p(v^s)^{-X}\right)_{\xi_1}\left(v-(v^s)^{-X}\right)\dif \xi_1\dif\xi^{\prime},
\]
\[
Y_3(t):=\int_{\mathbb{T}^2}\int_{\mathbb{R}}a^{-X}\rho(u_1^s)^{-X}_{\xi_1}
\left(u_1-(u_1^s)^{-X}-\frac{p(v)-p((v^s)^{-X})}{\sigma_\ast}\right)\dif \xi_1\dif\xi^{\prime},
\]
\[
Y_4(t):=-\frac{1}{2}\int_{\mathbb{T}^2}\int_{\mathbb{R}}a^{-X}_{\xi_1}\rho
\left(u_1-(u_1^s)^{-X}-\frac{p(v)-p((v^s)^{-X})}{\sigma_\ast}\right)\cdot
\left(u_1-(u_1^s)^{-X}+\frac{p(v)-p((v^s)^{-X})}{\sigma_\ast}\right)\dif \xi_1\dif\xi^{\prime},
\]
\[
Y_5(t):=-\int_{\mathbb{T}^2}\int_{\mathbb{R}}a^{-X}_{\xi_1}\rho\left(H(v|(v^s)^{-X})
+\frac{u_2^2+u_3^2}{2}\right)\dif \xi_1\dif\xi^{\prime}
-\int_{\mathbb{T}^2}\int_{\mathbb{R}}a^{-X}_{\xi_1}\rho
\frac{|p(v)-p((v^s)^{-X})|^2}{2\sigma_\ast^2}\dif \xi_1\dif\xi^{\prime},
\]
\[
Y_6(t):=-\int_{\mathbb{T}^2}\int_{\mathbb{R}}a^{-X}_{\xi_1}\rho
\frac{\tau|\Pi_1-(\Pi_1^s)^{-X}|^2}{4\mu}\dif \xi_1\dif\xi^{\prime},
Y_7(t):=-\int_{\mathbb{T}^2}\int_{\mathbb{R}}a^{-X}_{\xi_1}\rho
\frac{\tau|\Pi_2-(\Pi_2^s)^{-X}|^2}{2\lambda}\dif \xi_1\dif\xi^{\prime},
\]
\[
Y_8(t):=\int_{\mathbb{T}^2}\int_{\mathbb{R}}a^{-X}\rho \frac{\tau}{2\mu}\left(\Pi_1-(\Pi_1^s)^{-X}\right):\left((\Pi_1^s)^{-X}\right)_{\xi_1}\dif \xi_1\dif \xi^{\prime},
\]
\[
Y_9(t):=\int_{\mathbb{T}^2}\int_{\mathbb{R}}a^{-X}\rho \frac{\tau}{\lambda}\left(\Pi_2-(\Pi_2^s)^{-X}\right)\left((\Pi_2^s)^{-X}\right)_{\xi_1}\dif \xi_1\dif \xi^{\prime}.
\]
According to \eqref{shiftx}, we note that
\begin{equation}\label{4.10}
\dot{X}(t)=-\frac{M}{\delta}(Y_1(t)+Y_2(t)),
\end{equation}
which implies
\begin{equation}\label{4.11}
\dot{X}(t)Y(t)=-\frac{\delta}{M}|\dot{X}(t)|^2+\dot{X}(t)\sum\limits_{i=3}\limits^9Y_i(t).
\end{equation}

Then, we have the following lemma.
\begin{lemma}\label{le4.6}
There exist a constant $C, C_1>0$ independent of $\tau, \nu, \delta, \varepsilon_1, T$, such that
\begin{align*}
&-\frac{\delta}{2M}|\dot{X}(t)|^2+B_1(t)+B_2(t)+B_{3}(t)-G_1(t)-G_4(t)-\frac{5}{8}D(t)\\
&\le
-C_1\int_{\mathbb{T}^2}\int_{\mathbb{R}}(v^s)^{-X}_{\xi_1}|p(v)-p((v^s)^{-X})|^2\dif \xi_1\dif \xi^{\prime}
+C\int_{\mathbb{T}^2}\int_{\mathbb{R}}a^{-X}_{\xi_1}|p(v)-p((v^s)^{-X})|^3\dif \xi_1\dif \xi^{\prime}\\
&\qquad+C\frac{\delta}{\nu}G_3(t).
\end{align*}
where $G_3(t)$ is defined in Lemma \ref{le4.5} and
\[
D(t)=(\frac{4}{3}\mu+\lambda)\int_{\mathbb{T}^2}\int_{\mathbb{R}}
a^{-X}\frac{|\nabla(p(v)-p((v^s)^{-X}))|^2}{\gamma p^{1+\frac{1}{\gamma}}(v)}\dif \xi_1\dif \xi^{\prime}.
\]
\end{lemma}
\begin{proof}
Introduce variables $y$ and $w$ as follows (see \cite{WWS}):
\[
w:=p(v)-p((v^s)^{-X}),\quad
y_1:=\frac{p(v_-)-p((v^s(\xi_1))^{-X})}{\delta},\quad
y^{\prime}=(y_2, y_3)=\xi^{\prime}.
\]
Changing variables $\xi_1\in\mathbb{R}\mapsto y_1\in(0, 1)$ and applying the definition of $a(\xi_1)$, we have
$$\frac{\dif y_1}{\dif\xi_1}=-\frac{1}{\delta}p((v^s)^{-X})_{\xi_1}$$
 and
\begin{equation}\label{4.12}
a(\xi_1)^{-X}=1+\nu y_1,\quad
a^{-X}_{\xi_1}=-\nu\frac{\dif y_1}{\dif\xi_1},\quad
|a^{-X}-1|\le \nu\le\sqrt{\delta}.
\end{equation}
In addition, we have the following estimates:
\begin{equation}\label{4.13}
|\sigma_\ast-\sigma_-|\le C\delta
\end{equation}
and
\begin{equation}\label{4.14}
|\sigma_-^2+p^{\prime}((v^s)^{-X})|\le C\delta,\quad
\Big|\frac{1}{\sigma_-^2}-\frac{1}{\gamma p^{1+\frac{1}{\gamma}}((v^s)^{-X})}\Big|\le C\delta,
\end{equation}
where $\sigma_-=\sqrt{-p^{\prime}(v_-)}$.

Firstly, we estimates the shift part $-\frac{\delta}{2M}|\dot{X}(t)|^2$. Moreover, since $\dot{X}(t)=-\frac{M}{\delta}(Y_1(t)+Y_2(t))$, we will estimate $Y_1(t)$, $Y_2(t)$.
Using $\eqref{2.6}_2$, we get
\[
Y_1(t)=\frac{1}{\sigma_\ast^2}\int_{\mathbb{T}^2}\int_{\mathbb{R}}a^{-X}\frac{1}{v}
\left(p((v^s)^{-X}_{\xi_1}-(\Pi_{11}^s)_{\xi_1}-(\Pi_2^s)_{\xi_1}\right)
\left(p(v)-p((v^s)^{-X})\right)\dif \xi_1\dif \xi^{\prime}.
\]
Using Lemma \ref{pvsw} and applying the new variables $y$ and $w$, we deduce that
\[
\frac{1}{\sigma_\ast^2}\int_{\mathbb{T}^2}\int_{\mathbb{R}}a^{-X}\frac{1}{v}
p((v^s)^{-X}_{\xi_1}\left(p(v)-p((v^s)^{-X})\right)\dif \xi_1\dif \xi^{\prime}
=-\frac{\delta}{\sigma_\ast^2}\int_{\mathbb{T}^2}\int_0^1a^{-X}\frac{w}{v}\dif y_1\dif y^{\prime},
\]
and
\begin{align*}
-\frac{1}{\sigma_\ast^2}\int_{\mathbb{T}^2}\int_{\mathbb{R}}a^{-X}\frac{1}{v}
\left(\Pi_{11}^s)_{\xi_1}+(\Pi_2^s)_{\xi_1}\right)
\left(p(v)-p((v^s)^{-X})\right)\dif \xi_1\dif \xi^{\prime}
&\le C\delta\int_{\mathbb{T}^2}\int_{\mathbb{R}}(v^s)^{-X}_{\xi_1}|w|\dif \xi_1\dif \xi^{\prime}\\
&\le C\delta^2\int_{\mathbb{T}^2}\int_0^1|w|\dif y_1\dif y^{\prime}.
\end{align*}
Thus, using \eqref{4.12} and \eqref{4.14}, it holds
\begin{equation}\label{4.15}
\Big|Y_1(t)+\frac{\delta}{\sigma_-^2v_-}\int_{\mathbb{T}^2}\int_0^1w\dif y_1\dif y^{\prime}\Big|
\le C\delta(\nu+\delta)\int_0^1|w|\dif y_1\dif y^{\prime}.
\end{equation}
For $Y_2(t)$, we have
\[
Y_2(t)=-\int_{\mathbb{T}^2}\int_{\mathbb{R}}a^{-X}\rho\left(p(v^s)\right)_{\xi_1}^{-X}\left(v-(v^s)^{-X}\right)\dif \xi_1\dif\xi^{\prime}
=\delta\int_{\mathbb{T}^2}\int_0^1a^{-X}\frac{v-(v^s)^{-X}}{v}\dif y_1\dif y^{\prime}.
\]
On the other hand, noting that $v=p^{-\frac{1}{\gamma}}(v)$, we get
\[
\Big|v-(v^s)^{-X}-\frac{p(v)-p((v^s)^{-X})}{\gamma p^{\frac{1}{\gamma}+1}((v^s)^{-X})}\Big|
\le C|p(v)-p((v^s)^{-X})|^2.
\]
Together with \eqref{4.14} and \eqref{3.4}, it imply that
\[
\Big|v-(v^s)^{-X}-\frac{p(v)-p((v^s)^{-X})}{\sigma_-^2}\Big|
\le C(\delta+\varepsilon_1)|p(v)-p((v^s)^{-X})|.
\]
Then, we derive that
\begin{equation}\label{4.16}
\Big|Y_2(t)+\frac{\delta}{\sigma_-^2v_-}\int_{\mathbb{T}^2}\int_0^1w\dif y_1\dif y^{\prime}\Big|
\le C\delta(\nu+\delta+\varepsilon_1)\int_0^1|w|\dif y_1\dif y^{\prime}.
\end{equation}
Combining \eqref{4.10}, \eqref{4.15} and \eqref{4.16}, we conclude that
\begin{align*}
\Big|\dot{X}(t)-\frac{2M}{\sigma_-^2v_-}\int_{\mathbb{T}^2}\int_0^1w\dif y_1\dif y^{\prime}\Big|
&\le \frac{M}{\delta}\sum\limits_{i=1}^2
\Big|Y_i(t)+\frac{\delta}{\sigma_-^2v_-}\int_{\mathbb{T}^2}\int_0^1w\dif y_1\dif y^{\prime}\Big|\\
&\le C(\nu+\delta+\varepsilon_1)\int_{\mathbb{T}^2}\int_0^1|w|\dif y_1\dif y^{\prime},
\end{align*}
which implies
\[
\left(\Big|\frac{2M}{\sigma_-^2v_-}\int_{\mathbb{T}^2}\int_0^1w\dif y_1\dif y^{\prime}\Big|-|\dot{X}(t)|\right)^2
\le C(\nu+\delta+\varepsilon_1)^2\int_{\mathbb{T}^2}\int_0^1|w|^2\dif y_1\dif y^{\prime}.
\]
Therefore, we have
\begin{equation}\label{x1}
-\frac{\delta}{2M}|\dot{X}(t)|^2\le -\frac{M\delta}{\sigma_-^4v_-^2}
\left(\int_{\mathbb{T}^2}\int_0^1w\dif y_1\dif y^{\prime}\right)^2
+C\delta(\nu+\delta+\varepsilon_1)^2\int_{\mathbb{T}^2}\int_0^1|w|^2\dif y_1\dif y^{\prime}.
\end{equation}

Next, we estimate the bad terms $B_i(t)$ $(i=1,2,3)$. For $B_1(t)$, changing the variables and using \eqref{4.13}, we have
\[
B_1(t)=\frac{\nu}{2\sigma_\ast}\int_{\mathbb{T}^2}\int_0^1w^2\dif y_1\dif y^{\prime}
\le \frac{\nu}{2\sigma_-}\int_{\mathbb{T}^2}\int_0^1w^2\dif y_1\dif y^{\prime}
+C\nu\delta\int_{\mathbb{T}^2}\int_0^1w^2\dif y_1\dif y^{\prime}.
\]

For $B_2(t)$, using \eqref{4.14}, Lemma 2.1 in \cite{WY} and changing the variables, we deduce that
\begin{equation}\label{b2}
\begin{aligned}
B_2(t)&=\sigma_\ast\delta\int_{\mathbb{T}^2}\int_0^1
(1+\nu y_1)p(v|(v^s)^{-X})\frac{1}{-p^{\prime}((v^s)^{-X})}\dif y_1\dif y^{\prime}\\
&\le\sigma_\ast\delta(1+\nu)\int_{\mathbb{T}^2}\int_0^1
\left(\frac{\gamma+1}{2\gamma}\frac{1}{p((v^s)^{-X})}+C\varepsilon_1\right)
\frac{|p(v)-p((v^s)^{-X})|^2}{|p^{\prime}((v^s)^{-X})|}\dif y_1\dif y^{\prime}\\
&\le (C\delta+\sigma_-)\delta(1+\nu)\int_{\mathbb{T}^2}\int_0^1
\left(\frac{\gamma+1}{2\gamma\sigma_-p(v_-)}\frac{\sigma_-p(v_-)}{p((v^s)^{-X})}+C\varepsilon_1\right)
w^2\left(C\delta+\frac{1}{\sigma_-^2}\right)\dif y_1\dif y^{\prime}\\
&\le (C\delta+\sigma_-)\delta(1+\nu)\int_{\mathbb{T}^2}\int_0^1
\left(\alpha_-\sigma_-(1+C\delta)+C\varepsilon_1\right)
w^2\left(C\delta+\frac{1}{\sigma_-^2}\right)\dif y_1\dif y^{\prime}\\
&\le \delta\alpha_-(1+C(\nu+\delta+\varepsilon_1))\int_{\mathbb{T}^2}\int_0^1
w^2\dif y_1\dif y^{\prime},
\end{aligned}
\end{equation}
where $\alpha_-=\frac{\gamma+1}{2\gamma\sigma_-p(v_-)}$.

For $B_3(t)$, using Young's inequality, we can obtain that there exists a constant $\kappa>0$,
\begin{align*}
&\frac{\delta}{\nu}\int_{\mathbb{T}^2}\int_{\mathbb{R}}\frac{a^{-X}}{\sigma_\ast v}a^{-X}_{\xi_1}\left(u_1-(u_1^s)^{-X}\right)^2\dif \xi_1\dif \xi^{\prime}\\
&=\frac{\delta}{\nu}\int_{\mathbb{T}^2}\int_{\mathbb{R}}\frac{a^{-X}}{\sigma_\ast v}a^{-X}_{\xi_1}\left(u_1-(u_1^s)^{-X}-\frac{p(v)-p((v^s)^{-X})}{\sigma_\ast}
+\frac{p(v)-p((v^s)^{-X})}{\sigma_\ast}\right)^2\dif \xi_1\dif \xi^{\prime}\\
&\le C(1+\frac{1}{\kappa})\frac{\delta}{\nu}G_3(t)+
\frac{\delta}{\nu}\int_{\mathbb{T}^2}\int_{\mathbb{R}}\frac{a^{-X}}{\sigma_\ast^3 v}a^{-X}_{\xi_1}(1+\kappa)\left(p(v)-p((v^s)^{-X})\right)^2\dif \xi_1\dif \xi^{\prime}.
\end{align*}
Subsequently, by using \eqref{4.12} and \eqref{4.13}, we can get
\begin{align*}
&\frac{\delta}{\nu}\int_{\mathbb{T}^2}\int_{\mathbb{R}}\frac{a^{-X}}{\sigma_\ast^3 v}a^{-X}_{\xi_1}(1+\kappa)\left(p(v)-p((v^s)^{-X})\right)^2\dif \xi_1\dif \xi^{\prime}\\
&=\delta(1+\kappa)\int_{\mathbb{T}^2}\int_0^1\frac{1+\nu y_1}{\sigma_-^3 v_-}\frac{\sigma_-^3v_-}{\sigma_\ast^3v}w^2\dif y_1\dif y^{\prime}\\
&\le \delta(1+\kappa)(1+\nu)(1+C\varepsilon_1)(1+C\delta)\int_{\mathbb{T}^2}\int_0^1\frac{1}{\sigma_-^3 v_-}w^2\dif y_1\dif y^{\prime}\\
&\le (1+\kappa)\frac{2}{\gamma+2}\delta\alpha_-(1+C(\nu+\varepsilon_1+\delta))
\int_{\mathbb{T}^2}\int_0^1w^2\dif y_1\dif y^{\prime},
\end{align*}
where
\[
\frac{1}{\sigma_-^3v_-\alpha_-}
=\frac{1}{\sigma_-^3v_-}\frac{2\gamma\sigma_-p(v_-)}{\gamma+1}
=\frac{2}{\gamma+1}.
\]
Therefore, we conclude that
\begin{equation}\label{b3}
B_3(t)\le(1+\kappa)\frac{2}{\gamma+1}\delta\alpha_-(1+C(\nu+\varepsilon_1+\delta))
\int_{\mathbb{T}^2}\int_0^1w^2\dif y_1\dif y^{\prime}
+C(1+\frac{1}{\kappa})\frac{\delta}{\nu}G_3(t).
\end{equation}

Next, we deal with good terms $G_1(t)$ and $G_4(t)$. For $G_1(t)$, using Lemma 2.1 in \cite{WY}, we have
\begin{equation}\label{g111}
\begin{aligned}
G_1(t)&=\sigma_\ast\int_{\mathbb{T}^2}\int_{\mathbb{R}}a^{-X}_{\xi_1} H(v|(v^s)^{-X})\dif \xi_1\dif \xi^{\prime}\\
&\ge \underbrace{\sigma_\ast\int_{\mathbb{T}^2}\int_{\mathbb{R}}a^{-X}_{\xi_1} \frac{(p(v)-p((v^s)^{-X}))^2}{2\gamma p^{1+\frac{1}{\gamma}}((v^s)^{-X})}\dif \xi_1\dif \xi^{\prime}}\limits_{=:\mathcal{G}(t)}
-\sigma_\ast\int_{\mathbb{T}^2}\int_{\mathbb{R}}a^{-X}_{\xi_1} \frac{1+\gamma}{3\gamma^2} \frac{(p(v)-p((v^s)^{-X}))^3}{2\gamma p^{2+\frac{1}{\gamma}}((v^s)^{-X})}\dif \xi_1\dif \xi^{\prime}.
\end{aligned}
\end{equation}
By using \eqref{4.13}, \eqref{4.14} and changing variables, we can get
\begin{align*}
\mathcal{G}(t)&\ge \frac{1}{2\sigma_-}\frac{\sigma_\ast}{\sigma_-}(1-C\delta)\int_{\mathbb{T}^2}\int_{\mathbb{R}}a^{-X}_{\xi_1} (p(v)-p((v^s)^{-X})^2\dif \xi_1\dif \xi^{\prime}\\
&\ge \frac{\nu}{2\sigma_-}(1-C\delta)\int_{\mathbb{T}^2}\int_0^1w^2\dif y_1\dif y^{\prime},
\end{align*}
combining the estimates for $B_1(t)$, it holds
\begin{equation}\label{bg1}
B_1(t)-\mathcal{G}(t)\le C\nu\delta\int_{\mathbb{T}^2}\int_0^1w^2\dif y_1\dif y^{\prime}.
\end{equation}

For $G_4(t)$, we first noting that
\[
v-(v^s)^{-X}=\frac{p(v)-p((v^s)^{-X})}{p^{\prime}(\zeta)},
\]
where $\zeta$ between $v$ and $(v^s)^{-X}$.
Then, using new variables and \eqref{4.13}, we have
\begin{equation}\label{g4}
\begin{aligned}
G_4(t)&=-\int_{\mathbb{T}^2}\int_{\mathbb{R}}a^{-X}\frac{\sigma_\ast}{v}p^{\prime}((v^s)^{-X})(v^s)^{-X}_{\xi_1}(v-(v^s)^{-X})^2\dif \xi_1\dif \xi^{\prime}\\
&=\delta\int_{\mathbb{T}^2}\int_0^1a^{-X}\frac{\sigma_\ast}{v}\frac{w^2}{|p^{\prime}(\zeta)|^2}\dif y_1\dif y^{\prime}\\
&\ge \delta(\sigma_--C\delta)(\frac{1}{v_-}-C\varepsilon_1)
\int_{\mathbb{T}^2}\int_0^1\frac{1}{|p^{\prime}(v_-)|^2}\frac{|p^{\prime}(v_-)|^2}{|p^{\prime}(\zeta)|^2}w^2\dif y_1\dif y^{\prime}\\
&\ge \frac{\delta\sigma_-}{v_-|p^{\prime}(v_-)|^2}(1-C(\delta+\varepsilon_1))\int_{\mathbb{T}^2}
\int_0^1w^2\dif y_1\dif y^{\prime}\\
&=\frac{2}{\gamma+1}\delta\alpha_-(1-C(\delta+\varepsilon_1))\int_{\mathbb{T}^2}
\int_0^1w^2\dif y_1\dif y^{\prime}.
\end{aligned}
\end{equation}

Next, we estimate $D(t)$, applying new variables, we can obtain that
\begin{align*}
D(t)&=(\frac{4}{3}\mu+\lambda)\int_{\mathbb{T}^2}\int_{\mathbb{R}}
a^{-X}\frac{|\nabla_\xi(p(v)-p((v^s)^{-X}))|^2}{\gamma p^{1+\frac{1}{\gamma}}(v)}\dif \xi_1\dif \xi^{\prime}\\
&\ge(\frac{4}{3}\mu+\lambda)\int_{\mathbb{T}^2}\int_{\mathbb{R}}
\frac{|\partial_{\xi_1}(p(v)-p((v^s)^{-X}))|^2}{\gamma p^{1+\frac{1}{\gamma}}(v)}\dif \xi_1\dif \xi^{\prime}\\
&\qquad+(\frac{4}{3}\mu+\lambda)\int_{\mathbb{T}^2}\int_{\mathbb{R}}
\frac{|\nabla_{\xi^\prime}(p(v)-p((v^s)^{-X}))|^2}{\gamma p^{1+\frac{1}{\gamma}}(v)}\dif \xi_1\dif \xi^{\prime}\\
&=\underbrace{(\frac{4}{3}\mu+\lambda)\int_{\mathbb{T}^2}\int_0^1
\frac{|\partial_{y_1}w|^2}{\gamma p^{1+\frac{1}{\gamma}}(v)}\frac{\dif y_1}{\dif \xi_1}\dif y_1\dif y^{\prime}}\limits_{=:\mathcal{D}_1}
+\underbrace{(\frac{4}{3}\mu+\lambda)\int_{\mathbb{T}^2}\int_0^1
\frac{|\nabla_{y^\prime}w|^2}{\gamma p^{1+\frac{1}{\gamma}}(v)}\frac{\dif \xi_1}{\dif y_1}\dif y_1\dif y^{\prime}}\limits_{=:\mathcal{D}_2}.
\end{align*}
For $\mathcal{D}_1$, we have
\begin{align*}
	\mathcal{D}_1=\int_{\mathbb{T}^2}\int_0^1y_1(1-y_1)|\partial_{y_1}w|^2
	\left(\frac{p((v^s)^{-X})}{p(v)}\right)^{1+\frac{1}{\gamma}}
	\frac{1}{ y_1(1-y_1)}
	\frac{\frac{4}{3}\mu+\lambda}{\gamma p^{1+\frac{1}{\gamma}}((v^s)^{-X})
	}\frac{\dif y_1}{\dif \xi_1}
	\dif y_1\dif y^{\prime}.
\end{align*}
Subsequently, using \eqref{2.10}, it imply
\[
\frac{(\frac{4}{3}\mu+\lambda)}{\gamma p^{1+\frac{1}{\gamma}}((v^s)^{-X})}\frac{\dif y_1}{\dif \xi_1}=\frac{\frac{4}{3}\mu+\lambda}{\tau h^\prime((v^s)^{-X})+\frac{4}{3}\mu+\lambda}
\frac{-\left(\sigma_\ast^2((v^s)^{-X}-v_-)+\left(p((v^s)^{-X})-p(v_-)\right)\right)}{\delta\sigma_\ast}.
\]
On the other hand, using \eqref{2.8}, we can get
\begin{align*}
&\frac{-\left(\sigma_\ast^2((v^s)^{-X}-v_-)+\left(p((v^s)^{-X})-p(v_-)\right)\right)}{\sigma_\ast}\\
&=\frac{-1}{\sigma_\ast(v_--v_+)}\left((p(v_-)-p(v_+))((v^s)^{-X}-v_-)
-\left(p((v^s)^{-X})-p(v_-))(v_+-v_-)\right)\right)\\
&=\frac{-1}{\sigma_\ast(v_--v_+)}\left((p((v^s)^{-X})-p(v_+))((v^s)^{-X}-v_-)
-\left(p((v^s)^{-X})-p(v_-))((v^s)^{-X}-v_+)\right)\right),
\end{align*}
which together with $y_1=\frac{p(v_-)-p((v^s(\xi_1))^{-X})}{\delta}$ and $1-y_1=\frac{p((v^s(\xi_1))^{-X})-p(v_+)}{\delta}$, we can deduce that
\begin{align*}
&\frac{1}{ y_1(1-y_1)}
\frac{-\left(\sigma_\ast^2((v^s)^{-X}-v_-)+\left(p((v^s)^{-X})-p(v_-)\right)\right)}{ \sigma_\ast}\\
&=\frac{\delta^2}{\sigma_\ast(v_--v_+)}\left(\frac{(v^s)^{-X}-v_-}{p((v^s)^{-X})-p(v_-)}
-\frac{(v^s)^{-X}-v_+}{p((v^s)^{-X})-p(v_+)}\right).
\end{align*}
Then, we derive that
\begin{align*}
&\frac{1}{ y_1(1-y_1)}
\frac{\frac{4}{3}\mu+\lambda}{\gamma p^{1+\frac{1}{\gamma}}((v^s)^{-X})
}\frac{\dif y_1}{\dif \xi_1}\\
&=\frac{(\frac{4}{3}\mu+\lambda)}{\tau h^\prime((v^s)^{-X})+\frac{4}{3}\mu+\lambda}
\frac{\delta}{\sigma_\ast(v_--v_+)}\left(\frac{(v^s)^{-X}-v_-}{p((v^s)^{-X})-p(v_-)}
-\frac{(v^s)^{-X}-v_+}{p((v^s)^{-X})-p(v_+)}\right).
\end{align*}
Therefore, we can conclude that
\begin{align*}
&\Big|\frac{1}{ y_1(1-y_1)}
\frac{\frac{4}{3}\mu+\lambda}{\gamma p^{1+\frac{1}{\gamma}}((v^s)^{-X})}\frac{\dif y_1}{\dif \xi_1}-\frac{\delta p^{\prime\prime}(v_-)}{2\sigma_-(p^\prime(v_-))^2}\Big|\\
&=\Big|\left(\frac{(\frac{4}{3}\mu+\lambda)}{\tau h^\prime((v^s)^{-X})+\frac{4}{3}\mu+\lambda}-1\right)
\frac{\delta}{\sigma_\ast(v_--v_+)}\left(\frac{(v^s)^{-X}-v_-}{p((v^s)^{-X})-p(v_-)}
-\frac{(v^s)^{-X}-v_+}{p((v^s)^{-X})-p(v_+)}\right)\Big|\\
&\quad+\Big|\frac{\delta}{\sigma_\ast(v_--v_+)}\left(\frac{(v^s)^{-X}-v_-}{p((v^s)^{-X})-p(v_-)}
-\frac{(v^s)^{-X}-v_+}{p((v^s)^{-X})-p(v_+)}\right)-\frac{\delta p^{\prime\prime}(v_-)}{2\sigma_\ast(p^\prime(v_-))^2}\Big|\\
&\quad+\Big|\frac{\delta p^{\prime\prime}(v_-)}{2\sigma_\ast(p^\prime(v_-))^2}-\frac{\delta p^{\prime\prime}(v_-)}{2\sigma_-(p^\prime(v_-))^2}\Big|=:I_1+I_2+I_3.
\end{align*}
Using \eqref{sccs}, \eqref{4.13}, Lemma \ref{pvsw} and Lemma 3.1 in \cite{KV3}, we have
\begin{align*}
I_1&=\Big|\frac{\tau h^\prime((v^s)^{-X})}{\tau h^\prime((v^s)^{-X})+\frac{4}{3}\mu+\lambda}
\frac{\delta}{\sigma_\ast(v_--v_+)}\left(\frac{(v^s)^{-X}-v_-}{p((v^s)^{-X})-p(v_-)}
-\frac{(v^s)^{-X}-v_+}{p((v^s)^{-X})-p(v_+)}\right)\Big|\\
&\le \Big|\frac{2h^\prime((v^s)^{-X})}{\frac{4}{3}\mu+\lambda}
\frac{\delta}{\sigma_\ast(v_--v_+)}\left(\frac{(v^s)^{-X}-v_-}{p((v^s)^{-X})-p(v_-)}
-\frac{(v^s)^{-X}-v_+}{p((v^s)^{-X})-p(v_+)}\right)\Big|\\
&\le C\delta^3,
\end{align*}
\begin{align*}
I_2&=\Big|\frac{\delta}{\sigma_\ast(v_--v_+)}\left(\frac{(v^s)^{-X}-v_-}{p((v^s)^{-X})-p(v_-)}
-\frac{(v^s)^{-X}-v_+}{p((v^s)^{-X})-p(v_+)}-\frac{ p^{\prime\prime}(v_-)}{2\sigma_\ast(p^\prime(v_-))^2}(v_--v_+)\right)\Big|\\
&\le C\delta^2,
\end{align*}
and
\[
I_3=\Big|\frac{\delta p^{\prime\prime}(v_-)}{2(p^\prime(v_-))^2}\left(\frac{1}{\sigma_\ast}-\frac{1}{\sigma_-}\right)\Big|\le C\delta^2.
\]
Thus, we can get
\begin{equation}\label{5d1}
\Big|\frac{1}{ y_1(1-y_1)}
\frac{\frac{4}{3}\mu+\lambda}{\gamma p^{1+\frac{1}{\gamma}}((v^s)^{-X})}\frac{\dif y_1}{\dif \xi_1}-\frac{\delta p^{\prime\prime}(v_-)}{2\sigma_-(p^\prime(v_-))^2}\Big|\le C\delta^2.
\end{equation}
In addition, using Sobolev's embedding theorem and \eqref{3.4}, we note that
\[
\Big|\left(\frac{p((v^s)^{-X})}{p(v)}\right)^{1+\frac{1}{\gamma}}-1\Big|\le C|v-(v^s)^{-X}|\le C\varepsilon_1,
\]
and
\begin{equation}\label{5d2}
\frac{p^{\prime\prime}(v_-)}{2\sigma_-(p^\prime(v_-))^2}=\frac{\gamma+1}{2\gamma\sigma_-p(v_-)}=\alpha_-.
\end{equation}\
Therefore, we derive that
\begin{equation}\label{d1}
\begin{aligned}
\mathcal{D}_1
&\ge (1-C\varepsilon_1)\left(\frac{\delta p^{\prime\prime}(v_-)}{2\sigma_-(p^\prime(v_-))^2}-C\delta_2\right)
\int_{\mathbb{T}^2}\int_0^1y_1(1-y_1)|\partial_{y_1}w|^2
\dif y_1\dif y^{\prime}\\
&\ge \delta\alpha_-(1-C(\delta+\varepsilon_1))\int_{\mathbb{T}^2}\int_0^1y_1(1-y_1)|\partial_{y_1}w|^2
\dif y_1\dif y^{\prime}.
\end{aligned}
\end{equation}
For $\mathcal{D}_2$, firstly, applying \eqref{5d1} and \eqref{5d2}, we have
\[
\frac{1}{ y_1(1-y_1)}
\frac{\frac{4}{3}\mu+\lambda}{\gamma p^{1+\frac{1}{\gamma}}((v^s)^{-X})}\frac{\dif y_1}{\dif \xi_1}\le C\delta^2+\delta\alpha_-.
\]
Subsequently, choosing $\delta$ small enough such that $C\delta\le \alpha_-$, we can deduce that
\[
y_1(1-y_1)\frac{\dif \xi_1}{\dif y_1}\ge
\frac{\frac{4}{3}\mu+\lambda}{|p^{\prime}((v^s)^{-X})|}\frac{1} {C\delta^2+\delta\alpha_-}
\ge \frac{\frac{4}{3}\mu+\lambda}{2\delta\alpha_-|p^{\prime}(v_-)|}.
\]
Thus, we conclude that
\begin{equation}\label{d2}
\begin{aligned}
\mathcal{D}_2&=(\frac{4}{3}\mu+\lambda)\int_{\mathbb{T}^2}\int_0^1
\frac{|\nabla_{y^\prime}w|^2}{y_1(1-y_1)}\left(\frac{p(v_-)}{p(v)}\right)^{1+\frac{1}{\gamma}}
\frac{y_1(1-y_1)}{\gamma p^{1+\frac{1}{\gamma}}(v_-)}\frac{\dif \xi_1}{\dif y_1}\dif y_1\dif y^{\prime}\\
&\ge(\frac{4}{3}\mu+\lambda)(1-C(\delta+\varepsilon_1))\int_{\mathbb{T}^2}\int_0^1
\frac{|\nabla_{y^\prime}w|^2}{y_1(1-y_1)}
\frac{\frac{4}{3}\mu+\lambda}{2\delta\alpha_-|p^{\prime}(v_-)|^2}
\dif y_1\dif y^{\prime}\\
&\ge\frac{(\frac{4}{3}\mu+\lambda)^2}{p^{\prime\prime}(v_-)}\frac{\sigma_-}{\delta} (1-C(\delta+\varepsilon_1))\int_{\mathbb{T}^2}\int_0^1
\frac{|\nabla_{y^\prime}w|^2}{y_1(1-y_1)}
\dif y_1\dif y^{\prime}.
\end{aligned}
\end{equation}
Combining \eqref{d1} and \eqref{d2}, we get
\begin{equation}\label{ddd}
\begin{aligned}
D(t)\ge&\delta\alpha_-(1-C(\delta+\varepsilon_1))\int_{\mathbb{T}^2}\int_0^1y_1(1-y_1)|\partial_{y_1}w|^2
\dif y_1\dif y^{\prime}\\
&+\frac{(\frac{4}{3}\mu+\lambda)^2}{p^{\prime\prime}(v_-)}\frac{\sigma_-}{\delta} (1-C(\delta+\varepsilon_1))\int_{\mathbb{T}^2}\int_0^1
\frac{|\nabla_{y^\prime}w|^2}{y_1(1-y_1)}
\dif y_1\dif y^{\prime}.
\end{aligned}
\end{equation}

Next, combining \eqref{b2}, \eqref{b3}, \eqref{bg1}, \eqref{g4} and \eqref{ddd}, we can derive that
\begin{align*}
&B_1(t)+B_2(t)+B_3(t)-\mathcal{G}_1(t)-G_4(t)-\frac{5}{8}D(t)\\
\le &C\nu\delta\int_{\mathbb{T}^2}\int_0^1w^2\dif y_1\dif y^{\prime}
+\delta\alpha_-(1+C(\nu+\delta+\varepsilon_1))\int_{\mathbb{T}^2}\int_0^1
w^2\dif y_1\dif y^{\prime}\\
&+(1+\kappa)\frac{2\delta\alpha_-}{\gamma+1}(1+C(\nu+\varepsilon_1+\delta))\int_{\mathbb{T}^2}\int_0^1
w^2\dif y_1\dif y^{\prime}
-\frac{2\delta\alpha_-}{\gamma+1}(1-C(\delta+\varepsilon_1))\int_{\mathbb{T}^2}\int_0^1
w^2\dif y_1\dif y^{\prime}
\\
&
+C(1+\frac{1}{\kappa})\frac{\delta}{\nu}G_3(t)
-\frac{3}{4}\delta\alpha_-(1-C(\delta+\varepsilon_1))\int_{\mathbb{T}^2}\int_0^1y_1(1-y_1)|\partial_{y_1}w|^2
\dif y_1\dif y^{\prime}\\
& -\frac{3}{4}\frac{(\frac{4}{3}\mu+\lambda)^2}{p^{\prime\prime}(v_-)}\frac{\sigma_-}{\delta} (1-C(\delta+\varepsilon_1))\int_{\mathbb{T}^2}\int_0^1
\frac{|\nabla_{y^\prime}w|^2}{y_1(1-y_1)}
\dif y_1\dif y^{\prime}.
\end{align*}
Choosing $\delta, \nu, \varepsilon_1$ small enough, we have
\begin{align*}
&B_1(t)+B_2(t)+B_3(t)-\mathcal{G}_1(t)-G_4(t)-\frac{5}{8}D(t)\\
\le&
\frac{11}{10}\delta\alpha_-\int_{\mathbb{T}^2}\int_0^1
w^2\dif y_1\dif y^{\prime}
-\frac{9}{16}\delta\alpha_-\int_{\mathbb{T}^2}\int_0^1y_1(1-y_1)|\partial_{y_1}w|^2
\dif y_1\dif y^{\prime}\\
& -\frac{9}{16}\frac{(\frac{2}{3}\mu+\lambda)^2}{p^{\prime\prime}(v_-)}\frac{\sigma_-}{\delta} \int_{\mathbb{T}^2}\int_0^1
\frac{|\nabla_{y^\prime}w|^2}{y_1(1-y_1)}
\dif y_1\dif y^{\prime}
+C\frac{\delta}{\nu}G_3(t).
\end{align*}
Using 3-D Poincar\'e inequality in Lemma \ref{poinbds}, we derive that
\begin{align*}
&B_1(t)+B_2(t)+B_3(t)-\mathcal{G}_1(t)-G_4(t)-\frac{5}{8}D(t)\\
\le&
\frac{11}{10}\delta\alpha_-\int_{\mathbb{T}^2}\int_0^1
w^2\dif y_1\dif y^{\prime}
-\frac{9}{8}\delta\alpha_-\int_{\mathbb{T}^2}\int_0^1|w-\bar w|^2
\dif y_1\dif y^{\prime}
+\frac{5\delta\alpha_-}{64\pi}\int_{\mathbb{T}^2}\int_0^1
\frac{|\nabla_{y^\prime}w|^2}{y_1(1-y_1)}
\dif y_1\dif y^{\prime}\\
&-\frac{9}{16}\frac{(\frac{4}{3}\mu+\lambda)^2}{p^{\prime\prime}(v_-)}\frac{\sigma_-}{\delta} \int_{\mathbb{T}^2}\int_0^1
\frac{|\nabla_{y^\prime}w|^2}{y_1(1-y_1)}
\dif y_1\dif y^{\prime}
+C\frac{\delta}{\nu}G_3(t),
\end{align*}
where $\bar w=\int_{\mathbb{T}^2}\int_0^1
w\dif y_1\dif y^{\prime}$.
Noting that
\[
\int_{\mathbb{T}^2}\int_0^1|w-\bar w|^2
\dif y_1\dif y^{\prime}
=\int_{\mathbb{T}^2}\int_0^1w^2
\dif y_1\dif y^{\prime}
-\bar w^2,
\]
then, we have
\begin{equation}\label{bdg}
\begin{aligned}
&B_1(t)+B_2(t)+B_3(t)-\mathcal{G}_1(t)-G_4(t)-\frac{5}{8}D(t)\\
\le&
-\frac{\delta\alpha_-}{40}\int_{\mathbb{T}^2}\int_0^1
w^2\dif y_1\dif y^{\prime}
-\frac{9}{8}\delta\alpha_-\left(\int_{\mathbb{T}^2}\int_0^1w
\dif y_1\dif y^{\prime}\right)^2\\
&
-\frac{9}{16}\left(\frac{(\frac{4}{3}\mu+\lambda)^2}{p^{\prime\prime}(v_-)}\frac{\sigma_-}{\delta}
-\frac{\delta\alpha_-}{5\pi}
\right)
\int_{\mathbb{T}^2}\int_0^1
\frac{|\nabla_{y^\prime}w|^2}{y_1(1-y_1)}
\dif y_1\dif y^{\prime}
+C\frac{\delta}{\nu}G_3(t).
\end{aligned}
\end{equation}
Combining \eqref{x1}, \eqref{g111}, \eqref{bdg} and choosing $M=\frac{9}{8}\alpha_-\sigma_-^4v_-^2$ and $\delta\le(\frac{4}{3}\mu+\lambda)\sqrt{\frac{5\pi\sigma_-}{p^{\prime\prime}(v_-)\alpha_-}}$, we can deduce that
\[
\begin{aligned}
-&\frac{\delta}{2M}|\dot X(t)|^2+B_1(t)+B_2(t)+B_3(t)-\mathcal{G}_1(t)-G_4(t)-\frac{5}{8}D(t)\\
&\le
-\frac{\delta\alpha_-}{80}\int_{\mathbb{T}^2}\int_0^1
w^2\dif y_1\dif y^{\prime}
+C\int_{\mathbb{T}^2}\int_{\mathbb{R}}a^{-X}_{\xi_1} |p(v)-p((v^s)^{-X})|^3\dif \xi_1\dif \xi^{\prime}
+C\frac{\delta}{\nu}G_3(t).
\end{aligned}
\]

Then, the proof of this lemma is finished.
\end{proof}

Next, we prove the $L^2$ estimate by estimating the relative entropy quantity $\eta(t, \xi)$ in following lemma.
\begin{lemma}\label{le0}
Under the hypotheses of Proposition \ref{p1}, there exists constant $C>0$ independent of $\tau, \nu, \delta, \varepsilon_1, T$, such that for $t\in[0, T]$, we have
\begin{equation}\label{ljgu1.1}
\begin{aligned}
& \int_{\mathbb{T}^2}\int_{\mathbb{R}}\rho\eta(t,\xi)\dif \xi_1\dif \xi^{\prime}
+\frac{\delta}{4M}\int_0^t|\dot X(t)|^2\dif t
+(1-C\varepsilon_1)\int_0^tG_2(t)\dif t
+(1-C(\frac{\delta}{\nu}+\delta+\varepsilon_1)\int_0^tG_3(t)\dif t\\
&\quad
+(C_1-C(\varepsilon_1+\sqrt\nu+\delta))\int_0^tG^s(t)\dif t
+(1-C\varepsilon_1)\int_0^t\|(\frac{\Pi_1-(\Pi_1^s)^{-X}}{2\mu}, \frac{\Pi_2-(\Pi_2^s)^{-X}}{\lambda})\|_{L^2}^2\dif t\\
&\le C\left(\|v_0-v^s\|_{L^2}^2+\|u_0-u^s\|_{L^2}^2+\tau\|\Pi_{10}-\Pi_1^s\|_{L^2}^2
+\tau\|\Pi_{20}-\Pi_2^s\|_{L^2}^2\right)\\
&\quad+(\frac{5\mu}{6}+\frac{5\lambda}{8})(1+\chi+C(\nu+\varepsilon_1))\int_0^t\int_{\mathbb{T}^2}\int_{\mathbb{R}}
|p^\prime(v)||\nabla(v-(v^s)^{-X})|^2
\dif \xi_1\dif \xi^{\prime}\dif t,
\end{aligned}
\end{equation}
where the terms $G_2(t),  G_{3}(t)$ are defined in Lemma \ref{le4.5}, the parameter $0<\chi\ll 1$ chosen to be sufficiently small later, and
\[
G^s(t)=\int_{\mathbb{T}^2}\int_{\mathbb{R}}(v^s)^{-X}_{\xi_1}|p(v)-p((v^s)^{-X})|^2\dif \xi_1\dif \xi^{\prime}.
\]
\end{lemma}
\begin{proof}
Firstly, using Lemma \ref{le4.5}, we have
\begin{align*}
\frac{\dif}{\dif t}&\int_{\mathbb{T}^2}\int_{\mathbb{R}}a^{-X}\rho\eta(t, \xi)\dif \xi_1\dif \xi^{\prime}\\
&=-\frac{\delta}{2M}|\dot{X}(t)|^2+B_1(t)+B_2(t)+B_{3}(t)-G_1(t)-G_4(t)-\frac{5}{8}D(t)\\
&\quad -\frac{\delta}{2M}|\dot{X}(t)|^2+\dot{X}(t)\sum\limits_{i=3}\limits^9Y_i(t)
+\sum\limits_{i=4}\limits^{10}B_i(t)-G_2(t)-G_3(t)-\sum\limits_{i=5}\limits^8G_i(t)+\frac{5}{8}D(t).
\end{align*}
Using Lemma \ref{le4.6} and Young's inequality, we can obtain that
\begin{align*}
&\frac{\dif}{\dif t}\int_{\mathbb{T}^2}\int_{\mathbb{R}}a^{-X}\rho\left(H(v|(v^s)^{-X})+\frac{|u-(u^s)^{-X}|^2}{2}
\right)\dif \xi_1\dif \xi^{\prime}\\
&\quad+\tau\frac{\dif}{\dif t}\int_{\mathbb{T}^2}\int_{\mathbb{R}}a^{-X}\rho\left(
\frac{|\Pi_1-(\Pi_1^s)^{-X}|^2}{4\mu}+\frac{|\Pi_2-(\Pi_2^s)^{-X}|^2}{2\lambda}\right)\dif \xi_1\dif \xi^{\prime}\\
&=-C_1\underbrace{\int_{\mathbb{T}^2}\int_{\mathbb{R}}(v^s)^{-X}_{\xi_1}|p(v)-p((v^s)^{-X})|^2\dif \xi_1\dif \xi^{\prime}}\limits_{=:G^s(t)}
+C\underbrace{\int_{\mathbb{T}^2}\int_{\mathbb{R}}a^{-X}_{\xi_1}|p(v)-p((v^s)^{-X})|^3\dif \xi_1\dif \xi^{\prime}}\limits_{=:K(t)}
+C\frac{\delta}{\nu}G_3(t)\\
&\quad -\frac{\delta}{4M}|\dot{X}(t)|^2+\frac{C}{\delta}\sum\limits_{i=3}\limits^9|Y_i(t)|^2
+\sum\limits_{i=4}\limits^{10}B_i(t)-G_2(t)-G_3(t)-\sum\limits_{i=5}\limits^8G_i(t)+\frac{5}{8}D(t).
\end{align*}

Next, we control the bad terms $K(t)$, $Y_i(t)(i=3, 4, ..., 9)$ and $B_i(t)(i=4, 5, ..., 10)$.

For $K(t)$, using \eqref{weight1}, \eqref{3.4}, Lemmas \ref{pvsw}, \ref{czbds1}, H\"older inequality and the notation $w=p(v)-p((v^s)^{-X})$, we have
 \begin{align*}
 K(t)&\le C\frac{\nu}{\delta}\|w\|^2_{L^\infty}\int_{\mathbb{T}^2}\int_{\mathbb{R}}(v^s)^{-X}_{\xi_1}|w|\dif \xi_1\dif \xi^{\prime}\\
 &\le C\frac{\nu}{\delta}(\|w\|_{L^2}\|\partial_{\xi_1}w\|_{L^2}
 +\|\nabla_\xi w\|_{L^2}\|\nabla_\xi^2 w\|_{L^2})\left(\int_{\mathbb{R}}(v^s)^{-X}_{\xi_1}|w|^2\dif \xi_1\dif \xi^{\prime}\right)^{\frac{1}{2}}\left(\int_{\mathbb{R}}(v^s)^{-X}_{\xi_1}\dif \xi_1\dif \xi^{\prime}\right)^{\frac{1}{2}}\\
 &\le C\frac{\nu}{\sqrt{\delta}}(\|w\|_{L^2}
 +\|\nabla_\xi^2 w\|_{L^2})\|\nabla_\xi w\|_{L^2}\left(\int_{\mathbb{R}}(v^s)^{-X}_{\xi_1}|w|^2\dif \xi_1\dif \xi^{\prime}\right)^{\frac{1}{2}}\\
 &\le C\varepsilon_1(D(t)+G^s(t)).
 \end{align*}

For $Y_3(t)$, using \eqref{weight1} and H\"older inequality, we can get
 \begin{align*}
Y_3(t)&\le C\frac{\delta}{\nu}\int_{\mathbb{T}^2}\int_{\mathbb{R}}a^{-X}_{\xi_1}
\Big|u_1-(u_1^s)^{-X}-\frac{p(v)-p((v^s)^{-X})}{\sigma_\ast}\Big|\dif \xi_1\dif\xi^{\prime}\\
&\le C\frac{\delta}{\nu}\left(\int_{\mathbb{T}^2}\int_{\mathbb{R}}a^{-X}_{\xi_1}
\Big|u_1-(u_1^s)^{-X}-\frac{p(v)-p((v^s)^{-X})}{\sigma_\ast}\Big|^2\dif \xi_1\dif\xi^{\prime}\right)^{\frac{1}{2}}
\left(\int_{\mathbb{T}^2}\int_{\mathbb{R}}a^{-X}_{\xi_1}\dif \xi_1\dif\xi^{\prime}\right)^{\frac{1}{2}}\\
&\le C\frac{\delta}{\sqrt{\nu}}\sqrt{G_3(t)}.
\end{align*}
Similarly, for $Y_4(t)$, by using \eqref{3.4} and Sobolve's embedding theorem, we can obtain that
\[
Y_4(t)\le C\varepsilon_1\int_{\mathbb{T}^2}\int_{\mathbb{R}}a^{-X}_{\xi_1}
\Big|u_1-(u_1^s)^{-X}-\frac{p(v)-p((v^s)^{-X})}{\sigma_\ast}\Big|\dif \xi_1\dif\xi^{\prime}
\le C\varepsilon_1\sqrt{\nu}\sqrt{G_3(t)}.
\]
For $Y_5(t)$, using \eqref{weight1}, \eqref{3.4}, Young's inequality and Lemma \ref{pvsw}, we can get
\begin{align*}
\frac{C}{\delta}|Y_5(t)|^2&\le \frac{C\nu^2}{\delta^3}\left(\int_{\mathbb{T}^2}\int_{\mathbb{R}}(v^s)^{-X}_{\xi_1}|p(v)-p((v^s)^{-X})|^2\dif \xi_1\dif \xi^{\prime}\right)^2
+\frac{C}{\delta}
\left(\int_{\mathbb{T}^2}\int_{\mathbb{R}}a^{-X}_{\xi_1}
\frac{u_2^2+u_3^2}{2}\dif \xi_1\dif\xi^{\prime}\right)^2\\
&\le \frac{C\nu^2}{\delta}
\int_{\mathbb{T}^2}\int_{\mathbb{R}}|p(v)-p((v^s)^{-X})|^2\dif \xi_1\dif \xi^{\prime}
+\int_{\mathbb{T}^2}\int_{\mathbb{R}}(v^s)^{-X}_{\xi_1}|p(v)-p((v^s)^{-X})|^2\dif \xi_1\dif \xi^{\prime}
\\
&\quad+\frac{C}{\delta}
\int_{\mathbb{T}^2}\int_{\mathbb{R}}\frac{\nu}{\delta}(v^s)^{-X}_{\xi_1}
\frac{(u_1-(u_1^s)^{-X})^2+u_2^2+u_3^2}{2}\dif \xi_1\dif\xi^{\prime}
+\int_{\mathbb{T}^2}\int_{\mathbb{R}}a^{-X}_{\xi_1}
\frac{u_2^2+u_3^2}{2}\dif \xi_1\dif\xi^{\prime}\\
&\le C\varepsilon_1^2G^s(t)+C\nu\varepsilon_1^2G_2(t).
\end{align*}
For $Y_6(t)$ and $Y_7(t)$, using \eqref{weight1}, \eqref{3.4}, Sobolev's embedding theorem and H\"older inequality, we imply that
\begin{align*}
Y_6(t)&\le C\|\sqrt{\tau}(\Pi_1-(\Pi_1^s)^{-X})\|_{L^\infty}
\left(\int_{\mathbb{T}^2}\int_{\mathbb{R}}a^{-X}_{\xi_1}\dif \xi_1\dif\xi^{\prime}\right)^{\frac{1}{2}}
\left(\int_{\mathbb{T}^2}\int_{\mathbb{R}}a^{-X}_{\xi_1}
\frac{\tau|\Pi_1-(\Pi_1^s)^{-X}|^2}{2\mu}\dif \xi_1\dif\xi^{\prime}\right)^{\frac{1}{2}}\\
&\le C\|\sqrt{\tau}(\Pi_1-(\Pi_1^s)^{-X})\|_{H^2}\sqrt{\nu}\sqrt{G_5(t)}\\
&\le C\varepsilon_1\sqrt{\nu}\sqrt{G_5(t)},
\end{align*}
and
\[
Y_7(t)\le C\varepsilon_1\sqrt{\nu}\sqrt{G_7(t)}.
\]
For $Y_8(t)$ and $Y_9(t)$, using H\"older inequality and Lemma \ref{pvsw}, we can deduce that
\begin{align*}
Y_8(t)&\le C\left(\int_{\mathbb{T}^2}\int_{\mathbb{R}}a^{-X} \frac{|\Pi_1-(\Pi_1^s)^{-X}|^2}{\mu}\dif \xi_1\dif \xi^{\prime}\right)^{\frac{1}{2}}
\left(\int_{\mathbb{T}^2}\int_{\mathbb{R}}|\left((\Pi_1^s)^{-X}\right)_{\xi_1}|^2\dif \xi_1\dif \xi^{\prime}\right)^{\frac{1}{2}}\\
&\le C\delta^2\sqrt{G_6(t)},
\end{align*}
and
\[
Y_9(t)\le C\delta^2\sqrt{G_8(t)}.
\]

For $B_4(t)$, using \eqref{weight1}, Young's inequality and Lemma \ref{pvsw}, we have
\begin{align*}
B_4(t)&\le C\|a^{-X}_{\xi_1}\|^{\frac{1}{2}}
\left(\int_{\mathbb{T}^2}\int_{\mathbb{R}}a^{-X}_{\xi_1}|u_1-(u_1^s)^{-X}|^2\dif \xi_1\dif \xi^{\prime}+
\int_{\mathbb{T}^2}\int_{\mathbb{R}}a^{-X}\frac{|\Pi_2-(\Pi_2^s)^{-X}|^2}{\lambda}\dif \xi_1\dif \xi^{\prime}\right)\\
&\le C\sqrt{\nu\delta}\int_{\mathbb{T}^2}\int_{\mathbb{R}}a^{-X}_{\xi_1}\Big|u_1-(u_1^s)^{-X}
-\frac{p(v)-p((v^s)^{-X})}{\sigma_\ast}\Big|^2\dif \xi_1\dif \xi^{\prime}\\
&\quad+
C\sqrt{\nu\delta}\int_{\mathbb{T}^2}\int_{\mathbb{R}}a^{-X}_{\xi_1}\Big|\frac{p(v)-p((v^s)^{-X})}{\sigma_\ast}\Big|^2\dif \xi_1\dif \xi^{\prime}
 +C\sqrt{\nu\delta}G_8(t)\\
&\le C\sqrt{\nu\delta}(G_3(t)+G_8(t))+C\sqrt{\nu}G^s(t).
\end{align*}
In a similar way, for $B_5(t)$, we can get
\begin{align*}
B_5(t)&\le C\|a^{-X}_{\xi_1}\|^{\frac{1}{2}}\left(\int_{\mathbb{T}^2}\int_{\mathbb{R}}a^{-X}_{\xi_1}|u_1-(u_1^s)^{-X}|^2\dif \xi_1\dif \xi^{\prime}+
\int_{\mathbb{T}^2}\int_{\mathbb{R}}a^{-X}\frac{|\Pi_{11}-(\Pi_{11}^s)^{-X}|^2}{\mu}\dif \xi_1\dif \xi^{\prime}\right)\\
&\quad+C\|a^{-X}_{\xi_1}\|^{\frac{1}{2}}\left(\int_{\mathbb{T}^2}\int_{\mathbb{R}}a^{-X}_{\xi_1}
\frac{u_2^2+u_3^2}{2}\dif \xi_1\dif \xi^{\prime}+
\int_{\mathbb{T}^2}\int_{\mathbb{R}}a^{-X}\frac{|\Pi_{21}|^2+|\Pi_{31}|^2}{\mu}\dif \xi_1\dif \xi^{\prime}\right)\\
&\le C\sqrt{\nu\delta}(G_3(t)+G_6(t))+C\sqrt{\nu}G^s(t).
\end{align*}
For $B_6(t)$, using \eqref{weight1}, Lemma \ref{pvsw} and Young's inequality, we obtain that
\begin{align*}
B_6(t)&\le C\delta\int_{\mathbb{T}^2}\int_{\mathbb{R}}(v^s)^{-X}_{\xi_1}|p(v)-p((v^s)^{-X})||u_1-(u_1^s)^{-X}|\dif \xi_1\dif \xi^{\prime}\\
&\le C\delta\int_{\mathbb{T}^2}\int_{\mathbb{R}}(v^s)^{-X}_{\xi_1}|p(v)-p((v^s)^{-X})|^2\dif \xi_1\dif \xi^{\prime}\\
&\quad+C\delta\int_{\mathbb{T}^2}\int_{\mathbb{R}}(v^s)^{-X}_{\xi_1}\Big|u_1-(u_1^s)^{-X}
-\frac{p(v)-p((v^s)^{-X})}{\sigma_\ast}\Big|^2\dif \xi_1\dif \xi^{\prime}\\
&\le C\delta G^s(t)+C\frac{\delta^2}{\nu}G_3(t).
\end{align*}
In a similar way, for $B_7(t)$, we can deduce that
\begin{align*}
B_7(t)
&\le C\delta\int_{\mathbb{T}^2}\int_{\mathbb{R}}(v^s)^{-X}_{\xi_1}|p(v)-p((v^s)^{-X})|^2\dif \xi_1\dif \xi^{\prime}\\
&\quad+C\delta\int_{\mathbb{T}^2}\int_{\mathbb{R}}(v^s)^{-X}_{\xi_1}\Big|u_1-(u_1^s)^{-X}
-\frac{p(v)-p((v^s)^{-X})}{\sigma_\ast}\Big|^2\dif \xi_1\dif \xi^{\prime}\\
&\le C\delta G^s(t)+C\frac{\delta^2}{\nu}G_3(t).
\end{align*}
For $B_8(t)$, using the definition of $\eta(t, \xi)$ and Lemma 2.1 in \cite{WY}, we can derive that
\begin{align*}
B_8(t)
&= \int_{\mathbb{T}^2}\int_{\mathbb{R}}a^{-X}_{\xi_1}F\left(H(v|(v^s)^{-X})+\frac{|u-(u^s)^{-X}|^2}{2}
\right)\dif \xi_1\dif \xi^{\prime}\\
&\quad+\int_{\mathbb{T}^2}\int_{\mathbb{R}}a^{-X}_{\xi_1}F\left(
\frac{\tau|\Pi_1-(\Pi_1^s)^{-X}|^2}{4\mu}+\frac{\tau|\Pi_2-(\Pi_2^s)^{-X}|^2}{2\lambda}\right)\dif \xi_1\dif \xi^{\prime}\\
&\le C\int_{\mathbb{T}^2}\int_{\mathbb{R}}a^{-X}_{\xi_1}|F|\left(|p(v)-p((v^s)^{-X})|^2+
\Big|u-(u^s)^{-X}-\frac{p(v)-p((v^s)^{-X})}{\sigma_\ast}\Big|^2+\frac{u_2^2+u_3^2}{2}
\right)\dif \xi_1\dif \xi^{\prime}\\
&\quad+\int_{\mathbb{T}^2}\int_{\mathbb{R}}a^{-X}_{\xi_1}|F|\left(
\frac{\tau|\Pi_1-(\Pi_1^s)^{-X}|^2}{4\mu}+\frac{\tau|\Pi_2-(\Pi_2^s)^{-X}|^2}{2\lambda}\right)\dif \xi_1\dif \xi^{\prime}.
\end{align*}
On the one hand, using the definition of $F$ and Sobolev's embedding theorem, we know that
\begin{align*}
\|F\|_{L^\infty}
\le C\left(\|v-(v^s)^{-X}\|_{H^2}+\|u_1-(u_1^s)^{-X}\|_{H^2}\right)\le C\varepsilon_1,
\end{align*}
then, we have
\begin{align*}
&\int_{\mathbb{T}^2}\int_{\mathbb{R}}a^{-X}_{\xi_1}|F|\left(
\Big|u-(u^s)^{-X}-\frac{p(v)-p((v^s)^{-X})}{\sigma_\ast}\Big|^2+\frac{u_2^2+u_3^2}{2}
\right)\dif \xi_1\dif \xi^{\prime}\\
&\quad+\int_{\mathbb{T}^2}\int_{\mathbb{R}}a^{-X}_{\xi_1}|F|\left(
\frac{\tau|\Pi_1-(\Pi_1^s)^{-X}|^2}{2\mu}+\frac{\tau|\Pi_2-(\Pi_2^s)^{-X}|^2}{2\lambda}\right)\dif \xi_1\dif \xi^{\prime}\\
&\le C\varepsilon_1\left(G_2(t)+G_3(t)+G_5(t)+G_7(t)\right).
\end{align*}
On the other hand,
\begin{align*}
&\int_{\mathbb{T}^2}\int_{\mathbb{R}}a^{-X}_{\xi_1}|F||p(v)-p((v^s)^{-X})|^2\dif \xi_1\dif \xi^{\prime}\\
&\le C\int_{\mathbb{T}^2}\int_{\mathbb{R}}a^{-X}_{\xi_1}|p(v)-p((v^s)^{-X})|^3\dif \xi_1\dif \xi^{\prime}
+C\int_{\mathbb{T}^2}\int_{\mathbb{R}}a^{-X}_{\xi_1}|u_1-(u_1^s)^{-X}||p(v)-p((v^s)^{-X})|^2\dif \xi_1\dif \xi^{\prime}\\
&\le
C\underbrace{\int_{\mathbb{T}^2}\int_{\mathbb{R}}a^{-X}_{\xi_1} \Big|u-(u^s)^{-X}-\frac{p(v)-p((v^s)^{-X})}{\sigma_\ast}\Big||p(v)-p((v^s)^{-X})|^2\dif \xi_1\dif \xi^{\prime}}\limits_{=:K_1}\\
&\quad+C\underbrace{\int_{\mathbb{T}^2}\int_{\mathbb{R}}a^{-X}_{\xi_1}|p(v)-p((v^s)^{-X})|^3\dif \xi_1\dif \xi^{\prime}}\limits_{=:K_2}.
\end{align*}
Following an analogous method to that used for $K(t)$, we derive that
\begin{align*}
K_1&\le C\|p(v)-p((v^s)^{-X})\|^2_{L^\infty}\sqrt{G_3(t)}\left(\int_{\mathbb{T}^2}\int_{\mathbb{R}}\frac{\nu}{\delta}
(v^s)^{-X}_{\xi_1} \dif \xi_1\dif \xi^{\prime}\right)^{\frac{1}{2}}\\
&\le C(\|p(v)-p((v^s)^{-X})\|_{L^2}
 +\|\nabla_\xi^2 (p(v)-p((v^s)^{-X}))\|_{L^2})\|\nabla_\xi (p(v)-p((v^s)^{-X}))\|_{L^2}\sqrt{G_3(t)}\sqrt{\nu}\\
 &\le C\varepsilon_1\sqrt{\nu}(D(t)+G_3(t)).
\end{align*}
and
\[
K_2\le C\varepsilon_1(D(t)+G^s(t)).
\]
Therefore, we conclude that
\[
B_8(t)\le C\varepsilon_1\left(G_2(t)+G_3(t)+G_5(t)+G_7(t)+D(t)+G^s(t)\right).
\]
For $B_9(t)$ and $B_{10}(t)$, using \eqref{weight1}, Lemma \ref{pvsw} and Young's inequality, we imply that
\begin{align*}
B_9(t)&\le C\delta\int_{\mathbb{T}^2}\int_{\mathbb{R}}(v^s)^{-X}_{\xi_1}\left(|v-(v^s)^{-X}|+|u_1-(u_1^s)^{-X}|\right)
|\Pi_1-(\Pi_1^s)^{-X}|\dif \xi_1\dif \xi^{\prime}\\
&\le
C\delta\int_{\mathbb{T}^2}\int_{\mathbb{R}}(v^s)^{-X}_{\xi_1}|\Big|u-(u^s)^{-X}-\frac{p(v)-p((v^s)^{-X})}{\sigma_\ast}\Big|
|\Pi_1-(\Pi_1^s)^{-X}|\dif \xi_1\dif \xi^{\prime}\\
&\quad +C\delta\int_{\mathbb{T}^2}\int_{\mathbb{R}}(v^s)^{-X}_{\xi_1}|p(v)-p((v^s)^{-X})|
|\Pi_1-(\Pi_1^s)^{-X}|\dif \xi_1\dif \xi^{\prime}\\
&\le C\delta\int_{\mathbb{T}^2}\int_{\mathbb{R}}(v^s)^{-X}_{\xi_1}|\Big|u-(u^s)^{-X}-\frac{p(v)-p((v^s)^{-X})}{\sigma_\ast}\Big|^2
\dif \xi_1\dif \xi^{\prime}
+C\delta\int_{\mathbb{T}^2}\int_{\mathbb{R}}(v^s)^{-X}_{\xi_1}|p(v)-p((v^s)^{-X})|^2
\dif \xi_1\dif \xi^{\prime}\\
&\quad+C\delta\int_{\mathbb{T}^2}\int_{\mathbb{R}}a^{-X}(v^s)^{-X}_{\xi_1}
\frac{|\Pi_1-(\Pi_1^s)^{-X}|^2}{\mu}\dif \xi_1\dif \xi^{\prime}\\
&\le C\frac{\delta^2}{\nu}G_3(t)+C\delta(G_6(t)+G^s(t)).
\end{align*}
Similarly, for $B_{10}(t)$, we have
\[
B_{10}(t)\le C\frac{\delta^2}{\nu}G_3(t)+C\delta(G_8(t)+G^s(t)).
\]

For $D(t)$, by using Lemma \ref{pvsw}, we get
\begin{align*}
D(t)&= (\frac{4\mu}{3}+\lambda)\int_{\mathbb{T}^2}\int_{\mathbb{R}}
\frac{a^{-X}}{\gamma p^{1+\frac{1}{\gamma}}(v)}\left(|p^\prime(v)||\nabla(v-(v^s)^{-X})|
+|p^\prime(v)-p^\prime((v^s)^{-X})||\nabla((v^s)^{-X})|\right)^2
\dif \xi_1\dif \xi^{\prime}\\
&\le (\frac{4\mu}{3}+\lambda)(1+C\nu)(1+\chi)\int_{\mathbb{T}^2}\int_{\mathbb{R}}
|p^\prime (v)||\nabla(v-(v^s)^{-X})|^2
\dif \xi_1\dif \xi^{\prime}
+C\delta^2G^s(t).
\end{align*}

By combining the above estimates and applying Lemma 2.1 in \cite{WY}, the proof of this lemma is completed.
\end{proof}

\subsection{High-order energy estimates}

In this section, we will show the high-order energy estimates for the system \eqref{sys1}. First of all, we have the following lemma.
\begin{lemma}\label{legj}
Under the hypotheses of Proposition \ref{p1}, there exists constant $C>0$ independent of $\tau, \nu, \delta, \varepsilon_1, T$, such that for $t\in[0, T]$, we have
\begin{equation}\label{gjgj1.1}
\begin{aligned}
&\|\nabla_\xi(v-(v^s)^{-X})\|_{H^2}^2+\|\nabla_\xi(u-(u^s)^{-X})\|_{H^2}^2
+\tau\|\nabla_\xi(\Pi_1-(\Pi_1^s)^{-X})\|_{H^2}^2
+\tau\|\nabla_\xi(\Pi_2-(\Pi_2^s)^{-X})\|_{H^2}^2\\
&\quad
+\int_0^t\|(\nabla_\xi(\Pi_1-(\Pi_1^s)^{-X}), \nabla_\xi(\Pi_2-(\Pi_2^s)^{-X}))\|_{H^2}\dif t\\
&\le C\Big(\|\nabla_\xi(v_0-v^s)\|_{H^2}^2+\|\nabla_\xi(u_0-u^s)\|_{H^2}^2
+\tau\|\nabla_\xi(\Pi_{10}-\Pi_1^s)\|_{H^2}^2
+\tau\|\nabla_\xi(\Pi_{20}-\Pi_2^s)\|_{H^2}^2\Big)\\
&+C(\varepsilon_1+\delta)\int_0^t
\left(\|\left(\nabla_\xi(v-(v^s)^{-X}), \nabla_\xi(u-(u^s)^{-X})\right)\|_{H^2}^2
+\|(\Pi_1-(\Pi_1^s)^{-X}, \Pi_2-(\Pi_2^s)^{-X})\|_{L^2}
\right)
\dif t\\
&
+C\delta^2\int_0^t|\dot X(t)|^2\dif t
+C(\varepsilon_1+\delta)\int_0^t
G^s(t)
\dif t
+C\delta\int_0^t
G_3(t)
\dif t,
\end{aligned}
\end{equation}
where $G_3(t)$ and $G^s(t)$ are defined in Lemma \ref{le4.5} and Lemma \ref{le0}, respectively.
\end{lemma}
\begin{proof}
To simplify the calculations, let $\Phi=v-(v^s)^{-X}, \Psi=u-(u^s)^{-X}, Q_1=\Pi_1-(\Pi_1^s)^{-X}, Q_2=\Pi_2-(\Pi_2^s)^{-X}$ and applying $\partial_\xi^\alpha(|\alpha|=1, 2, 3)$ to the system \eqref{sys1}, we derive that
\begin{equation}\label{sys3}
\begin{cases}
&\partial_t\partial_\xi^\alpha\Phi-\sigma\partial_{\xi_1}\partial_\xi^\alpha\Phi
        + \partial_\xi^\alpha\left(u\cdot\nabla_\xi\Phi\right)
        -\dot{X}(t)\partial_{\xi_1}\partial_\xi^\alpha(v^s)^{-X}
 +\partial_\xi^\alpha\left(vF\partial_{\xi_1}(v^s)^{-X}\right)
 =\partial_\xi^\alpha\left(v\ddiv_\xi\Psi\right),\\
&\partial_t\partial_\xi^\alpha\Psi-\sigma\partial_{\xi_1}\partial_\xi^\alpha\Psi
        +\partial_\xi^\alpha\left( u\cdot\nabla_\xi\Psi\right)
        -\dot{X}(t)\partial_{\xi_1}\partial_\xi^\alpha(u^s)^{-X}
        +\partial_\xi^\alpha\left(vF\partial_{\xi_1}(u^s)^{-X}\right)\\
&\qquad +\partial_\xi^\alpha\left(v\nabla_\xi\left(p(v)-p((v^s)^{-X})\right)\right)
               =\partial_\xi^\alpha\left(v\ddiv_\xi Q_1\right)
               +\partial_\xi^\alpha\left(v\nabla_\xi Q_2\right),\\
&\tau\partial_t\partial_\xi^\alpha Q_1
-\tau\sigma\partial_{\xi_1}\partial_\xi^\alpha Q_1
+\tau\partial_\xi^\alpha\left( u\cdot\nabla_\xi Q_1\right)
-\tau\dot{X}(t)\partial_{\xi_1}\partial_\xi^\alpha(\Pi_1^s)^{-X}
+\tau\partial_\xi^\alpha\left(v F\partial_{\xi_1}(\Pi_1^s)^{-X}\right)
\\
 & \qquad+\partial_\xi^\alpha\left(vQ_1\right)
 =\mu\partial_\xi^\alpha\left(v\Big(\nabla_\xi\Psi+\left(\nabla_\xi\Psi\right)^T
   -\frac{2}{3}\ddiv_\xi\Psi \mathrm I_3\Big)\right),\\
&\tau\partial_t\partial_\xi^\alpha Q_2
-\tau\sigma\partial_{\xi_1}\partial_\xi^\alpha Q_2
+\tau\partial_\xi^\alpha\left( u\cdot\nabla_\xi Q_2\right)
-\tau\dot{X}(t)\partial_{\xi_1}\partial_\xi^\alpha(\Pi_2^s)^{-X}
+\tau \partial_\xi^\alpha\left(vF\partial_{\xi_1}(\Pi_2^s)^{-X}\right)
\\
&\qquad+\partial_\xi^\alpha\left(vQ_2\right)
=\lambda\partial_\xi^\alpha\left(v\ddiv_\xi \Psi\right).
\end{cases}
\end{equation}

Multiplying $\eqref{sys3}_1$ by $-p^\prime(v)\partial_\xi^\alpha\Phi$ and integrating over $\mathbb{R}\times\mathbb{T}^2$, we have
\begin{align*}
\frac{\dif}{\dif t}&\int_{\mathbb{T}^2}\int_{\mathbb{R}}
\frac{-p^\prime(v)}{2}|\partial_\xi^\alpha\Phi|^2\dif \xi_1\dif \xi^{\prime}\\
&=\int_{\mathbb{T}^2}\int_{\mathbb{R}}\left(\frac{-p^{\prime\prime}(v)}{2}(v_t-\sigma v_{\xi_1})|\partial_\xi^\alpha\Phi|^2
\right)\dif \xi_1\dif \xi^{\prime}
+\int_{\mathbb{T}^2}\int_{\mathbb{R}}p^\prime(v)\partial_\xi^\alpha\left(u\cdot\nabla_\xi\Phi\right)\partial_\xi^\alpha\Phi\dif \xi_1\dif \xi^{\prime}\\
&\quad-\dot{X}(t)\int_{\mathbb{T}^2}\int_{\mathbb{R}}p^\prime(v)\partial_{\xi_1}\partial_\xi^\alpha(v^s)^{-X}\partial_\xi^\alpha\Phi\dif \xi_1\dif \xi^{\prime}
+\int_{\mathbb{T}^2}\int_{\mathbb{R}}p^\prime(v)\partial_\xi^\alpha\left(vF\partial_{\xi_1}(v^s)^{-X}\right)\partial_\xi^\alpha\Phi\dif \xi_1\dif \xi^{\prime}\\
&\quad-\int_{\mathbb{T}^2}\int_{\mathbb{R}}p^\prime(v)\partial_\xi^\alpha\left(v\ddiv_\xi\Psi\right)\partial_\xi^\alpha\Phi\dif \xi_1\dif \xi^{\prime}
=:\sum\limits_{i=1}\limits^5R_{1i}.
\end{align*}
Firstly, the definition of $\Phi, \Psi$,  Lemma \ref{pvsw} and Sobolev's embedding theorem yields $\|u\|_{L^\infty} \le C$ and $\|(\partial_{\xi_i}v, \partial_{\xi_i}u)\|_{L^\infty} \le C(\varepsilon_1+\delta)(i=1,2,3)$. Then,
for $R_{11}$, by using $\eqref{3.1}_1$, we get
\begin{equation}\label{r11}
\begin{aligned}
R_{11}&=\int_{\mathbb{T}^2}\int_{\mathbb{R}}
\left(\frac{-p^{\prime\prime}(v)}{2}(-u\cdot\nabla_\xi v+v\ddiv_\xi u)|\partial_\xi^\alpha\Phi|^2\right)
\dif \xi_1\dif \xi^{\prime}\\
&\le C(\|u\|_{L^\infty} \|\nabla_{\xi}v\|_{L^\infty} +\|\ddiv_{\xi}u\|_{L^\infty} )\int_{\mathbb{T}^2}\int_{\mathbb{R}}|\partial_\xi^\alpha\Phi|^2\dif \xi_1\dif \xi^{\prime}\\
&\le C(\delta+\varepsilon_1)\int_{\mathbb{T}^2}\int_{\mathbb{R}}|\partial_\xi^\alpha\Phi|^2\dif \xi_1\dif \xi^{\prime}.
\end{aligned}
\end{equation}

For $R_{12}$, we establish the result for $|\alpha|=3$, the cases $|\alpha|=1$ and $|\alpha|=2$ can be handled similarly. Firstly, we know that $\|(\partial^\alpha_\xi v, \partial^\alpha_\xi u)\|_{L^2}\le C(\delta+\varepsilon_1)$. Subsequently,
using \eqref{3.4}, Lemma \ref{pvsw}, H\"older inequality, Young's inequality and Sobolev's embedding theorem, we have
\begin{align*}
\int_{\mathbb{T}^2}\int_{\mathbb{R}}p^\prime(v)
\left( u\cdot\nabla_\xi\partial_{\xi}^\alpha\Phi\right)\partial_{\xi}^\alpha\Phi\dif \xi_1\dif \xi^{\prime}
&= \int_{\mathbb{T}^2}\int_{\mathbb{R}}
\left(\nabla_\xi p^\prime(v) u+p^\prime(v)\ddiv_\xi u\right)\frac{|\partial_{\xi}^\alpha\Phi|^2}{2}\dif \xi_1\dif \xi^{\prime}\\
&\le C(\varepsilon_1+\delta)\int_{\mathbb{R}}|\partial_\xi^\alpha\Phi|^2\dif \xi_1\dif \xi^{\prime},
\end{align*}
\begin{align*}
\int_{\mathbb{T}^2}\int_{\mathbb{R}}p^\prime(v)
\left(\partial_{\xi_i} u\cdot\nabla_\xi\partial_{\xi}^{\alpha-1}\Phi\right)\partial_{\xi}^\alpha\Phi\dif \xi_1\dif \xi^{\prime}
\le C(\varepsilon_1+\delta)\int_{\mathbb{R}}|\partial_\xi^\alpha\Phi|^2\dif \xi_1\dif \xi^{\prime},
\end{align*}
\begin{align*}
\int_{\mathbb{T}^2}\int_{\mathbb{R}}p^\prime(v)
\left(\partial_{\xi}^{\alpha-1} u\cdot\nabla_\xi\partial_{\xi_i} \Phi\right)\partial_{\xi}^\alpha\Phi\dif \xi_1\dif \xi^{\prime}
&\le C\|\partial_{\xi}^{\alpha-1} u\|_{L^3}\|\nabla_\xi\partial_{\xi_i} \Phi\|_{L^6}\|\partial_{\xi}^\alpha\Phi\|_{L^2}\\
&\le C\|\partial_{\xi}^{\alpha-1} u\|_{H^1}\|\nabla_\xi\partial_{\xi_i} \Phi\|_{H^1}\|\partial_{\xi}^\alpha\Phi\|_{L^2}\\
&\le C(\varepsilon_1+\delta)\|\partial_{\xi}^\alpha\Phi\|_{L^2},
\end{align*}
and
\begin{align*}
\int_{\mathbb{T}^2}\int_{\mathbb{R}}p^\prime(v)
\left(\partial_{\xi}^{\alpha} u\cdot\nabla_\xi \Phi\right)\partial_{\xi}^\alpha\Phi\dif \xi_1\dif \xi^{\prime}
&\le C\|\partial_{\xi}^{\alpha} u\|_{L^2}\|\nabla_\xi\Phi\|_{L^\infty}\|\partial_{\xi}^\alpha\Phi\|_{L^2}\\
&\le C\|\partial_{\xi}^{\alpha} u\|_{L^2}\|\nabla_\xi\Phi\|_{H^2}\|\partial_{\xi}^\alpha\Phi\|_{L^2}\\
&\le C(\varepsilon_1+\delta)\|\nabla_\xi\Phi\|_{H^2},
\end{align*}
Thus, we conclude that
\begin{equation}\label{r12}
R_{12}\le
C(\varepsilon_1+\delta)
\|\nabla_\xi\Phi\|_{H^{|\alpha|-1}}^2
.
\end{equation}
For $R_{13}$, using Lemma \ref{pvsw} and Young's inequality, we get
\begin{equation}\label{r13}
\begin{aligned}
R_{13}&\le C\delta|\dot{X}(t)|^2\int_{\mathbb{T}^2}\int_{\mathbb{R}}|\partial_{\xi_1}\partial_\xi^\alpha(v^s)^{-X}|\dif \xi_1\dif \xi^{\prime}
+\frac{C}{\delta}\int_{\mathbb{T}^2}\int_{\mathbb{R}}|\partial_{\xi_1}\partial_\xi^\alpha(v^s)^{-X}||\partial_\xi^\alpha\Phi|^2\dif \xi_1\dif \xi^{\prime}\\
&\le C\delta^2|\dot{X}(t)|^2+C\delta\int_{\mathbb{T}^2}\int_{\mathbb{R}}|\partial_\xi^\alpha\Phi|^2\dif \xi_1\dif \xi^{\prime}.
\end{aligned}
\end{equation}
For $R_{14}$, first of all, by using \eqref{wcxf}, we know that
\[
vF=\sigma_\ast(v-(v^s)^{-X})+u_1-(u_1^s)^{-X},
\]
then, using Lemma \ref{pvsw} and Young's inequality, we have
\begin{align*}
&\int_{\mathbb{T}^2}\int_{\mathbb{R}}p^\prime(v)vF\left(\partial_\xi^\alpha\partial_{\xi_1}(v^s)^{-X}\right)\partial_\xi^\alpha\Phi\dif \xi_1\dif \xi^{\prime}\\
&\le C\delta\int_{\mathbb{T}^2}\int_{\mathbb{R}}(v^s)^{-X}_{\xi_1}
\left(\Big|u_1-(u_1^s)^{-X}-\frac{p(v)-p((v^s)^{-X})}{\sigma_\ast}\Big|
+\Big|\frac{p(v)-p((v^s)^{-X})}{\sigma_\ast}\Big|\right)
|\partial_\xi^\alpha\Phi|\dif \xi_1\dif \xi^{\prime}
\\
&\le C\frac{\delta^{\frac{3}{2}}}{\sqrt{\nu}}
\int_{\mathbb{T}^2}\int_{\mathbb{R}}\left(a^{-X}_{\xi_1}\right)^{\frac{1}{2}}
|u_1-(u_1^s)^{-X}-\frac{p(v)-p((v^s)^{-X})}{\sigma_\ast}|
|\partial_\xi^\alpha\Phi|\dif \xi_1\dif \xi^{\prime}\\
&\quad+C\delta\int_{\mathbb{T}^2}\int_{\mathbb{R}}(v^s)^{-X}_{\xi_1}
p(v)-p((v^s)^{-X})||\partial_\xi^\alpha\Phi|\dif \xi_1\dif \xi^{\prime}\\
&\le C\delta\left(G_3(t)+G^s(t)
+
\|\partial_\xi^\alpha\Phi\|^2_{L^2}\right).
\end{align*}
and
\begin{align*}
&\int_{\mathbb{T}^2}\int_{\mathbb{R}}p^\prime(v)\Big(\partial_\xi^\alpha\left(vF\partial_{\xi_1}(v^s)^{-X}\right)-vF\left(\partial_\xi^\alpha\partial_{\xi_1}(v^s)^{-X}\right)\Big)\partial_\xi^\alpha\Phi\dif \xi_1\dif \xi^{\prime}\\
&\le C\delta\left(\|\nabla_\xi\Psi\|_{H^{|\alpha|-1}}^2+\|\nabla_\xi\Phi\|_{H^{|\alpha|-1}}^2\right).
\end{align*}
Thus, we obtain that
\begin{equation}\label{r14}
R_{14}\le C\delta\left(\|\nabla_\xi\Psi\|_{H^{|\alpha|-1}}^2+\|\nabla_\xi\Phi\|_{H^{|\alpha|-1}}^2+G_3(t)+G^s(t)\right).
\end{equation}
For $R_{15}$, when $|\alpha|=3$, a direct calculation yields
\begin{align*}
&-\int_{\mathbb{T}^2}\int_{\mathbb{R}}p^\prime(v)\partial_{\xi_i}v\partial_\xi^{\alpha-1}\ddiv_\xi\Psi\partial_\xi^\alpha\Phi\dif \xi_1\dif \xi^{\prime}
\le C(\delta+\varepsilon_1)(\|\partial_\xi^{\alpha-1}\ddiv_\xi\Psi\|_{L^2}^2+\|\partial_\xi^\alpha\Phi\|_{L^2}^2),
\end{align*}
\begin{align*}
-\int_{\mathbb{T}^2}\int_{\mathbb{R}}p^\prime(v)\partial_\xi^{\alpha-1}v\partial_{\xi_i}\ddiv_\xi\Psi\partial_\xi^\alpha\Phi\dif \xi_1\dif \xi^{\prime}
&\le C \|\partial_\xi^{\alpha-1}v\|_{L^3}\|\partial_{\xi_i}\ddiv_\xi\Psi\|_{L^6}
\|\partial_\xi^\alpha\Phi\|_{L^2}\\
&\le C\|\partial_\xi^{\alpha-1}v\|_{H^1}\|\partial_{\xi_i}\ddiv_\xi\Psi\|_{H^1}
\|\partial_\xi^\alpha\Phi\|_{L^2}\\
&\le C(\delta+\varepsilon_1)(\|\partial_{\xi_i}\ddiv_\xi\Psi\|_{H^1}^2+
\|\partial_\xi^\alpha\Phi\|_{L^2}^2),
\end{align*}
and
\begin{align*}
-\int_{\mathbb{T}^2}\int_{\mathbb{R}}p^\prime(v)\partial_\xi^\alpha v\ddiv_\xi\Psi\partial_\xi^\alpha\Phi\dif \xi_1\dif \xi^{\prime}&\le C\|\partial_\xi^\alpha v\|_{L^2}\|\ddiv_\xi\Psi\|_{L^\infty}\|\partial_\xi^\alpha\Phi\|_{L^2}\\
&\le C\|\partial_\xi^\alpha v\|_{L^2}\|\ddiv_\xi\Psi\|_{H^2}\|\partial_\xi^\alpha\Phi\|_{L^2}\\
&\le C(\delta+\varepsilon_1)(\|\ddiv_\xi\Psi\|_{H^2}^2+
\|\partial_\xi^\alpha\Phi\|_{L^2}^2).
\end{align*}
Additionally, the cases $|\alpha|=1$ and $|\alpha|=2$ are analogous. Thus, we get
\begin{equation}\label{r15}
R_{15}\le -\int_{\mathbb{T}^2}\int_{\mathbb{R}}p^\prime(v)
v\partial_\xi^\alpha\ddiv_\xi\Psi\partial_\xi^\alpha\Phi\dif \xi_1\dif \xi^{\prime}
+C(\varepsilon_1+\delta)\left(
\|\partial_\xi^\alpha\Phi\|^2_{L^2}
+
\|\nabla_\xi\Psi\|_{H^{|\alpha|-1}}^2\right).
\end{equation}
Therefore, combining \eqref{r11}, \eqref{r12}, \eqref{r13}, \eqref{r14} and \eqref{r15}, we conclude that
\begin{equation}\label{hg1.1}
\begin{aligned}
\frac{\dif}{\dif t}&\int_{\mathbb{T}^2}\int_{\mathbb{R}}
\frac{-p^\prime(v)}{2}|\partial_\xi^\alpha\Phi|^2\dif \xi_1\dif \xi^{\prime}\\
&\le C(\varepsilon_1+\delta)\left(
\|\nabla_\xi\Phi\|_{H^{|\alpha|-1}}^2
+
\|\nabla_\xi\Psi\|_{H^{|\alpha|-1}}^2\right)+C\delta^2|\dot{X}(t)|^2
+C\delta(G_3(t)+G^s(t))\\
&\qquad\underbrace{-\int_{\mathbb{T}^2}\int_{\mathbb{R}}p^\prime(v)
v\partial_\xi^\alpha\ddiv_\xi\Psi\partial_\xi^\alpha\Phi\dif \xi_1\dif \xi^{\prime}}\limits_{=:R_{1\ast}}.
\end{aligned}
\end{equation}

Multiplying $\eqref{sys3}_2$ by $\partial_\xi^\alpha\Psi$ and integrating over $\mathbb{R}\times\mathbb{T}^2$, we have
\[
\begin{aligned}
\frac{\dif}{\dif t}&\int_{\mathbb{T}^2}\int_{\mathbb{R}}
\frac{1}{2}|\partial_\xi^\alpha\Psi|^2\dif \xi_1\dif \xi^{\prime}\\
&=-\int_{\mathbb{T}^2}\int_{\mathbb{R}}\partial_\xi^\alpha\left( u\cdot\nabla_\xi\Psi\right)\cdot\partial_\xi^\alpha\Psi\dif \xi_1\dif \xi^{\prime}
+\dot{X}(t)\int_{\mathbb{T}^2}\int_{\mathbb{R}}\partial_{\xi_1}\partial_\xi^\alpha(u^s)^{-X}\cdot\partial_\xi^\alpha\Psi\dif \xi_1\dif \xi^{\prime}\\
&\quad-\int_{\mathbb{T}^2}\int_{\mathbb{R}}\partial_\xi^\alpha\left(vF\partial_{\xi_1}(u^s)^{-X}\right)\cdot\partial_\xi^\alpha\Psi\dif \xi_1\dif \xi^{\prime}
-\int_{\mathbb{T}^2}\int_{\mathbb{R}}\partial_\xi^\alpha\left(v\nabla_\xi\left(p(v)-p((v^s)^{-X})\right)\right)\cdot\partial_\xi^\alpha\Psi\dif \xi_1\dif \xi^{\prime}\\
&\quad+\int_{\mathbb{T}^2}\int_{\mathbb{R}}\partial_\xi^\alpha\left(v\ddiv_\xi Q_1\right)\cdot\partial_\xi^\alpha\Psi\dif \xi_1\dif \xi^{\prime}
+\int_{\mathbb{T}^2}\int_{\mathbb{R}}\partial_\xi^\alpha\left(v\nabla_\xi Q_2\right)\cdot\partial_\xi^\alpha\Psi\dif \xi_1\dif \xi^{\prime}=:\sum\limits_{i=1}\limits^{6}R_{2i}.
\end{aligned}
\]
By applying methodologies corresponding to those used for $R_{12}$, $R_{13}$ and $R_{14}$ to $R_{21}$, $R_{22}$ and $R_{23}$, respectively, we first derive $R_{21}$ as follows
\begin{equation}\label{r21}
R_{21}\le C(\varepsilon_1+\delta)\|\nabla_\xi\Psi\|^2_{H^{|\alpha|-1}}.
\end{equation}
Then, for $R_{22}$, we obtain that
\begin{equation}\label{r22}
\begin{aligned}
R_{22}
\le C\delta^2|\dot{X}(t)|^2+C\delta\int_{\mathbb{T}^2}\int_{\mathbb{R}}|\partial_\xi^\alpha\Psi|^2\dif \xi_1\dif \xi^{\prime}.
\end{aligned}
\end{equation}
Next, for $R_{23}$, we imply that
\begin{equation}\label{r23}
\begin{aligned}
R_{23}
\le  C\delta\left(\|\nabla_\xi\Psi\|_{H^{|\alpha|-1}}^2+\|\nabla_\xi\Phi\|_{H^{|\alpha|-1}}^2+G_3(t)+G^s(t)\right)
.
\end{aligned}
\end{equation}
For $R_{24}$, firstly, we note that
\[
\nabla_\xi\left(p(v)-p((v^s)^{-X})\right)=p^\prime(v)\nabla_\xi\Phi+\nabla_\xi(v^s)^{-X}\left(p^\prime(v)-p^\prime((v^s)^{-X})\right),
\]
then, for $|\alpha|=3$, using Young's inequality, Lemma \ref{pvsw} and Sobolev's embedding theorem, we get
\begin{align*}
-\int_{\mathbb{T}^2}\int_{\mathbb{R}}\partial_{\xi_i}\left(vp^\prime(v)\right)\partial_\xi^{\alpha-1}\nabla_\xi\Phi
\cdot\partial_\xi^\alpha\Psi\dif \xi_1\dif \xi^{\prime}
\le C(\delta+\varepsilon_1)(\|\partial_\xi^{\alpha-1}\nabla_\xi\Phi\|_{L^2}^2+\|\partial_\xi^\alpha\Psi\|_{L^2}^2),
\end{align*}
\begin{align*}
&-\int_{\mathbb{T}^2}\int_{\mathbb{R}}\partial_\xi^{\alpha-1}\left(vp^\prime(v)\right)\partial_{\xi_i}\nabla_\xi\Phi
\cdot\partial_\xi^\alpha\Psi\dif \xi_1\dif \xi^{\prime}\\
&\le C\int_{\mathbb{T}^2}\int_{\mathbb{R}}\left(|\partial_{\xi_i}v|+|\partial_\xi^{\alpha-1}v|\right)|\partial_{\xi_i}\nabla_\xi\Phi|
|\partial_\xi^\alpha\Psi|\dif \xi_1\dif \xi^{\prime}\\
&\le C\left(\|\partial_{\xi_i}v\|_{L^\infty}\int_{\mathbb{T}^2}\int_{\mathbb{R}}|\partial_{\xi_i}\nabla_\xi\Phi|
|\partial_\xi^\alpha\Psi|\dif \xi_1\dif \xi^{\prime}+\|\partial_\xi^{\alpha-1}v\|_{L^3}\|\partial_{\xi_i}\nabla_\xi\Phi\|_{L^6}\|\partial_\xi^\alpha\Psi\|_{L^2}\right)\\
&\le C(\delta+\varepsilon_1)(\|\partial_{\xi_i}\nabla_\xi\Phi\|_{H^1}^2+\|\partial_\xi^\alpha\Psi\|_{L^2}^2),
\end{align*}
and
\begin{align*}
&-\int_{\mathbb{T}^2}\int_{\mathbb{R}}\partial_\xi^{\alpha}\left(vp^\prime(v)\right)\nabla_\xi\Phi
\cdot\partial_\xi^\alpha\Psi\dif \xi_1\dif \xi^{\prime}\\
&\le C\int_{\mathbb{T}^2}\int_{\mathbb{R}}\left(|\partial_{\xi_i}v|+|\partial_\xi^{\alpha-1}v|+|\partial_\xi^{\alpha}v|\right)|\nabla_\xi\Phi|
|\partial_\xi^\alpha\Psi|\dif \xi_1\dif \xi^{\prime}\\
&\le C\left(\|\partial_{\xi_i}v\|_{L^\infty}\int_{\mathbb{T}^2}\int_{\mathbb{R}}|\nabla_\xi\Phi|
|\partial_\xi^\alpha\Psi|\dif \xi_1\dif \xi^{\prime}+\|\partial_{\xi}^{\alpha-1}v\|_{L^3}\|\nabla_\xi\Phi\|_{L^6}\|\partial_\xi^\alpha\Psi\|_{L^2}
+\|\partial_{\xi}^{\alpha}v\|_{L^2}\|\nabla_\xi\Phi\|_{L^\infty}\|\partial_\xi^\alpha\Psi\|_{L^2}\right)\\
&\le C(\delta+\varepsilon_1)(\|\nabla_\xi\Phi\|_{H^2}^2+\|\partial_\xi^\alpha\Psi\|_{L^2}^2).
\end{align*}
On the other hand, in a similar way, we deduce that
\begin{align*}
-&\int_{\mathbb{T}^2}\int_{\mathbb{R}}\partial_\xi^\alpha\Big(v\nabla_\xi(v^s)^{-X}\left(p^\prime(v)-p^\prime((v^s)^{-X})\right)\Big)\cdot\partial_\xi^\alpha\Psi\dif \xi_1\dif \xi^{\prime}\\
&\le C(\delta+\varepsilon_1)(\|\nabla_\xi\Phi\|_{H^{|\alpha|-1}}^2+G^s(t)+\|\partial_\xi^\alpha\Psi\|_{L^2}^2).
\end{align*}
In addition, the cases $|\alpha|=1$ and $|\alpha|=2$ can be handled similarly. Thus, combining above estimates, we conclude that
\begin{equation}\label{r24}
\begin{aligned}
R_{24}\le C(\varepsilon_1+\delta)\left(G^s(t)+\|\nabla_\xi\Phi\|_{H^{|\alpha|-1}}^2+\|\partial_\xi^\alpha\Psi\|_{L^2}^2\right)
-\int_{\mathbb{T}^2}\int_{\mathbb{R}}
vp^\prime(v)\partial_\xi^\alpha\nabla_\xi\Phi
\cdot\partial_\xi^\alpha\Psi\dif \xi_1\dif \xi^{\prime}.
\end{aligned}
\end{equation}
For $R_{25}$ and $R_{26}$, which are treated similarly to $R_{15}$, it holds
\begin{equation}\label{r25}
R_{25}\le
C(\varepsilon_1+\delta)\left(
\|\partial_\xi^\alpha\Psi\|^2_{L^2}
+
\|\nabla_\xi Q_1\|_{H^{|\alpha|-1}}^2\right)
+\int_{\mathbb{T}^2}\int_{\mathbb{R}}v\partial_\xi^\alpha\ddiv_\xi Q_1\cdot\partial_\xi^\alpha\Psi\dif \xi_1\dif \xi^{\prime}.
\end{equation}
\begin{equation}\label{r26}
R_{26}\le
C(\varepsilon_1+\delta)\left(
\|\partial_\xi^\alpha\Psi\|^2_{L^2}
+
\|\nabla_\xi Q_2\|_{H^{|\alpha|-1}}^2\right)
+\int_{\mathbb{T}^2}\int_{\mathbb{R}}v\partial_\xi^\alpha\nabla_\xi Q_2\cdot\partial_\xi^\alpha\Psi\dif \xi_1\dif \xi^{\prime}.
\end{equation}
Therefore, combining \eqref{r21}, \eqref{r22}, \eqref{r23}, \eqref{r24}, \eqref{r25} and \eqref{r26}, we conclude that
\begin{equation}\label{hg1.2}
\begin{aligned}
&\frac{\dif}{\dif t}\int_{\mathbb{T}^2}\int_{\mathbb{R}}
\frac{1}{2}|\partial_\xi^\alpha\Psi|^2\dif \xi_1\dif \xi^{\prime}\\
&\le C(\varepsilon_1+\delta)\left(G^s(t)+\|(\nabla_\xi\Phi, \nabla_\xi\Psi) \|_{H^{|\alpha|-1}}^2+\|(\nabla_\xi Q_1, \nabla_\xi Q_2) \|_{H^{|\alpha|-1}}^2\right)
+C\delta G_3(t)
\\&\underbrace{-\int_{\mathbb{T}^2}\int_{\mathbb{R}}vp^\prime(v)\partial_\xi^\alpha\nabla_\xi\Phi\cdot\partial_\xi^\alpha\Psi
\dif \xi_1\dif \xi^{\prime}}\limits_{=:R_{21\ast}}
+\underbrace{\int_{\mathbb{T}^2}\int_{\mathbb{R}}v\partial_\xi^\alpha\ddiv_\xi Q_1\cdot\partial_\xi^\alpha\Psi\dif \xi_1\dif \xi^{\prime}}\limits_{=:R_{22\ast}}
+\underbrace{\int_{\mathbb{T}^2}\int_{\mathbb{R}}v\partial_\xi^\alpha\nabla_\xi Q_2\cdot\partial_\xi^\alpha\Psi\dif \xi_1\dif \xi^{\prime}}\limits_{=:R_{23\ast}}.
\end{aligned}
\end{equation}

Multiplying $\eqref{sys3}_3$ by $\frac{\partial_\xi^\alpha Q_1}{2\mu}$ and integrating over $\mathbb{R}\times\mathbb{T}^2$, we have
\begin{align*}
&\frac{\dif}{\dif t}\int_{\mathbb{T}^2}\int_{\mathbb{R}}
\frac{\tau}{4\mu}|\partial_\xi^\alpha Q_1|^2
\dif \xi_1\dif \xi^{\prime}\\
&=-\frac{\tau}{2\mu}\int_{\mathbb{T}^2}\int_{\mathbb{R}}\partial_\xi^\alpha\left( u\cdot\nabla_\xi Q_1\right):\partial_\xi^\alpha Q_1\dif \xi_1\dif \xi^{\prime}
+\dot{X}(t)\frac{\tau}{2\mu}\int_{\mathbb{T}^2}\int_{\mathbb{R}}\partial_{\xi_1}\partial_\xi^\alpha (\Pi_1^s)^{-X}:\partial_\xi^\alpha Q_1\dif \xi_1\dif \xi^{\prime}\\
&\quad-\frac{\tau}{2\mu}\int_{\mathbb{T}^2}\int_{\mathbb{R}}\partial_\xi^\alpha \left(v F\partial_{\xi_1}(\Pi_1^s)^{-X}\right):\partial_\xi^\alpha Q_1\dif \xi_1\dif \xi^{\prime}
-\frac{1}{2\mu}\int_{\mathbb{T}^2}\int_{\mathbb{R}}\partial_\xi^\alpha \left(vQ_1\right):\partial_\xi^\alpha Q_1\dif \xi_1\dif \xi^{\prime}\\
&\quad+\frac{1}{2}\int_{\mathbb{T}^2}\int_{\mathbb{R}}
\partial_\xi^\alpha \left(v\Big(\nabla_\xi\Psi+\left(\nabla_\xi\Psi\right)^T
   -\frac{2}{3}\ddiv_\xi\Psi \mathrm I_3\Big)\right):\partial_\xi^\alpha Q_1\dif \xi_1\dif \xi^{\prime}=:\sum\limits_{i=1}\limits^5R_{3i}.
\end{align*}
For $R_{31}$, following an approach analogous to the estimation of $R_{12}$, we deduce that
\begin{equation}\label{r31}
R_{31}\le C(\varepsilon_1+\delta)\|\nabla_\xi Q_1\|_{H^{|\alpha|-1}}^2.
\end{equation}
For $R_{32}$, using Young's inequality and Lemma \ref{pvsw}, we get
\begin{equation}\label{r32}
R_{32}\le C\delta^2|\dot{X}(t)|^2+C\delta\|\partial_\xi^\alpha Q_1\|_{L^2}^2.
\end{equation}
For $R_{33}$, by using a method similar to that used for estimating $R_{14}$, we have
\begin{equation}\label{r33}
\begin{aligned}
R_{33}\le C\delta\left(\|\nabla_\xi\Psi\|_{H^{|\alpha|-1}}^2+\|\nabla_\xi\Phi\|_{H^{|\alpha|-1}}^2+\|\nabla_\xi Q_1\|_{H^{|\alpha|-1}}^2+G_3(t)+G^s(t)\right)
.
\end{aligned}
\end{equation}
For $R_{34}$, straightforward computation yields for $|\alpha|=3$,
\begin{align*}
-\frac{1}{2\mu}\int_{\mathbb{T}^2}\int_{\mathbb{R}}\partial_{\xi_i}v  \left(\partial_\xi^{\alpha-1} Q_1\right):\partial_\xi^\alpha Q_1\dif \xi_1\dif \xi^{\prime}
\le C(\delta+\varepsilon_1)(\|\partial_\xi^{\alpha-1} Q_1\|_{L^2}^2+\|\partial_\xi^\alpha Q_1\|_{L^2}^2),
\end{align*}
\begin{align*}
-\frac{1}{2\mu}\int_{\mathbb{T}^2}\int_{\mathbb{R}}\partial_\xi^{\alpha-1} v \left(\partial _{\xi_i}Q_1\right):\partial_\xi^\alpha Q_1\dif \xi_1\dif \xi^{\prime}&
\le C\|\partial_\xi^{\alpha-1} v\|_{L^6}\|\partial _{\xi_i}Q_1\|_{L^3}\|\partial_\xi^\alpha Q_1\|_{L^2}\\
&\le C\|\partial_\xi^{\alpha-1} v\|_{H^1}\|\partial _{\xi_i}Q_1\|_{H^1}\|\partial_\xi^\alpha Q_1\|_{L^2}\\
&\le C(\delta+\varepsilon_1)(\|\partial_{\xi_i} Q_1\|_{H^1}^2+\|\partial_\xi^\alpha Q_1\|_{L^2}^2),
\end{align*}
and
\begin{align*}
-\frac{1}{2\mu}\int_{\mathbb{T}^2}\int_{\mathbb{R}}\partial_\xi^\alpha v Q_1:\partial_\xi^\alpha Q_1\dif \xi_1\dif \xi^{\prime}&\le C\|\partial_\xi^{\alpha} v\|_{L^2}\|Q_1\|_{L^\infty}\|\partial_\xi^\alpha Q_1\|_{L^2}\\
&\le C\|\partial_\xi^{\alpha-1} v\|_{L^2}\|Q_1\|_{H^2}\|\partial_\xi^\alpha Q_1\|_{L^2}\\
&\le C(\delta+\varepsilon_1)\| Q_1\|_{H^3}^2.
\end{align*}
As well as, the cases for $|\alpha|=1$ and $|\alpha|=2$ follow similarly. Thus, we can get
\begin{equation}\label{r34}
\begin{aligned}
R_{34}&\le -\frac{1}{2\mu}\int_{\mathbb{T}^2}\int_{\mathbb{R}}
 v|\partial_\xi^\alpha Q_1|^2\dif \xi_1\dif \xi^{\prime}+C(\varepsilon_1+\delta)\|Q_1\|_{H^{|\alpha|}}^2.
\end{aligned}
\end{equation}
In a similar way, for $R_{35}$, we first note that $Q_1$ is a symmetric and traceless matrix and
\[
\nabla_\xi\Psi=\frac{\nabla_\xi\Psi+(\nabla_\xi\Psi)^T}{2}+\frac{\nabla_\xi\Psi-(\nabla_\xi\Psi)^T}{2}.
\]
Then,
\begin{equation}\label{r35}
\begin{aligned}
R_{35}&=\int_{\mathbb{T}^2}\int_{\mathbb{R}}
\partial_\xi^\alpha \left(v \nabla_\xi\Psi\right):\partial_\xi^\alpha Q_1\dif \xi_1\dif \xi^{\prime}\\
&\le  C(\varepsilon_1+\delta)\left(\|\nabla_\xi\Psi\|_{H^{|\alpha|-1}}^2+
   \|\partial_\xi^\alpha Q_1\|_{L^2}^2\right)
 +\int_{\mathbb{T}^2}\int_{\mathbb{R}}
v\partial_\xi^\alpha\nabla_\xi\Psi:\partial_\xi^\alpha Q_1
\dif \xi_1\dif \xi^{\prime}.
\end{aligned}
\end{equation}
Therefore, combining \eqref{r31}-\eqref{r35} and using smallness of $\varepsilon_1$ and $\delta$, we conclude that
\begin{equation}\label{hg1.3}
\begin{aligned}
\frac{\dif}{\dif t}&\int_{\mathbb{T}^2}\int_{\mathbb{R}}
\frac{\tau}{2\mu}|\partial_\xi^\alpha Q_1|^2
\dif \xi_1\dif \xi^{\prime}
+\frac{1}{2\mu}\int_{\mathbb{T}^2}\int_{\mathbb{R}}
 v|\partial_\xi^\alpha Q_1|^2\dif \xi_1\dif \xi^{\prime}\\
&\le C(\varepsilon_1+\delta)\left(\|\nabla_\xi\Psi\|_{H^{|\alpha|-1}}^2+
   \| Q_1\|_{H^3}^2\right)
+C\delta\left(\|\nabla_\xi\Phi\|_{H^{|\alpha|-1}}^2+G_3(t)+G^s(t)\right)\\
&\quad+C\delta^2|\dot X(t)|^2+\underbrace{\int_{\mathbb{T}^2}\int_{\mathbb{R}}
v\partial_\xi^\alpha\nabla_\xi\Psi:\partial_\xi^\alpha Q_1
\dif \xi_1\dif \xi^{\prime}}\limits_{=:R_{3\ast}}.
\end{aligned}
\end{equation}
Multiplying $\eqref{sys3}_4$ by $\frac{\partial_\xi^\alpha Q_2}{\lambda}$ and integrating over $\mathbb{R}\times\mathbb{T}^2$, we have
\begin{align*}
&\frac{\dif}{\dif t}\int_{\mathbb{T}^2}\int_{\mathbb{R}}
\frac{\tau}{2\lambda}|\partial_\xi^\alpha Q_2|^2\dif \xi_1\dif \xi^{\prime}\\
&=-\frac{\tau}{\lambda}\int_{\mathbb{T}^2}\int_{\mathbb{R}}\partial_\xi^\alpha \left( u\cdot\nabla_\xi Q_2\right)\cdot\partial_\xi^\alpha Q_2\dif \xi_1\dif \xi^{\prime}
+\dot{X}(t)\frac{\tau}{\lambda}\int_{\mathbb{T}^2}\int_{\mathbb{R}}\partial_{\xi_1}\partial_\xi^\alpha (\Pi_2^s)^{-X}\cdot\partial_\xi^\alpha Q_2\dif \xi_1\dif \xi^{\prime}\\
&\quad-\frac{\tau}{\lambda}\int_{\mathbb{T}^2}\int_{\mathbb{R}}\partial_\xi^\alpha \left(v F\partial_{\xi_1}(\Pi_2^s)^{-X}\right)\cdot\partial_\xi^\alpha Q_2\dif \xi_1\dif \xi^{\prime}
-\frac{1}{\lambda}\int_{\mathbb{T}^2}\int_{\mathbb{R}}\partial_\xi^\alpha \left(vQ_2\right)\cdot\partial_\xi^\alpha Q_2\dif \xi_1\dif \xi^{\prime}\\
&\quad+\int_{\mathbb{T}^2}\int_{\mathbb{R}}
\partial_\xi^\alpha \left(v\ddiv_\xi \Psi\right)\cdot\partial_\xi^\alpha Q_2\dif \xi_1\dif \xi^{\prime}=:\sum\limits_{i=1}\limits^5R_{4i}.
\end{align*}
Following the method employed to estimate $R_{31}$-$R_{35}$, we can directly extend these results to the $R_{41}$-$R_{45}$. Then, for $R_{41}$, it holds that
\begin{equation}\label{r41}
R_{41}
\le C(\varepsilon_1+\delta) \|\nabla_\xi Q_2\|_{H^{|\alpha|-1}}^2.
\end{equation}
For $R_{42}$, we get
\begin{equation}\label{r42}
R_{42}\le C\delta^2|\dot{X}(t)|^2+C\delta\|\partial_\xi^\alpha Q_2\|_{L_2}^2.
\end{equation}
For $R_{43}$, we obtain that
\begin{equation}\label{r43}
R_{43}\le C\delta\left(\|\nabla_\xi\Psi\|_{H^{|\alpha|-1}}^2+\|\nabla_\xi\Phi\|_{H^{|\alpha|k-1}}^2+
\|\nabla_\xi Q_2\|_{H^{|\alpha|k-1}}^2+G_3(t)+G^s(t)\right).
\end{equation}
For $R_{44}$, we imply that
\begin{equation}\label{r44}
R_{44}
\le
C(\varepsilon_1+\delta)\|Q_2\|_{H^{|\alpha|}}^2
-\frac{1}{\lambda}\int_{\mathbb{T}^2}\int_{\mathbb{R}}
 v|\partial_\xi^\alpha Q_2|^2\dif \xi_1\dif \xi^{\prime}.
\end{equation}
And, for $R_{45}$, we deduce that
\begin{equation}\label{r45}
\begin{aligned}
R_{45}
\le  C(\varepsilon_1+\delta)\left(\|\nabla_\xi\Psi\|_{H^{|\alpha|-1}}^2+
   \|\partial_\xi^\alpha Q_2\|_{L^2}^2\right)
 +\int_{\mathbb{T}^2}\int_{\mathbb{R}}
v\partial_\xi^\alpha\ddiv_\xi\Psi\cdot\partial_\xi^\alpha Q_2
\dif \xi_1\dif \xi^{\prime}.
\end{aligned}
\end{equation}
Therefore, combining \eqref{r41}-\eqref{r45} and using smallness of $\varepsilon_1$ and $\delta$, we conclude that
\begin{equation}\label{hg1.4}
\begin{aligned}
\frac{\dif}{\dif t}&\int_{\mathbb{T}^2}\int_{\mathbb{R}}
\frac{\tau}{2\lambda}|\partial_\xi^\alpha Q_2|^2\dif \xi_1\dif \xi^{\prime}
+\frac{1}{2\lambda}\int_{\mathbb{T}^2}\int_{\mathbb{R}}
 v|\partial_\xi^\alpha Q_2|^2\dif \xi_1\dif \xi^{\prime}\\
&\le C(\varepsilon_1+\delta)\left(\|\nabla_\xi\Psi\|_{H^{|\alpha|-1}}^2+
   \| Q_2\|_{H^3}^2\right)
+C\delta\left(\|\nabla_\xi\Phi\|_{H^{|\alpha|-1}}^2+G_3(t)+G^s(t)\right)\\
&\quad+C\delta^2|\dot X(t)|^2+\underbrace{\int_{\mathbb{T}^2}\int_{\mathbb{R}}
v\partial_\xi^\alpha\ddiv_\xi \Psi\cdot\partial_\xi^\alpha Q_2\dif \xi_1\dif \xi^{\prime}}\limits_{=:R_{4\ast}}
\end{aligned}
\end{equation}

On the other hand, by applying integration by parts and using Young's inequality, Sobolev's embedding theorem, \eqref{3.4} and Lemma \ref{pvsw}, we obtain that
\begin{equation}\label{rr1}
\begin{aligned}
R_{1\ast}+R_{21\ast}&=\int_{\mathbb{T}^2}\int_{\mathbb{R}}\nabla_\xi\left(vp^{\prime}(v)\right)\cdot
\partial_\xi^\alpha\Phi\partial_\xi^\alpha\Psi
\dif \xi_1\dif \xi^{\prime}
\le C(\varepsilon_1+\delta)\left(\|\partial_\xi^\alpha\Phi\|_{L^2}^2+\|\partial_\xi^\alpha\Psi\|_{L^2}^2\right).
\end{aligned}
\end{equation}
\begin{equation}\label{rr2}
\begin{aligned}
R_{22\ast}+R_{3\ast}&=-\int_{\mathbb{T}^2}\int_{\mathbb{R}}\nabla_\xi v\cdot\partial_\xi^\alpha Q_1\cdot\partial_\xi^\alpha\Psi\dif \xi_1\dif \xi^{\prime}
\le C(\varepsilon_1+\delta)\left(\|\partial_\xi^\alpha\Psi\|_{L^2}^2+\|\partial_\xi^\alpha Q_1\|_{L^2}^2\right).
\end{aligned}
\end{equation}
And
\begin{equation}\label{rr3}
\begin{aligned}
R_{23\ast}+R_{4\ast}&=-\int_{\mathbb{T}^2}\int_{\mathbb{R}}\nabla_\xi v\cdot\partial_\xi^\alpha\Psi\partial_\xi^\alpha Q_2\dif \xi_1\dif \xi^{\prime}
\le C(\varepsilon_1+\delta)\left(\|\partial_\xi^\alpha\Psi\|_{L^2}^2+\|\partial_\xi^\alpha Q_2\|_{L^2}^2\right).
\end{aligned}
\end{equation}
Thus, combining \eqref{hg1.1}, \eqref{hg1.2}, \eqref{hg1.3}, \eqref{hg1.4} with \eqref{rr1}-\eqref{rr3} and choosing $\varepsilon_1, \delta$ sufficiently small, we can derive that
\begin{equation}\label{gjgj}
\begin{aligned}
&\frac{\dif}{\dif t}\int_{\mathbb{T}^2}\int_{\mathbb{R}}\left(
-\frac{p^\prime(v)}{2}|\partial_\xi^\alpha \Phi|^2+\frac{1}{2}|\partial_\xi^\alpha k\Psi|^2
+\frac{\tau}{2\mu}|\partial_\xi^\alpha Q_1|^2+\frac{\tau}{2\lambda}|\partial_\xi^\alpha Q_2|^2
\right)\dif \xi_1\dif \xi^{\prime}\\
&\qquad+\frac{1}{4\mu}\int_{\mathbb{T}^2}\int_{\mathbb{R}}
 v|\partial_\xi^\alpha Q_1|^2\dif \xi_1\dif \xi^{\prime}
 +\frac{1}{4\lambda}\int_{\mathbb{T}^2}\int_{\mathbb{R}}
 v|\partial_\xi^\alpha Q_2|^2\dif \xi_1\dif \xi^{\prime}\\
 &\le C\delta^2|\dot X(t)|^2+C(\delta+\varepsilon_1)\left(G^s(t)+\|(\nabla_\xi \Phi,  \nabla_\xi\Psi) \|_{H^{|\alpha|-1}}^2
 +\|(Q_1, Q_2)\|_{L^2}^2\right)
 +C\delta G_3(t).
\end{aligned}
\end{equation}

Finally, integrating the equality \eqref{gjgj} over $[0, t]$, the proof of this lemma is finished.
\end{proof}

\subsection{Dissipative estimates}
In the following lemmas, we give the dissipative estimates of given solutions of system \eqref{sys1}.
 \begin{lemma}\label{lehs1}
Under the hypotheses of Proposition \ref{p1}, there exists constant $C>0$ independent of $\tau, \nu, \delta, \varepsilon_1, T$, such that for $t\in[0, T]$, we have
\begin{equation}\label{hsgj1.1}
\begin{aligned}
&\int_0^t\|\nabla_\xi(v-(v^s)^{-X})|_{L^2}^2\dif t+\int_0^t\|\nabla_\xi(u-(u^s)^{-X})|_{L^2}^2\dif t\\
&\le 
\chi_1\|\sqrt\rho(u-(u^s)^{-X})\|_{L^2}^2
+C\|(v_0-v^s, u_0-u^s
\sqrt{\tau}(\Pi_{10}-\Pi_1^s), \sqrt{\tau}(\Pi_{20}-\Pi_2^s))\|_{H^1}^2.
\\
&\quad+C\delta^2\int_0^t|\dot X(t)|^2\dif t
+C(\delta
+\varepsilon_1)\int_0^t\left(G_3(t)+G^s(t)\right)\dif t+C\delta\int_0^t G_2(t)\dif t,
\end{aligned}
\end{equation}
where $\chi_1$ is positive constants to be chosen sufficiently small later, $G_2(t)$ and $G_{3}(t)$ are defined in Lemma \ref{le4.5} and $G^s(t)$ is defined in Lemma \ref{le0}.
\end{lemma}
\begin{proof}
Multiplying equation $\eqref{sys3}_2$ by $\rho\nabla_\xi\Phi$ and integrating over $(0, t)\times\mathbb{R}\times\mathbb{T}^2$, we have
\begin{align*}
-&\int_0^t\int_{\mathbb{T}^2}\int_{\mathbb{R}}p^\prime(v)|\nabla_\xi\Phi|^2\dif \xi_1\dif \xi^{\prime}\dif t\\
=&\int_0^t\int_{\mathbb{T}^2}\int_{\mathbb{R}}
\rho\partial_t\Psi
\cdot\nabla_\xi\Phi\dif \xi_1\dif \xi^{\prime}\dif t
-\sigma\int_0^t\int_{\mathbb{T}^2}\int_{\mathbb{R}}
\rho\partial_{\xi_1}\Psi
\cdot\nabla_\xi\Phi\dif \xi_1\dif \xi^{\prime}\dif t\\
&+\int_0^t\int_{\mathbb{T}^2}\int_{\mathbb{R}}
        \rho( u\cdot\nabla_\xi\Psi)
\cdot\nabla_\xi\Phi\dif \xi_1\dif \xi^{\prime}\dif t
-\int_0^t\int_{\mathbb{T}^2}\int_{\mathbb{R}}
    \rho \dot{X}(t)\partial_{\xi_1}(u^s)^{-X}
\cdot\nabla_\xi\Phi\dif \xi_1\dif \xi^{\prime}\dif t\\
&+\int_0^t\int_{\mathbb{T}^2}\int_{\mathbb{R}}
        (F\partial_{\xi_1}(u^s)^{-X})
\cdot\nabla_\xi\Phi\dif \xi_1\dif \xi^{\prime}\dif t
+\int_0^t\int_{\mathbb{T}^2}\int_{\mathbb{R}}
    \ddiv_\xi Q_1
\cdot\nabla_\xi\Phi\dif \xi_1\dif \xi^{\prime}\dif t\\
&+\int_0^t\int_{\mathbb{T}^2}\int_{\mathbb{R}}
     \nabla_\xi Q_2
\cdot\nabla_\xi\Phi\dif \xi_1\dif \xi^{\prime}\dif t
+\int_0^t\int_{\mathbb{T}^2}\int_{\mathbb{R}}
(p(v)-p((v^s)^{-X})\nabla_\xi(v^s)^{-X}\cdot\nabla_\xi\Phi
\dif \xi_1\dif \xi^{\prime}\dif t\\
=:&\sum\limits_{i=1}\limits^8M_i.
\end{align*}
For $M_1$, by applying integration by parts and using $\eqref{sys1}_1$, we get
\begin{align*}
M_1
=&\int_{\mathbb{T}^2}\int_{\mathbb{R}}
\left(\rho\Psi
\cdot\nabla_\xi\Phi-\rho_0\Psi_0
\cdot\nabla_\xi\Phi_0\right)\dif \xi_1\dif \xi^{\prime}
+\sigma\int_0^t\int_{\mathbb{T}^2}\int_{\mathbb{R}}
\ddiv_\xi(\rho\Psi) \partial_{\xi_1}\Phi
\dif \xi_1\dif \xi^{\prime}\dif t\\
&-\int_0^t\int_{\mathbb{T}^2}\int_{\mathbb{R}}
\ddiv_\xi(\rho\Psi)  (u\cdot\nabla_\xi\Phi)
        \dif \xi_1\dif \xi^{\prime}\dif t
+\int_0^t\int_{\mathbb{T}^2}\int_{\mathbb{R}}
  \ddiv_\xi(\rho\Psi)  (\dot{X}(t)\partial_{\xi_1}(v^s)^{-X})
    \dif \xi_1\dif \xi^{\prime}\dif t\\
&-\int_0^t\int_{\mathbb{T}^2}\int_{\mathbb{R}}
      \ddiv_\xi(\rho\Psi)  (vF\partial_{\xi_1}(v^s)^{-X})
        \dif \xi_1\dif \xi^{\prime}\dif t
+\int_0^t\int_{\mathbb{T}^2}\int_{\mathbb{R}}
  \ddiv_\xi(\rho\Psi)  (v\ddiv_\xi\Psi)
    \dif \xi_1\dif \xi^{\prime}\dif t\\
    &
   + \int_{\mathbb{T}^2}\int_{\mathbb{R}}
\ddiv_{\xi}(\rho u)\Psi
\cdot\nabla_\xi\Phi\dif \xi_1\dif \xi^{\prime}\dif t
    =:\sum\limits_{j=1}\limits^7M_{1j},
\end{align*}
where $\Psi_0=(u_0-u^s)(\xi)$ and $\nabla_\xi\Phi_0=\nabla_\xi (v_0-v^s)(\xi)$. By using Young's inequality and Lemma \ref{pvsw}, we imply that
\[
M_{11}\le \chi_1\|\sqrt\rho\Psi\|_{L^2}^2+C(\chi_1)\|\nabla_\xi\Phi\|_{L^2}^2
+C(\|\Psi_0\|_{L^2}^2+\|\nabla_\xi\Phi_0\|_{L^2}^2),
\]
and
\begin{align*}
M_{12}&=
\sigma\int_0^t\int_{\mathbb{T}^2}\int_{\mathbb{R}}
\partial_{\xi_1}\rho \Psi \cdot\nabla_\xi\Phi
\dif \xi_1\dif \xi^{\prime}\dif t
+\sigma\int_0^t\int_{\mathbb{T}^2}\int_{\mathbb{R}}
\rho\partial_{\xi_1}\Psi \cdot\nabla_\xi\Phi
\dif \xi_1\dif \xi^{\prime}\dif t\\
&\le C\|\Psi\|_{L^\infty}\int_0^t\int_{\mathbb{T}^2}\int_{\mathbb{R}}|\nabla_\xi\Phi|^2\dif \xi_1\dif \xi^{\prime}\dif t
+C\int_0^t\int_{\mathbb{T}^2}\int_{\mathbb{R}}
(v^s)^{-X}_{\xi_1}|p(v)-p((v^s)^{-X})||\nabla_\xi\Phi|\dif \xi_1\dif \xi^{\prime}\dif t\\
&\quad+\frac{\delta}{\nu}\int_0^t\int_{\mathbb{T}^2}\int_{\mathbb{R}}
a_{\xi_1}^{-X}\Big|u_1-(u_1^s)^{-X}-\frac{p(v)-p((v^s)^{-X})}{\sigma_\ast}\Big||\nabla_\xi\Phi|\dif \xi_1\dif \xi^{\prime}\dif t\\
&\quad+\frac{\delta}{\nu}\int_0^t\int_{\mathbb{T}^2}\int_{\mathbb{R}}
a_{\xi_1}^{-X}(|u_1|+|u_3|)|\nabla_\xi\Phi|\dif \xi_1\dif \xi^{\prime}\dif t
-M_2\\
&\quad\le C(\delta+\varepsilon_1)\int_0^t\|\nabla_\xi\Phi\|_{L^2}^2\dif t
+C\delta\int_0^t\left(G_2(t)+G_3(t)+G^s(t)\right)\dif t-M_2.
\end{align*}
Similarly, for $M_{13}$,
\begin{align*}
M_{13}&\le C\int_0^t\int_{\mathbb{T}^2}\int_{\mathbb{R}}|\Psi||\nabla_\xi\Phi|^2\dif \xi_1\dif \xi^{\prime}\dif t
+C\int_0^t\int_{\mathbb{T}^2}\int_{\mathbb{R}}(v^s)^{-X}_{\xi_1}|\Psi||\nabla_\xi\Phi|\dif \xi_1\dif \xi^{\prime}\dif t\\
&\quad+C\int_0^t\int_{\mathbb{T}^2}\int_{\mathbb{R}}|\ddiv_{\xi}\Psi||\Psi||\nabla_\xi\Phi|\dif \xi_1\dif \xi^{\prime}\dif t
-\int_0^t\int_{\mathbb{T}^2}\int_{\mathbb{R}}
\rho\ddiv_\xi\Psi(u_1^s)^{-X}\partial_{\xi_1}\Phi\dif \xi_1\dif \xi^{\prime}\dif t\\
&\le  C(\delta+\varepsilon_1)\int_0^t\|\nabla_\xi\Phi\|_{L^2}^2\dif t
+C\varepsilon_1\int_0^t\|\nabla_\xi\Psi\|_{L^2}^2\dif z
+C\delta\int_0^t\left(G_2(t)+G_3(t)+G^s(t)\right)\dif t\\
&\quad-\int_0^t\int_{\mathbb{T}^2}\int_{\mathbb{R}}
\rho\ddiv_\xi\Psi(u_1^s)^{-X}\partial_{\xi_1}\Phi\dif \xi_1\dif \xi^{\prime}\dif t.
\end{align*}
For $M_{14}$,
\begin{align*}
M_{14}&\le
C\int_0^t\int_{\mathbb{T}^2}\int_{\mathbb{R}}
 |\ddiv_\xi\Psi|  |\dot{X}(t)|(v^s)^{-X}_{\xi_1}
    \dif \xi_1\dif \xi^{\prime}\dif t+
C\int_0^t\int_{\mathbb{T}^2}\int_{\mathbb{R}}
 |\nabla_\xi\Phi||\Psi|  |\dot{X}(t)|(v^s)^{-X}_{\xi_1}
    \dif \xi_1\dif \xi^{\prime}\dif t\\
&\quad+C\int_0^t\int_{\mathbb{T}^2}\int_{\mathbb{R}}
|\Psi|  |\dot{X}(t)|\left((v^s)^{-X}_{\xi_1}\right)^2
    \dif \xi_1\dif \xi^{\prime}\dif t\\
&\le C\delta^2\int_0^t|\dot X(z)|^2\dif t
+C\varepsilon_1\int_0^t\|\nabla_\xi\Phi\|_{L^2}^2\dif t
+C\delta\int_0^t\left(\|\nabla_\xi\Psi\|_{L^2}^2+G_2(t)+G_3(t)+G^s(t)\right)\dif t.
\end{align*}
In particular, together with \eqref{wcxf}, we obtain that
\begin{align*}
M_{15}&\le
C\int_0^t\int_{\mathbb{T}^2}\int_{\mathbb{R}}
      |\nabla_\xi\Phi||\Psi| \left(\Big|u_1-(u_1)^s-\frac{p(v)-p((v^s)^{-X})}{\sigma_\ast}\Big|
      +|p(v)-p((v^s)^{-X})|\right)(v^s)^{-X}_{\xi_1}
        \dif \xi_1\dif \xi^{\prime}\dif t\\
&\quad+C\int_0^t\int_{\mathbb{T}^2}\int_{\mathbb{R}}
  |\Psi| \left(\Big|u_1-(u_1)^s-\frac{p(v)-p((v^s)^{-X})}{\sigma_\ast}\Big|
      +|p(v)-p((v^s)^{-X})|\right)\left((v^s)^{-X}_{\xi_1}\right)^2
        \dif \xi_1\dif \xi^{\prime}\dif t\\
&\quad+C\int_0^t\int_{\mathbb{T}^2}\int_{\mathbb{R}}
      |\ddiv_\xi\Psi|
      \left(\Big|u_1-(u_1)^s-\frac{p(v)-p((v^s)^{-X})}{\sigma_\ast}\Big|
      +|p(v)-p((v^s)^{-X})|\right)(v^s)^{-X}_{\xi_1}
        \dif \xi_1\dif \xi^{\prime}\dif t\\
&\le C\varepsilon_1\int_0^t\|\nabla_\xi\Phi\|_{L^2}^2\dif t
+C\delta\int_0^t\|\nabla_\xi\Phi\|_{L^2}^2\dif t+
C(\delta+\varepsilon_1)\int_0^t\left(G_3(t)+G^s(t)\right)\dif t.
\end{align*}
In a similar way, for $M_{16}$,
\begin{align*}
M_{16}&\le C\int_0^t\int_{\mathbb{T}^2}\int_{\mathbb{R}}
  |\nabla_\xi\Phi||\Psi||\ddiv_\xi\Psi|
    \dif \xi_1\dif \xi^{\prime}\dif t
    +C\int_0^t\int_{\mathbb{T}^2}\int_{\mathbb{R}}
(v^s)^{-X}_{\xi_1}|\Psi||\ddiv_\xi\Psi|
    \dif \xi_1\dif \xi^{\prime}\dif t\\
&\quad+\int_0^t\int_{\mathbb{T}^2}\int_{\mathbb{R}}
 |\ddiv_\xi\Psi|^2
    \dif \xi_1\dif \xi^{\prime}\dif t\\
 &\le C\varepsilon_1\int_0^t\|\nabla_\xi\Phi\|_{L^2}^2\dif t
 +C(\varepsilon_1+\delta)\int_0^t\|\nabla_\xi\Psi\|_{L^2}^2\dif t
+C\delta\int_0^t\left(G_2(z)+G_3(z)+G^s(z)\right)\dif t\\
&\quad+\int_0^t\int_{\mathbb{T}^2}\int_{\mathbb{R}}
 |\ddiv_\xi\Psi|^2
    \dif \xi_1\dif \xi^{\prime}\dif t.
\end{align*}
For $M_{17}$,
\begin{align*}
M_{17}&\le C\int_{\mathbb{T}^2}\int_{\mathbb{R}}
|\Psi|
|\nabla_\xi\Phi|^2\dif \xi_1\dif \xi^{\prime}\dif t
+C\int_{\mathbb{T}^2}\int_{\mathbb{R}}
(v^s)^{-X}_{\xi_1}|\Psi|
|\nabla_\xi\Phi|\dif \xi_1\dif \xi^{\prime}\dif t\\
&\quad+C\int_{\mathbb{T}^2}\int_{\mathbb{R}}
|\ddiv_\xi\Psi|  |\Psi|
|\nabla_\xi\Phi|\dif \xi_1\dif \xi^{\prime}\dif t
+C\int_{\mathbb{T}^2}\int_{\mathbb{R}}
|(u_1^s)^{-X}_{\xi_1}|  |\Psi|
|\nabla_\xi\Phi|\dif \xi_1\dif \xi^{\prime}\dif t\\
&\le C\varepsilon_1\int_0^t\|\nabla_\xi\Psi\|_{L^2}^2\dif t
 +C(\varepsilon_1+\delta)\int_0^t\|\nabla_\xi\Phi\|_{L^2}^2\dif t
+C\delta\int_0^t\left(G_2(t)+G_3(t)+G^s(t)\right)\dif t.
\end{align*}
Thus, combining the above estimates, we derive that
\begin{equation}\label{m1}
\begin{aligned}
M_1+M_2&\le \chi_1\|\sqrt\rho\Psi\|_{L^2}^2+C(\chi_1)\|\nabla_\xi\Phi\|_{L^2}^2
+C(\|\Psi_0\|_{L^2}^2+\|\nabla_\xi\Phi_0\|_{L^2}^2)
+C\delta\int_0^t G_2(t)\dif t\\
&+C\delta^2\int_0^t|\dot X(t)|^2\dif t
+C(\delta
+\varepsilon_1)\int_0^t\left(\|\nabla_\xi\Phi\|_{L^2}^2+\|\nabla_\xi\Psi\|_{L^2}^2+G_3(t)+G^s(t)\right)\dif t\\
&\underbrace{-\int_0^t\int_{\mathbb{T}^2}\int_{\mathbb{R}}
\rho\ddiv_\xi\Psi(u_1^s)^{-X}\partial_{\xi_1}\Phi\dif \xi_1\dif \xi^{\prime}\dif t}\limits_{=:M_{13\ast}}.
\end{aligned}
\end{equation}
In a similar way, using \eqref{wcxf}, Young's inequality and Lemma \ref{pvsw}, we can get
\begin{equation}\label{m3}
\begin{aligned}
M_3&\le C\int_0^t\int_{\mathbb{T}^2}\int_{\mathbb{R}}
        |\Psi||\nabla_\xi\Psi|
|\nabla_\xi\Phi|\dif \xi_1\dif \xi^{\prime}\dif t
+\int_0^t\int_{\mathbb{T}^2}\int_{\mathbb{R}}
        \rho (u_1^s)^{-X}\partial_{\xi_1}\Psi
\cdot\nabla_\xi\Phi\dif \xi_1\dif \xi^{\prime}\dif t\\
&\le C\varepsilon_1\int_0^t\left(\|\nabla_\xi\Phi\|_{L^2}^2+\|\nabla_\xi\Psi\|_{L^2}^2\right)\dif t
+\underbrace{\int_0^t\int_{\mathbb{T}^2}\int_{\mathbb{R}}
        \rho (u_1^s)^{-X}\partial_{\xi_1}\Psi
\cdot\nabla_\xi\Phi\dif \xi_1\dif \xi^{\prime}\dif t}\limits_{=:M_{3\ast}}.
\end{aligned}
\end{equation}
\begin{equation}\label{m4}
M_4 \le C\delta^2\int_0^t|\dot X(t)|^2\dif t+C\delta\int_0^t\|\nabla_\xi\Phi\|_{L^2}^2\dif t.
\end{equation}
\begin{equation}\label{m5}
M_5 \le C\delta\int_0^t\left(G_3(t)+G^s(t)+\|\nabla_\xi\Phi\|_{L^2}^2\right)\dif t.
\end{equation}
\begin{equation}\label{m6}
M_6 \le -\frac{\chi}{2}\int_0^t\int_{\mathbb{T}^2}\int_{\mathbb{R}}vp^\prime(v)|\nabla_\xi\Phi|^2\dif \xi_1\dif \xi^{\prime}\dif t+C\int_0^t\|\nabla_\xi Q_1\|_{L^2}^2\dif t.
\end{equation}
\begin{equation}\label{m7}
M_7 \le -\frac{\chi}{2}\int_0^t\int_{\mathbb{T}^2}\int_{\mathbb{R}}vp^\prime(v)|\nabla_\xi\Phi|^2\dif \xi_1\dif \xi^{\prime}\dif t+C\int_0^t\|\nabla_\xi Q_2\|_{L^2}^2\dif t.
\end{equation}
And
\begin{equation}\label{m8}
M_8 \le C\delta\int_0^t\left(G^s(t)+\|\nabla_\xi\Phi\|_{L^2}^2\right)\dif t.
\end{equation}
On the other hand, using Lemma \ref{pvsw}, we deduce that
\begin{equation}\label{m9}
\begin{aligned}
&M_{13\ast}+M_{3\ast}\\
&=\int_0^t\int_{\mathbb{T}^2}\int_{\mathbb{R}}
\left(\rho(u_1^s)^{-X}\right)_{\xi_1}\ddiv_\xi\Psi\Phi\dif \xi_1\dif \xi^{\prime}\dif t
-\int_0^t\int_{\mathbb{T}^2}\int_{\mathbb{R}}
\nabla_\xi\left(\rho(u_1^s)^{-X}\right)\cdot\Psi_{\xi_1}\Phi\dif \xi_1\dif \xi^{\prime}\dif t\\
&\le C(\delta
+\varepsilon_1)\int_0^t\|\nabla_\xi\Psi\|_{L^2}^2\dif t
+C\varepsilon_1\int_0^t\|\nabla_\xi\Phi\|_{L^2}^2\dif t
+C\delta\int_0^tG^s(t)\dif t.
\end{aligned}
\end{equation}

Therefore, combining \eqref{m1}-\eqref{m9}, and using the smallness of $\delta$ and using Lemma \ref{legj}, we conclude that
\begin{equation}\label{hsgj10301}
	\begin{aligned}
		&(1-\chi-C(\delta+\varepsilon_1))\int_0^t\int_{\mathbb{T}^2}\int_{\mathbb{R}}|p^\prime(v)||\nabla_\xi\Phi|^2\dif \xi_1\dif \xi^{\prime}\dif t\\
		&\le
		\chi_1\|\sqrt\rho\Psi\|_{L^2}^2
		+C\|(\Phi_0, \Psi_0,
		\sqrt{\tau}Q_{10}, \sqrt{\tau}Q_{20})\|_{H^3}^2
	+C\delta^2\int_0^t|\dot X(t)|^2\dif t	+C\delta\int_0^t G_2(t)\dif t\\
&\quad		+C(\delta
		+\varepsilon_1)\int_0^t\left(\|\nabla_\xi\Psi\|_{L^2}^2+G_3(t)+G^s(t)\right)\dif t
	+\int_0^t
		\| \ddiv_\xi\Psi\|_{L^2}^2\dif t,
	\end{aligned}
\end{equation}
where $Q_{10}=(\Pi_{10}-\Pi_1^s)(\xi)$ and $Q_{20}=(\Pi_{20}-\Pi_2^s)(\xi)$.
Next, multiplying equation $\eqref{sys1}_3$ by $\rho\nabla_\xi\Psi$ and integrating over $(0, t)\times\mathbb{R}\times\mathbb{T}^2$, we have
\begin{align*}
&\mu\int_0^t\int_{\mathbb{T}^2}\int_{\mathbb{R}}|\nabla_\xi\Psi|^2\dif \xi_1\dif \xi^{\prime}\dif t
+\frac{\mu}{3}\int_0^t\int_{\mathbb{T}^2}\int_{\mathbb{R}}
   |\ddiv_\xi\Psi|^2
\dif \xi_1\dif \xi^{\prime}\dif t\\
=&\tau\int_0^t\int_{\mathbb{T}^2}\int_{\mathbb{R}}\rho
\partial_tQ_1
:\nabla_\xi\Psi\dif \xi_1\dif \xi^{\prime}\dif t
-\tau\sigma\int_0^t\int_{\mathbb{T}^2}\int_{\mathbb{R}}\rho
\partial_{\xi_1}Q_1
:\nabla_\xi\Psi\dif \xi_1\dif \xi^{\prime}\dif t\\
&+\tau\int_0^t\int_{\mathbb{T}^2}\int_{\mathbb{R}}\rho
        ( u\cdot\nabla_\xi Q_1)
:\nabla_\xi\Psi\dif \xi_1\dif \xi^{\prime}\dif t
-\tau\int_0^t\int_{\mathbb{T}^2}\int_{\mathbb{R}}\rho
     \dot{X}(t)\partial_{\xi_1}(\Pi_1^s)^{-X}
:\nabla_\xi\Psi\dif \xi_1\dif \xi^{\prime}\dif t\\
&+\tau\int_0^t\int_{\mathbb{T}^2}\int_{\mathbb{R}}\rho
        (vF\partial_{\xi_1}(\Pi_1^s)^{-X})
:\nabla_\xi\Psi\dif \xi_1\dif \xi^{\prime}\dif t
+\int_0^t\int_{\mathbb{T}^2}\int_{\mathbb{R}}\rho
    v  Q_1
:\nabla_\xi\Psi\dif \xi_1\dif \xi^{\prime}\dif t\\
=:&\sum\limits_{i=1}\limits^6N_i.
\end{align*}
For $N_1$, by applying integration by parts and using $\eqref{sys1}_2$, we get
\begin{align*}
N_1
=&\tau\int_{\mathbb{T}^2}\int_{\mathbb{R}}
\left(\rho Q_1
\cdot\nabla_\xi\Psi-\rho_0 Q_{10}
\cdot\nabla_\xi\Psi_0\right)\dif \xi_1\dif \xi^{\prime}
+\tau\sigma\int_0^t\int_{\mathbb{T}^2}\int_{\mathbb{R}}
\ddiv_\xi(\rho Q_1)\cdot \partial_{\xi_1}\Psi
\dif \xi_1\dif \xi^{\prime}\dif t\\
&-\tau\int_0^t\int_{\mathbb{T}^2}\int_{\mathbb{R}}
\ddiv_\xi (\rho Q_1)\cdot (u\cdot\nabla_\xi\Psi)
        \dif \xi_1\dif \xi^{\prime}\dif t
+\tau\int_0^t\int_{\mathbb{T}^2}\int_{\mathbb{R}}
  \ddiv_\xi (\rho Q_1)\cdot (\dot{X}(t)\partial_{\xi_1}(u^s)^{-X})
    \dif \xi_1\dif \xi^{\prime}\dif t\\
&-\tau\int_0^t\int_{\mathbb{T}^2}\int_{\mathbb{R}}
      \ddiv_\xi (\rho Q_1)\cdot (vF\partial_{\xi_1}(u^s)^{-X})
        \dif \xi_1\dif \xi^{\prime}\dif t\\
&+\tau\int_0^t\int_{\mathbb{T}^2}\int_{\mathbb{R}}
  \ddiv_\xi (\rho Q_1)\cdot (v\nabla_\xi\left(p(v)-p((v^s)^{-X})\right))
    \dif \xi_1\dif \xi^{\prime}\dif t\\
    &-\tau\int_0^t\int_{\mathbb{T}^2}\int_{\mathbb{R}}
      \ddiv_\xi (\rho Q_1)\cdot (v\ddiv_\xi Q_1)
        \dif \xi_1\dif \xi^{\prime}\dif t
+\tau\int_0^t\int_{\mathbb{T}^2}\int_{\mathbb{R}}
  \ddiv_\xi (\rho Q_1)\cdot (v\nabla_\xi Q_2)
    \dif \xi_1\dif \xi^{\prime}\dif t\\
    &+\int_0^t\int_{\mathbb{T}^2}\int_{\mathbb{R}}
  \ddiv_\xi (\rho u)Q_1:\nabla_\xi\Psi
    \dif \xi_1\dif \xi^{\prime}\dif t
    =:\sum\limits_{j=1}\limits^9N_{1j}.
\end{align*}
 By using \eqref{sccs}, Young's inequality and Lemma \ref{pvsw}, we imply that
\[
N_{11}\le \chi_1\tau\|\sqrt\rho Q_{1}\|_{L^2}^2+C(\chi_1)\|\nabla_\xi\Psi\|_{L^2}^2
+C(\tau\|Q_{10}\|_{L^2}^2+\|\nabla_\xi\Psi_0\|_{L^2}^2),
\]
\begin{align*}
N_{12}&=\tau\sigma\int_0^t\int_{\mathbb{T}^2}\int_{\mathbb{R}}
\rho_{\xi_1}Q_1
:\nabla_\xi\Psi\dif \xi_1\dif \xi^{\prime}\dif t
+\tau\sigma\int_0^t\int_{\mathbb{T}^2}\int_{\mathbb{R}}\rho
\partial_{\xi_1}Q_1
:\nabla_\xi\Psi\dif \xi_1\dif \xi^{\prime}\dif t\\
&\le C(\varepsilon_1+\delta)\int_0^t\left(\| Q_1\|_{L^2}^2+\|\nabla_\xi \Psi\|_{L^2}^2\right)\dif t -N_2,
\end{align*}
\[
N_{13}\le \frac{\chi}{4}\int_0^t\int_{\mathbb{T}^2}\int_{\mathbb{R}}|\nabla_\xi\Psi|^2\dif \xi_1\dif \xi^{\prime}\dif t
+C\int_0^t\|\nabla_\xi Q_1\|_{L^2}^2\dif t
+C(\varepsilon_1+\delta)\int_0^t\| Q_1\|_{L^2}^2\dif t,
\]
\[
N_{14}\le C\delta^2\int_0^t|\dot X(t)|^2\dif t+C\delta\int_0^t\|\nabla_\xi Q_1\|_{L^2}^2\dif t
+C(\varepsilon_1+\delta)\int_0^t\| Q_1\|_{L^2}^2\dif t,
\]
\[
N_{15}\le C(\varepsilon_1+\delta)\int_0^t\left(G_3(t)+G^s(t)+\| Q_1\|_{L^2}^2\right)\dif t+C\delta\int_0^t\|\nabla_\xi Q_1\|_{L^2}^2\dif t,
\]
\begin{align*}
N_{16}
    \le \frac{-\chi}{2}\int_0^t\int_{\mathbb{T}^2}\int_{\mathbb{R}}vp^\prime(v)|\nabla_\xi\Phi|^2\dif \xi_1\dif \xi^{\prime}\dif t
    +C\int_0^t\|\nabla_\xi Q_1\|_{L^2}^2\dif t
    +C(\varepsilon_1+\delta)\int_0^t\left(G^s(z)+\| Q_1\|_{L^2}^2\right)\dif t,
  \end{align*}
\[
 N_{17} \le C(\varepsilon_1+\delta)\int_0^t\| Q_1\|_{L^2}^2\dif t
 +C\int_0^t\|\nabla_\xi Q_1\|_{L^2}^2\dif t,
\]
\[
  N_{18} \le C\int_0^t\left(\|\nabla_\xi Q_1\|_{L^2}^2+\|\nabla_\xi Q_2\|_{L^2}^2\right)\dif t
  +C(\varepsilon_1+\delta)\int_0^t\| Q_1\|_{L^2}^2\dif t,
\]
and
\[
N_{19}\le C(\varepsilon_1+\delta)\int_0^t\left(\| Q_1\|_{L^2}^2+\|\nabla_\xi \Psi\|_{L^2}^2\right)\dif t.
\]
Thus, combining the above estimates, we derive that
\begin{equation}\label{n1}
\begin{aligned}
&N_1+N_2\\&
\le \chi_1\tau\|\sqrt\rho Q_{1}\|_{L^2}^2+C(\chi_1)\|\nabla_\xi\Psi\|_{L^2}^2
+C(\tau\|Q_{10}\|_{L^2}^2+\|\nabla_\xi\Psi_0\|_{L^2}^2)
+C\delta^2\int_0^t|\dot X(t)|^2\dif t\\
&
+(\frac{\chi}{2}+C(\delta+\varepsilon_1))\int_0^t\int_{\mathbb{T}^2}\int_{\mathbb{R}}|\nabla_\xi\Psi|^2\dif \xi_1\dif \xi^{\prime}\dif t
+\frac{-\chi}{4}\int_0^t\int_{\mathbb{T}^2}\int_{\mathbb{R}}vp^\prime(v)|\nabla_\xi\Phi|^2\dif \xi_1\dif \xi^{\prime}\dif t\\
&+C(\delta+\varepsilon_1)\int_0^t\left(G_3(t)+G^s(t)+\|Q_1\|_{L^2}^2\right)\dif t
+C\int_0^t\left(\|\nabla_\xi Q_1\|_{L^2}^2+\|\nabla_\xi Q_2\|_{L^2}^2\right)\dif t.
\end{aligned}
\end{equation}
Similarly, we obtain that
\begin{equation}\label{n3}
N_3\le \frac{\chi}{4}\int_0^t\int_{\mathbb{T}^2}\int_{\mathbb{R}}|\nabla_\xi\Psi|^2\dif \xi_1\dif \xi^{\prime}\dif t
+C\int_0^t\|\nabla_\xi Q_1\|_{L^2}^2\dif t,
\end{equation}
\begin{equation}\label{n4}
N_4\le C\delta^2\int_0^t|\dot X(t)|^2\dif t+C\delta\int_0^t\|\nabla_\xi \Psi\|_{L^2}^2\dif t,
\end{equation}
\begin{equation}\label{n5}
N_5\le C\delta\int_0^t\left(G_3(t)+G^s(t)+\|\nabla_\xi \Psi\|_{L^2}^2\right)\dif t,
\end{equation}
and
\begin{equation}\label{n6}
N_6\le\frac{\mu}{2}
\int_0^t\int_{\mathbb{T}^2}\int_{\mathbb{R}}|\nabla_\xi\Psi|^2\dif \xi_1\dif \xi^{\prime}\dif t
+\frac{1}{2\mu}
\int_0^t\int_{\mathbb{T}^2}\int_{\mathbb{R}}|Q_1|^2\dif \xi_1\dif \xi^{\prime}\dif t.
\end{equation}

Therefore, combining \eqref{n1}-\eqref{n6}, and using the smallness of $\delta$ and $\varepsilon_1$, we conclude that
\begin{equation}\label{hsu1}
\begin{aligned}
&(\frac{\mu}{2}-\frac{\chi}{2}-C(\delta+\varepsilon_1))\int_0^t\int_{\mathbb{T}^2}\int_{\mathbb{R}}|\nabla_\xi\Psi|^2\dif \xi_1\dif \xi^{\prime}\dif t
+\frac{\mu}{3}\int_0^t\int_{\mathbb{T}^2}\int_{\mathbb{R}}
   |\ddiv_\xi\Psi|^2
\dif \xi_1\dif \xi^{\prime}\dif t\\
&\le
\chi_1\tau\|\sqrt\rho Q_{1}\|_{L^2}^2
+C(\|\nabla_\xi\Psi\|_{L^2}^2+\tau\|Q_{10}\|_{L^2}^2+\|\nabla_\xi\Psi_0\|_{L^2}^2)
+C\delta^2\int_0^t|\dot X(t)|^2\dif t\\
&
+\frac{-\chi}{2}\int_0^t\int_{\mathbb{T}^2}\int_{\mathbb{R}}p^\prime(v)|\nabla_\xi\Phi|^2\dif \xi_1\dif \xi^{\prime}\dif t
+C\int_0^t\left(
\|\nabla_\xi Q_1\|_{L^2}^2+\|\nabla_\xi Q_2\|_{L^2}^2\right)\dif t\\
&\quad+C(\delta+\varepsilon_1)\int_0^t\left(G^s(t)+G_3(t)\right)\dif t
+\left(\frac{1}{2\mu}+C(\delta+\varepsilon_1)\right)\int_0^t\|Q_1\|_{L^2}^2\dif t.
\end{aligned}
\end{equation}

Similarly, multiplying equation $\eqref{sys1}_4$ by $\rho\ddiv_\xi\Psi$ and integrating over $(0, t)\times\mathbb{R}\times\mathbb{T}^2$, we have
\begin{equation}\label{hsu2}
\begin{aligned}
&\lambda\int_0^t\int_{\mathbb{T}^2}\int_{\mathbb{R}}
   |\ddiv_\xi\Psi|^2
\dif \xi_1\dif \xi^{\prime}\dif t\\
&\le
\chi_1\tau\|\sqrt\rho Q_{2}\|_{L^2}^2
+C(\|\nabla_\xi\Psi\|_{L^2}^2+\tau\|Q_{20}\|_{L^2}^2+\|\nabla_\xi\Psi_0\|_{L^2}^2)
+C\delta^2\int_0^t|\dot X(t)|^2\dif t\\
&\quad
\frac{-\chi}{2}\int_0^t\int_{\mathbb{T}^2}\int_{\mathbb{R}}p^\prime(v)|\nabla_\xi\Phi|^2\dif \xi_1\dif \xi^{\prime}\dif t
+(\frac{\chi}{2}+C(\delta+\varepsilon_1))\int_0^t\int_{\mathbb{T}^2}\int_{\mathbb{R}}|\nabla_\xi\Psi|^2\dif \xi_1\dif \xi^{\prime}\dif t\\
&\quad+C\int_0^t\left(
\|\nabla_\xi Q_1\|_{L^2}^2+\|\nabla_\xi Q_2\|_{L^2}^2\right)\dif t
+C(\delta+\varepsilon_1)\int_0^t\left(G^s(z)+G_3(z)\right)\dif t\\
&\quad+\left(\frac{1}{\lambda}+C(\delta+\varepsilon_1)\right)
\int_0^t\|Q_2\|_{L^2}^2\dif t.
\end{aligned}
\end{equation}

Therefore, choosing $\delta,\chi_1,\varepsilon_1$ sufficiently small, and combining \eqref{hsu1}, \eqref{hsu2} with Lemmas \ref{le0} and \ref{legj}, we have
\begin{equation}\label{hsgj10302}
	\begin{aligned}
		&(\frac{\mu}{2}-\chi-C(\delta+\varepsilon_1))\int_0^t\int_{\mathbb{T}^2}\int_{\mathbb{R}}|\nabla_\xi\Psi|^2\dif \xi_1\dif \xi^{\prime}\dif t
		+\left(\frac{\mu}{3}+\lambda\right)\int_0^t\int_{\mathbb{T}^2}\int_{\mathbb{R}}
		|\ddiv_\xi\Psi|^2
		\dif \xi_1\dif \xi^{\prime}\dif t\\
		&\le\Big(\chi+(\mu+\frac{3}{4}\lambda)(1+\chi+C(\nu+\varepsilon_1+\delta))\Big)\int_0^t\int_{\mathbb{T}^2}\int_{\mathbb{R}}|p^\prime(v)||\nabla_\xi\Phi|^2\dif \xi_1\dif \xi^{\prime}\dif t
		\\
		&\quad
		+C\|(\Phi_0, \Psi_0
		\sqrt{\tau}Q_{10}, \sqrt{\tau}Q_{20})\|_{H^1}^2.
	\end{aligned}
\end{equation}

Now, we firstly note that $\Delta_\xi u=\nabla_\xi\ddiv_{\xi}u-\nabla_{\xi}\times\nabla_{\xi}\times u$. Subsequently, multiplying \eqref{hsgj10301} and \eqref{hsgj10302} by $C_2(\frac{5}{6}\mu+\frac{5}{8}\lambda)$ and $C_3$, which are arbitrary positive constants,  respectively, and combining \eqref{ljgu1.1}, we derive that
\begin{align*}
	&C_5\int_0^t\int_{\mathbb{T}^2}\int_{\mathbb{R}}|p^\prime(v)||\nabla_\xi\Phi|^2\dif \xi_1\dif \xi^{\prime}\dif t+(\frac{\mu}{2}-\chi-C(\delta+\varepsilon_1))\int_0^t\int_{\mathbb{T}^2}\int_{\mathbb{R}}|\nabla_\xi\times\Psi|^2\dif \xi_1\dif \xi^{\prime}\dif t\\
&\quad	+C_6\int_0^t\int_{\mathbb{T}^2}\int_{\mathbb{R}}
	|\ddiv_\xi\Psi|^2
	\dif \xi_1\dif \xi^{\prime}\dif t\\
	&\le 
	\chi_1\|\sqrt\rho\Psi\|_{L^2}^2
	+C\|(\Phi_0, \Psi_0
	\sqrt{\tau}Q_{10}, \sqrt{\tau}Q_{20})\|_{H^1}^2.
	\\
	&\quad+C\delta^2\int_0^t|\dot X(t)|^2\dif t
	+C(\delta
	+\varepsilon_1)\int_0^t\left(G_3(t)+G^s(t)\right)\dif t+C\delta\int_0^t G_2(t)\dif t,
\end{align*}
where $C_5=\left(C_2(\frac{5}{6}\mu+\frac{5}{8}\lambda)(1-\chi-C(\delta+\varepsilon_1))-C_3\Big(\chi+(\frac{5}{6}\mu+\frac{5}{8}\lambda)(1+\chi+C(\nu+\varepsilon_1+\delta))\Big)\right)$ and $C_6=\left(C_3(\frac{5}{6}\mu+\lambda-\chi-C(\delta+\varepsilon_1))-C_2(\frac{5}{6}\mu+\frac{5}{8}\lambda)\right)$. Consequently, the condition $C_5, C_6>0$ if and only if $1>\frac{C_3}{C_2}>\frac{\frac{5}{6}\mu+\frac{5}{8}\lambda}{\frac{5}{6}\mu\lambda}$.

Thus, we get the desire result.

\end{proof}

Thus, choosing $\chi_1$ small enough, using the smallness of $\delta$ and $\varepsilon_1$, the desired result follows from an application of Lemma \ref{le0} and Lemma 2.1 in \cite{WY}.
\begin{corollary} \label{coro111}
	Under the hypotheses of Proposition \ref{p1}, there exists constant $C>0$ independent of $\tau, \nu, \delta, \varepsilon_1, T$, such that for $t\in[0, T]$, we have
	\begin{equation}\label{hsgjcoll}
		\begin{aligned}
			&\|v-(v^s)^{-X}\|_{L^2}^2+\|u-(u^s)^{-X}\|_{L^2}^2
			+\tau\|\Pi_1-(\Pi_1^s)^{-X}\|_{L^2}^2
			+\tau\|\Pi_2-(\Pi_2^s)^{-X}\|_{L^2}^2\\
			&\quad
			+\int_0^t\|(\Pi_1-(\Pi_1^s)^{-X}, \Pi_2-(\Pi_2^s)^{-X})\|_{L^2}\dif t
			+\int_0^t\|(\nabla_\xi(v-(v^s)^{-X}), \nabla_\xi(u-(u^s)^{-X}))\|_{L^2}\dif t\\
			&\quad+\delta\int_0^t|\dot{X}(t)|^2
			\dif t+\int_0^t(G_2(t)+G_3(t)+G^s(t))
			\dif t\\
			&\le
			C\|(v_{0}-v^s, u_1-u^s,
			\sqrt{\tau}(\Pi_{10}-\Pi_1^s), \sqrt{\tau}(\Pi_{20}-\Pi_2^s))\|_{H^1}^2,
		\end{aligned}
	\end{equation}
	where $G_2(t), G_3(t)$ are defined in Lemma \ref{le4.5} and $G^s(t)$ is defined in Lemma \ref{le0}, respectively.
\end{corollary}

 \begin{lemma}\label{lehs3}
Under the hypotheses of Proposition \ref{p1}, there exists constant $C>0$ independent of $\tau, \nu, \delta, \varepsilon_1, T$, such that for $t\in[0, T]$, we have
\begin{equation}\label{hsgj1.3}
\begin{aligned}
&\int_0^t\left(\|\nabla^2_\xi(v-(v^s)^{-X})\|_{H^1}^2
 +\|\nabla^2_\xi(u-(u^s)^{-X})\|_{H^1}^2\right)\dif t\\&\le 	C\|(v_{0}-v^s, u_1-u^s,
\sqrt{\tau}(\Pi_{10}-\Pi_1^s), \sqrt{\tau}(\Pi_{20}-\Pi_2^s))\|_{H^3}^2.
\end{aligned}
\end{equation}
\end{lemma}
\begin{proof}
Multiplying $\eqref{sys3}_2$ by $\partial_\xi^\alpha\nabla_\xi\Phi$ for $|\alpha|=1, 2$, and integrating over $(0, t)\times\mathbb{R}\times\mathbb{T}^2$, we have
\begin{align*}
&-\int_0^t\int_{\mathbb{T}^2}\int_{\mathbb{R}}vp^\prime(v)|\partial_\xi^\alpha\nabla_\xi\Phi|^2
\dif \xi_1\dif \xi^{\prime}\dif t\\
&=\int_0^t\int_{\mathbb{T}^2}\int_{\mathbb{R}}\partial_t\partial_\xi^\alpha\Psi\cdot\partial_\xi^\alpha\nabla_\xi\Phi
\dif \xi_1\dif \xi^{\prime}\dif t
-\sigma\int_0^t\int_{\mathbb{T}^2}\int_{\mathbb{R}}\partial_{\xi_1}\partial_\xi^\alpha\Psi\cdot\partial_\xi^\alpha\nabla_\xi\Phi
\dif \xi_1\dif \xi^{\prime}\dif t\\
&\quad+\int_0^t\int_{\mathbb{T}^2}\int_{\mathbb{R}}\partial_\xi^\alpha\left( u\cdot\nabla_\xi\Psi\right)\cdot\partial_\xi^\alpha\nabla_\xi\Phi
\dif \xi_1\dif \xi^{\prime}\dif t
-\int_0^t\int_{\mathbb{T}^2}\int_{\mathbb{R}}\dot{X}(t)\partial_{\xi_1}\partial_\xi^\alpha(u^s)^{-X}\cdot\partial_\xi^\alpha\nabla_\xi\Phi
\dif \xi_1\dif \xi^{\prime}\dif t\\
&\quad+\int_0^t\int_{\mathbb{T}^2}\int_{\mathbb{R}}
\left(\partial_\xi^\alpha\left(v\nabla_\xi\left(p(v)-p((v^s)^{-X})\right)\right)-vp^\prime(v)\partial_\xi^\alpha\nabla_\xi\Phi\right)\cdot\partial_\xi^\alpha\nabla_\xi\Phi
\dif \xi_1\dif \xi^{\prime}\dif t\\
&\quad+\int_0^t\int_{\mathbb{T}^2}\int_{\mathbb{R}}\partial_\xi^\alpha\left(vF\partial_{\xi_1}(u^s)^{-X}\right)\cdot\partial_\xi^\alpha\nabla_\xi\Phi
\dif \xi_1\dif \xi^{\prime}\dif t
-\int_0^t\int_{\mathbb{T}^2}\int_{\mathbb{R}}\partial_\xi^\alpha\left(v\ddiv_\xi Q_1\right)\cdot\partial_\xi^\alpha\nabla_\xi\Phi
\dif \xi_1\dif \xi^{\prime}\dif t\\
&\quad
-\int_0^t\int_{\mathbb{T}^2}\int_{\mathbb{R}}\partial_\xi^\alpha\left(v\nabla_\xi Q_2\right)\cdot\partial_\xi^\alpha\nabla_\xi\Phi
\dif \xi_1\dif \xi^{\prime}\dif t=:\sum\limits_{i=1}\limits^8A_i.
\end{align*}
For $A_1$, applying equation $\eqref{sys3}_1$ and performing straightforward computation yields
\begin{align*}
A_1=&\int_{\mathbb{T}^2}\int_{\mathbb{R}}\left(\partial_\xi^\alpha\Psi\cdot\partial_\xi^\alpha\nabla_\xi\Phi
-\partial_\xi^\alpha\Psi_0\cdot\partial_\xi^\alpha\nabla_\xi\Phi_0\right)
\dif \xi_1\dif \xi^{\prime}
-\sigma\int_0^t\int_{\mathbb{T}^2}\int_{\mathbb{R}}
\partial_\xi^\alpha\Psi\cdot\partial_{\xi_1}\partial_\xi^\alpha\nabla_\xi\Phi
\dif \xi_1\dif \xi^{\prime}\dif t\\
&
+\int_0^t\int_{\mathbb{T}^2}\int_{\mathbb{R}}
\partial_\xi^\alpha\Psi\cdot\partial_\xi^\alpha\nabla_\xi\left(u\cdot\nabla_\xi\Phi\right)
\dif \xi_1\dif \xi^{\prime}\dif t
-\int_0^t\int_{\mathbb{T}^2}\int_{\mathbb{R}}\dot X(t)
\partial_\xi^\alpha\Psi\cdot\partial_{\xi_1}\partial_\xi^\alpha\nabla_\xi(v^s)^{-X}
\dif \xi_1\dif \xi^{\prime}\dif t\\
&
+\int_0^t\int_{\mathbb{T}^2}\int_{\mathbb{R}}
\partial_\xi^\alpha\Psi\cdot\partial_\xi^\alpha\nabla_\xi\left(vF\partial_{\xi_1}(v^s)^{-X}\right)
\dif \xi_1\dif \xi^{\prime}\dif t
-\int_0^t\int_{\mathbb{T}^2}\int_{\mathbb{R}}
\partial_\xi^\alpha\Psi\cdot\partial_\xi^\alpha\nabla_\xi\left(v\ddiv_\xi\Psi\right)
\dif \xi_1\dif \xi^{\prime}\dif t\\
=:&\sum\limits_{j=1}\limits^6A_{1j}.
\end{align*}
Using Young's inequality, we deduce that
\[
A_{11}\le \frac{1}{2}\left(\|\partial_\xi^\alpha\Psi\|_{L^2}^2+\|\partial_\xi^\alpha\nabla_\xi\Phi\|_{L^2}^2
+\|\partial_\xi^\alpha\Psi_0\|_{L^2}^2+\|\partial_\xi^\alpha\nabla_\xi\Phi_0\|_{L^2}^2\right).
\]
By integration by parts, we get
\[
A_{12}=\sigma\int_0^t\int_{\mathbb{T}^2}\int_{\mathbb{R}}
\partial_{\xi_1}\partial_\xi^\alpha\Psi\cdot\partial_\xi^\alpha\nabla_\xi\Phi
\dif \xi_1\dif \xi^{\prime}\dif t
=-A_2.
\]
Next, using similar estimates as in Lemma \ref{legj}, we obtain that
\begin{align*}
A_{13}\le C(\varepsilon_1+\delta)\int_0^t\left(\|\partial_\xi^\alpha\Psi\|_{H^1}^2+\|\nabla_\xi\Phi\|_{H^{|\alpha|}}^2\right)\dif t
+\int_0^t\int_{\mathbb{T}^2}\int_{\mathbb{R}}\partial_\xi^\alpha\Psi\cdot\left(u\cdot\partial_\xi^\alpha\nabla_\xi^2\Phi\right)
\dif \xi_1\dif \xi^{\prime}\dif t,
\end{align*}
\begin{align*}
A_{14}
\le C\delta^2\int_0^t|\dot X(t)|^2\dif t+C\delta\int_0^t\|\partial_\xi^\alpha\Psi\|_{L^2}^2\dif t,
\end{align*}
\begin{align*}
A_{15}\le C\delta\int_0^t\left(\|\nabla_\xi\Psi\|_{H^{|\alpha|}}^2+\|\nabla_\xi\Phi\|_{H^{|\alpha|}}^2+G_3(t)+G^s(t)\right)
\dif t,
\end{align*}
and
\begin{align*}
A_{16}
\le \int_0^t\int_{\mathbb{T}^2}\int_{\mathbb{R}}
v|\partial_\xi^\alpha\ddiv_\xi\Psi|^2
\dif \xi_1\dif \xi^{\prime}\dif t
+C(\delta+\varepsilon_1)\int_0^t\|\nabla_\xi\Psi\|_{H^{|\alpha|}}^2\dif t.
\end{align*}
Therefore, combining above estimates, we have
\begin{equation}\label{a1a2}
\begin{aligned}
A_1&+A_2\le \frac{1}{2}\left(\|\partial_\xi^\alpha\Psi\|_{L^2}^2+\|\partial_\xi^\alpha\nabla_\xi\Phi\|_{L^2}^2
+\|\partial_\xi^\alpha\Psi_0\|_{L^2}^2+\|\partial_\xi^\alpha\nabla_\xi\Phi_0\|_{L^2}^2\right)
+C\delta^2\int_0^t|\dot X(t)|^2\dif t\\
&
+C(\varepsilon_1+\delta)\int_0^t\left(\|\nabla_\xi\Psi\|_{H^{|\alpha|}}^2+\|\nabla_\xi\Phi\|_{H^{|\alpha|}}^2\right)\dif t
+\int_0^t\int_{\mathbb{T}^2}\int_{\mathbb{R}}
v|\partial_\xi^\alpha\ddiv_\xi\Psi|^2
\dif \xi_1\dif \xi^{\prime}\dif t
\\
&+C\delta\int_0^t\left(G_3(t)+G^s(t)\right)
\dif t
+\underbrace{\int_0^t\int_{\mathbb{T}^2}\int_{\mathbb{R}}\partial_\xi^\alpha\Psi\cdot\left(u\cdot\partial_\xi^\alpha\nabla_\xi^2\Phi\right)
\dif \xi_1\dif \xi^{\prime}\dif t}\limits_{=:A_{1\ast}}.
\end{aligned}
\end{equation}
Similarly, for $A_3$, we get
\begin{equation}\label{a3}
\begin{aligned}
A_3
\le C(\varepsilon_1+\delta)
\int_0^t\|(\nabla_\xi\Psi, \nabla_\xi\Phi)\|_{H^{|\alpha|}}^2\dif t
+\underbrace{\int_0^t\int_{\mathbb{T}^2}\int_{\mathbb{R}}\left( u\cdot\partial_\xi^\alpha\nabla_\xi\Psi\right)\cdot\partial_\xi^\alpha\nabla_\xi\Phi
\dif \xi_1\dif \xi^{\prime}\dif t}\limits_{=:A_{3\ast}},
\end{aligned}
\end{equation}
\begin{equation}\label{a4}
A_4\le C\delta^2\int_{0}^{t}|\dot{X}(t)|^2\dif t+C\delta\int_0^t\|\partial_\xi^\alpha\nabla_\xi\Phi\|_{L^2}^2\dif t,
\end{equation}
\begin{equation}\label{a5}
A_5\le C(\varepsilon_1+\delta)\int_0^t\left(G^s(z)+\|\nabla_\xi\Phi\|_{H^{|\alpha|}}^2\right)\dif t,
\end{equation}
\begin{equation}\label{a6}
\begin{aligned}
A_6\le C\delta\int_0^t\left(\|\nabla_\xi\Psi\|_{H^{|\alpha|-1}}^2+\|\nabla_\xi\Phi\|_{H^{|\alpha|}}^2+G_3(t)+G^s(t)\right)
\dif t,
\end{aligned}
\end{equation}
For $A_7$ and $A_8$, using Young's inequality, we get
\begin{equation}\label{a8}
\begin{aligned}
A_7+A_8\le-\frac{1}{8}\int_0^t\int_{\mathbb{T}^2}\int_{\mathbb{R}}vp^\prime(v)|\partial_\xi^\alpha\nabla_\xi\Phi|^2
\dif \xi_1\dif \xi^{\prime}\dif t+C\int_0^t\|(\nabla_\xi Q_1, \nabla_\xi Q_2)\|_{H^{|\alpha|}}^2\dif t
.
\end{aligned}
\end{equation}
In addition, using Lemma \ref{pvsw}, we get
\begin{equation}\label{a9}
\begin{aligned}
A_{1\ast}+A_{3\ast}\le C(\delta+\varepsilon_1)\int_0^t\left(\|\partial_\xi^\alpha\Psi\|_{L^2}^2+\|\partial_\xi^\alpha\nabla_\xi\Phi\|_{L^2}^2\right)
\dif t.
\end{aligned}
\end{equation}

Therefore, combining \eqref{a1a2}-\eqref{a9} and choosing $\delta$ and $\varepsilon_1$ sufficiently small, we conclude that
\begin{equation}\label{gjhsgj3}
\begin{aligned}
&\frac{3}{4}\int_0^t\|\sqrt{|vp^\prime(v)}|\nabla_\xi^2\Phi\|_{H^1}^2
\dif \xi_1\dif \xi^{\prime}\dif t\\
&\le \frac{1}{2}\left(\|\nabla_\xi\Psi\|_{H^1}^2+\|\nabla_\xi^2\Phi\|_{H^1}^2
+\|\nabla_\xi\Psi_0\|_{H^1}^2+\|\nabla_\xi^2\Phi_0\|_{H^1}^2\right)
+C\delta^2\int_0^t|\dot X(t)|^2\dif t\\
&\quad
+C(\varepsilon_1+\delta)\int_0^t\left(\|\nabla_\xi\Psi\|_{H^2}^2+\|\nabla_\xi\Phi\|_{L^2}^2
+G^s(t)\right)\dif t
+C\delta\int_0^t
G_3(t)\dif t
\\
&\quad+\int_0^t\|
\sqrt{v}\nabla_\xi\ddiv_\xi\Psi\|_{H^1}^2
\dif t+C\int_0^t\left(\|\nabla_\xi Q_1\|_{H^2}^2+\|\nabla_\xi Q_2\|_{H^2}^2\right)
\dif t
\end{aligned}
\end{equation}

 Multiplying equation $\eqref{sys3}_3$ by $\partial_\xi^\alpha\nabla_\xi\Psi$ for $|\alpha|=1,2$, and integrating over $(0, t)\times\mathbb{R}\times\mathbb{T}^2$, we have
\begin{align*}
&\mu\int_0^t\int_{\mathbb{T}^2}\int_{\mathbb{R}}v|\partial_\xi^\alpha\nabla_\xi\Psi|^2\dif \xi_1\dif \xi^{\prime}\dif t
\\
=&\tau\int_0^t\int_{\mathbb{T}^2}\int_{\mathbb{R}}
\partial_t\partial_\xi^\alpha Q_1
:\partial_\xi^\alpha\nabla_\xi\Psi\dif \xi_1\dif \xi^{\prime}\dif t
-\tau\sigma\int_0^t\int_{\mathbb{T}^2}\int_{\mathbb{R}}
\partial_{\xi_1}\partial_\xi^\alpha Q_1
:\partial_\xi^\alpha\nabla_\xi\Psi\dif \xi_1\dif \xi^{\prime}\dif t\\
&+\tau\int_0^t\int_{\mathbb{T}^2}\int_{\mathbb{R}}
        \partial_\xi^\alpha ( u\cdot\nabla_\xi Q_1)
:\partial_\xi^\alpha\nabla_\xi\Psi\dif \xi_1\dif \xi^{\prime}\dif t
-\tau\int_0^t\int_{\mathbb{T}^2}\int_{\mathbb{R}}
     \dot{X}(t)\partial_{\xi_1}\partial_\xi^\alpha(\Pi_1^s)^{-X}
:\partial_\xi^\alpha\nabla_\xi\Psi\dif \xi_1\dif \xi^{\prime}\dif t\\
&+\tau\int_0^t\int_{\mathbb{T}^2}\int_{\mathbb{R}}
        \partial_\xi^\alpha(vF\partial_{\xi_1}(\Pi_1^s)^{-X})
:\partial_\xi^\alpha\nabla_\xi\Psi\dif \xi_1\dif \xi^{\prime}\dif t
+\int_0^t\int_{\mathbb{T}^2}\int_{\mathbb{R}}
    \partial_\xi^\alpha(v  Q_1)
:\partial_\xi^\alpha\nabla_\xi\Psi\dif \xi_1\dif \xi^{\prime}\dif t\\
&-\mu\int_0^t\int_{\mathbb{T}^2}\int_{\mathbb{R}}
\partial_\xi^\alpha\left(v\left(\nabla_\xi\Psi\right)^T
-\frac{2v}{3}\left(\ddiv_\xi\Psi I_3\right)\right):\partial_\xi^\alpha\nabla_\xi\Psi
\dif \xi_1\dif \xi^{\prime}\dif t=:\sum\limits_{i=1}\limits^7W_i.
\end{align*}
For $W_1$, by applying integration by parts and using $\eqref{sys1}_2$, we get
\begin{align*}
W_1
=&\tau\int_{\mathbb{T}^2}\int_{\mathbb{R}}
\left(\partial_\xi^\alpha Q_1
:\partial_\xi^\alpha\nabla_\xi\Psi
-\partial_\xi^\alpha Q_{10}
:\partial_\xi^\alpha\nabla_\xi\Psi_0\right)\dif \xi_1\dif \xi^{\prime}
-\tau\sigma\int_0^t\int_{\mathbb{T}^2}\int_{\mathbb{R}}
\partial_\xi^\alpha Q_1: \partial_{\xi_1}\partial_\xi^\alpha\nabla_\xi\Psi
\dif \xi_1\dif \xi^{\prime}\dif t\\
&+\tau\int_0^t\int_{\mathbb{T}^2}\int_{\mathbb{R}}
\partial_\xi^\alpha Q_1: \partial_\xi^\alpha\nabla_\xi(u\cdot\nabla_\xi\Psi)
        \dif \xi_1\dif \xi^{\prime}\dif t
+\tau\int_0^t\int_{\mathbb{T}^2}\int_{\mathbb{R}}
\partial_\xi^\alpha Q_1: (\dot{X}(t)\partial_{\xi_1}\partial_\xi^\alpha\nabla_\xi(u^s)^{-X})
    \dif \xi_1\dif \xi^{\prime}\dif t\\
&+\tau\int_0^t\int_{\mathbb{T}^2}\int_{\mathbb{R}}
      \partial_\xi^\alpha Q_1: \partial_\xi^\alpha\nabla_\xi(vF\partial_{\xi_1}(u^s)^{-X})
        \dif \xi_1\dif \xi^{\prime}\dif t\\
&-\tau\int_0^t\int_{\mathbb{T}^2}\int_{\mathbb{R}}
 \partial_\xi^\alpha Q_1: \partial_\xi^\alpha\nabla_\xi(v\nabla_\xi\left(p(v)-p((v^s)^{-X})\right))
    \dif \xi_1\dif \xi^{\prime}\dif t\\
    &+\tau\int_0^t\int_{\mathbb{T}^2}\int_{\mathbb{R}}
      \partial_\xi^\alpha Q_1: \partial_\xi^\alpha\nabla_\xi(v\ddiv_\xi Q_1)
        \dif \xi_1\dif \xi^{\prime}\dif t
-\tau\int_0^t\int_{\mathbb{T}^2}\int_{\mathbb{R}}
  \partial_\xi^\alpha Q_1: \partial_\xi^\alpha\nabla_\xi(v\nabla_\xi Q_2)
    \dif \xi_1\dif \xi^{\prime}\dif t\\
    &=:\sum\limits_{j=1}\limits^8W_{1j},
\end{align*}
Using the same method as for estimating $A_1$ and $A_5$, we imply that
\[
W_{11}\le \frac{1}{2}\left(\tau\|\partial_\xi^\alpha Q_{1}\|_{L^2}^2+\|\partial_\xi^\alpha\nabla_\xi\Psi\|_{L^2}^2
+\tau\|\partial_\xi^\alpha Q_{10}\|_{L^2}^2+\|\partial_\xi^\alpha\nabla_\xi\Psi_0\|_{L^2}^2\right),
\]
\[
W_{12}=\tau\sigma\int_0^t\int_{\mathbb{T}^2}\int_{\mathbb{R}}
\partial_{\xi_1}\partial_\xi^\alpha Q_1
:\partial_\xi^\alpha\nabla_\xi\Psi\dif \xi_1\dif \xi^{\prime}\dif t=-W_2,
\]
\begin{align*}
W_{13}&\le
 \frac{\mu}{24}\int_0^t\int_{\mathbb{T}^2}\int_{\mathbb{R}}v|\partial_\xi^\alpha\nabla_\xi\Psi|^2\dif \xi_1\dif \xi^{\prime}\dif t
+C\int_0^t\|\partial_\xi^\alpha Q_1\|_{H^1}^2\dif t
+C(\delta+\varepsilon_1)\int_0^t\|\nabla^k_\xi \Psi\|_{H^{|\alpha|-1}}^2\dif t,
\end{align*}
\[
W_{14}\le C\delta^2\int_0^t|\dot X(t)|^2\dif t+C\delta\int_0^t\|\partial_\xi^\alpha Q_1\|_{L^2}^2\dif t,
\]
\[
W_{15}\le C\delta\int_0^t\left(G_3(t)+G^s(t)+\|\partial_\xi^\alpha Q_1\|_{L^2}^2
+\|\nabla_\xi\Psi\|_{H^{|\alpha|}}^2+\|\nabla_\xi\Phi\|_{H^{|\alpha|}}^2\right)\dif t,
\]
\begin{align*}
W_{16}&\le \frac{-\mu}{200}\int_0^t\int_{\mathbb{T}^2}\int_{\mathbb{R}}vp^\prime(v)|\partial_\xi^\alpha\nabla_\xi\Phi|^2\dif \xi_1\dif \xi^{\prime}\dif t
    +C\int_0^t\|\partial_\xi^\alpha Q_1\|_{H^1}^2\dif t\\
    &\quad+C(\delta+\varepsilon_1)\int_0^t\left(G^s(z)+\|\nabla_\xi \Phi\|_{H^{|\alpha|-1}}^2\right)\dif t,
  \end{align*}
and
\[
 W_{17} \le C\int_0^t\|\nabla_\xi Q_1\|_{H^{|\alpha|}}^2\dif t,
 \quad
  W_{18} \le C\int_0^t\left(\|\partial_\xi^\alpha Q_1\|_{H^1}^2+\|\nabla_\xi Q_2\|_{H^{|\alpha|}}^2\right)\dif t.
\]
Thus, combining the above estimates, we derive that
\begin{equation}\label{w1}
\begin{aligned}
W_1+W_2\le& \frac{1}{2}\left(\tau\|\partial_\xi^\alpha Q_{1}\|_{L^2}^2+\|\partial_\xi^\alpha\nabla_\xi\Psi\|_{L^2}^2
+\tau\|\partial_\xi^\alpha Q_{10}\|_{L^2}^2+\|\partial_\xi^\alpha\nabla_\xi\Psi_0\|_{L^2}^2\right)
+C\delta^2\int_0^t|\dot X(t)|^2\dif t\\
&
+\frac{\mu}{24}\int_0^t\int_{\mathbb{T}^2}\int_{\mathbb{R}}v|\partial_\xi^\alpha\nabla_\xi\Psi|^2\dif \xi_1\dif \xi^{\prime}\dif t
+\frac{-\mu}{200}\int_0^t\int_{\mathbb{T}^2}\int_{\mathbb{R}}vp^\prime(v)|\partial_\xi^\alpha\nabla_\xi\Phi|^2\dif \xi_1\dif \xi^{\prime}\dif t\\
&+C(\delta+\varepsilon_1)\int_0^t\left(G_3(t)+G^s(t)+\|\nabla_\xi \Phi\|_{H^{|\alpha|}}^2+\|\nabla_\xi \Psi\|_{H^{|\alpha|}}^2\right)\dif t\\
&\quad+C\int_0^t\left(\|\nabla_\xi Q_1\|_{H^{|\alpha|}}^2+\|\nabla_\xi Q_2\|_{H^{\alpha}}^2\right)\dif t.
\end{aligned}
\end{equation}
Similarly, we obtain that
\begin{equation}\label{w3}
W_3\le \frac{\mu}{24}\int_0^t\int_{\mathbb{T}^2}\int_{\mathbb{R}}v|\partial_\xi^\alpha\nabla_\xi\Psi|^2\dif \xi_1\dif \xi^{\prime}\dif t
+C\int_0^t\|\partial_\xi^\alpha Q_1\|_{H^{1}}^2\dif t,
\end{equation}
\begin{equation}\label{w4}
W_4\le C\delta^2\int_0^t|\dot X(t)|^2\dif t+C\delta\int_0^t\|\partial_\xi^\alpha\nabla_\xi \Psi\|_{L^2}^2\dif t,
\end{equation}
\begin{equation}\label{w5}
W_5\le  C\delta\int_0^t\left(G_3(t)+G^s(t)
+\|\nabla_\xi\Psi\|_{H^{|\alpha|}}^2+\|\nabla_\xi\Phi\|_{H^{|\alpha|-1}}^2\right)\dif t,
\end{equation}
and
\begin{equation}\label{w6}
W_6\le \frac{\mu}{24}\int_0^t\int_{\mathbb{T}^2}\int_{\mathbb{R}}v|\partial_\xi^\alpha\nabla_\xi\Psi|^2\dif \xi_1\dif \xi^{\prime}\dif t
+C\int_0^t\|\nabla_\xi Q_1\|_{H^{|\alpha|-1}}^2\dif t.
\end{equation}
For $W_7$, first of all, straightforward computation yields
\begin{align*}
	&{\mu}\left(v\partial_\xi^\alpha\left(\nabla_\xi\Psi\right)^T
	-\frac{2v}{3}\partial_\xi^\alpha\left(\ddiv_\xi\Psi I_3\right)\right):\partial_\xi^\alpha\nabla_\xi\Psi
	=
	\mu\sum\limits_{i=1}\limits^3\sum\limits_{j=1}\limits^3
	\left(
	v \partial_\xi^\alpha\Psi_{i\xi_j}\cdot\partial_\xi^\alpha\Psi_{j\xi_i}
	-\frac{2v}{3} \partial_\xi^\alpha\Psi_{i\xi_i}\cdot\partial_\xi^\alpha\Psi_{j\xi_j}
	\right).
\end{align*}
By integration by parts, it holds
\begin{align*}
	&\int_0^t\int_{\mathbb{T}^2}\int_{\mathbb{R}}v \partial_\xi^\alpha\Psi_{i\xi_j}\cdot\partial_\xi^\alpha\Psi_{j\xi_i}
	\dif \xi_1\dif \xi^{\prime}\dif t\\
	&=
	-\int_0^t\int_{\mathbb{T}^2}\int_{\mathbb{R}}v_{\xi_i} \partial_\xi^\alpha\Psi_{i\xi_j}\cdot\partial_\xi^\alpha\Psi_{j}
	\dif \xi_1\dif \xi^{\prime}\dif t
	+\int_0^t\int_{\mathbb{T}^2}\int_{\mathbb{R}}v_{\xi_j} \partial_\xi^\alpha\Psi_{i\xi_i}\cdot\partial_\xi^\alpha\Psi_{j}
	\dif \xi_1\dif \xi^{\prime}\dif t\\
	&\quad+\int_0^t\int_{\mathbb{T}^2}\int_{\mathbb{R}}v \partial_\xi^\alpha\Psi_{i\xi_i}\cdot\partial_\xi^\alpha\Psi_{j\xi_j}
	\dif \xi_1\dif \xi^{\prime}\dif t.
\end{align*}
So,
\begin{align*}
	\sum\limits_{i=1}\limits^3\sum\limits_{j=1}\limits^3
	\int_0^t\int_{\mathbb{T}^2}\int_{\mathbb{R}}v \partial_\xi^\alpha\Psi_{i\xi_i}\cdot\partial_\xi^\alpha\Psi_{j\xi_j}
	\dif \xi_1\dif \xi^{\prime}\dif t
	=\int_0^t
	\|\sqrt{v} \partial_\xi^\alpha\ddiv_\xi\Psi\|_{L^2}^2
	\dif t
	,
\end{align*}
and
	\begin{align*}
		&\sum\limits_{i=1}\limits^3\sum\limits_{j=1}\limits^3
		\left(-\int_0^t\int_{\mathbb{T}^2}\int_{\mathbb{R}}v_{\xi_i} \partial_\xi^\alpha\Psi_{i\xi_j}\cdot\partial_\xi^\alpha\Psi_{j}
		\dif \xi_1\dif \xi^{\prime}\dif t
		+\int_0^t\int_{\mathbb{T}^2}\int_{\mathbb{R}}v_{\xi_j} \partial_\xi^\alpha\Psi_{i\xi_i}\cdot\partial_\xi^\alpha\Psi_{j}
		\dif \xi_1\dif \xi^{\prime}\dif t\right)\\
		&\le C(\delta+\varepsilon_1)
		\int_0^t \|\nabla_\xi\Psi\|^2_{H^{|\alpha|-1}}
		\dif t
		.
	\end{align*}
In addition,
	\begin{align*}
		&\mu\int_0^t\int_{\mathbb{T}^2}\int_{\mathbb{R}}
		\partial_\xi^\alpha\left(v\left(\nabla_\xi\Psi\right)^T
		-\frac{2v}{3}\left(\ddiv_\xi\Psi I_3\right)\right):\partial_\xi^\alpha\nabla_\xi\Psi
		\dif \xi_1\dif \xi^{\prime}\dif t\\
		&\quad-\mu\int_0^t\int_{\mathbb{T}^2}\int_{\mathbb{R}}
		\left(v\partial_\xi^\alpha\left(\nabla_\xi\Psi\right)^T
		-\frac{2v}{3}\partial_\xi^\alpha\left(\ddiv_\xi\Psi I_3\right)\right):\partial_\xi^\alpha\nabla_\xi\Psi
		\dif \xi_1\dif \xi^{\prime}\dif t\\
		&\le
		C(\delta+\varepsilon_1)
		\int_0^t \|\nabla_\xi\Psi\|^2_{H^{|\alpha|}}
		\dif t.
	\end{align*}
Thus, combining the above estimates, we derive that
\begin{equation}\label{w7}
W_7\le -\int_0^t
\|\sqrt{v} \partial_\xi^\alpha\ddiv_\xi\Psi\|_{L^2}^2
\dif t+C(\delta+\varepsilon_1)
\int_0^t \|\nabla_\xi\Psi\|^2_{H^{|\alpha|}}
\dif t.
\end{equation}

Therefore, combining \eqref{w1}-\eqref{w7} and using the smallness of $\delta$ and $\varepsilon_1$, we conclude that
\begin{equation}\label{gjhsgj4}
\begin{aligned}
&\frac{\mu}{2}\int_0^t\|\sqrt{v}\nabla^2_\xi\Psi\|_{H^1}^2\dif t
+\frac{\mu}{3}\int_0^t\|
\sqrt{v}\nabla_\xi\ddiv_\xi\Psi\|_{H^1}^2
\dif t\\
&\le
\frac{1}{2}\left(\tau\|\nabla_\xi Q_{1}\|_{L^2}^2+\|\nabla^2_\xi\Psi\|_{L^2}^2
+\tau\|\nabla_\xi Q_{10}\|_{L^2}^2+\|\nabla^2_\xi\Psi_0\|_{L^2}^2\right)
+C\delta^2\int_0^t|\dot X(t)|^2\dif t\\
&\quad
+\frac{-\mu}{50}\int_0^t\int_{\mathbb{T}^2}\int_{\mathbb{R}}\|\sqrt{|vp^\prime(v)|}\nabla^2_\xi\Phi\|_{H^1}^2\dif \xi_1\dif \xi^{\prime}\dif t
+C\int_0^t\left(\|\nabla_\xi Q_1\|_{H^{2}}^2+\|\nabla_\xi Q_2\|_{H^{2}}^2\right)\dif t\\
&\quad+C(\delta+\varepsilon_1)\int_0^t\left(G^s(t)+\|\nabla_\xi \Phi\|_{L^2}^2+\|\nabla_\xi \Psi\|_{L^2}^2\right)\dif t
+C\delta\int_0^tG_3(t)\dif t.
\end{aligned}
\end{equation}

Now, multiplying \eqref{gjhsgj3} and \eqref{gjhsgj4} by $\frac{\mu}{6}$ and $1$, respectively, and using the smallness of $\delta, \varepsilon_1$ together with Lemma \ref{legj}, we derive that
\begin{equation}\label{gjhsgj6}
\begin{aligned}
&\frac{\mu}{16}\int_0^t\|\sqrt{|vp^\prime(v)}|\nabla_\xi^2\Phi\|_{H^1}^2
\dif \xi_1\dif \xi^{\prime}\dif t
+\frac{\mu}{4}\int_0^t\|\sqrt{v}\nabla^2_\xi\Psi\|_{H^1}^2\dif t
+\frac{\mu}{6}\int_0^t\|
\sqrt{v}\nabla_\xi\ddiv_\xi\Psi\|_{H^1}^2
\dif t\\
&\le C\Big(\|\nabla_\xi\Phi_0\|_{H^2}^2+\|\nabla_\xi\Psi_0\|_{H^2}^2
+\tau\|\nabla_\xi Q_{10}\|_{H^2}^2
+\tau\|\nabla_\xi Q_{20}\|_{H^2}^2\Big)+C\delta^2\int_0^t|\dot X(t)|^2\dif t\\
&\quad+C(\varepsilon_1+\delta)\int_0^t
\left(\|\left(\nabla_\xi\Phi\|_{L^2}^2+\| \nabla_\xi\Psi\right)\|_{L^2}^2
\right)
\dif t
+C(\varepsilon_1+\delta)\int_0^t
G^s(t)
\dif t
+C\delta\int_0^t
G_3(t)
\dif t.
\end{aligned}
\end{equation}

Thus, combining Corollary \ref{coro111}, the proof of this is finished.
\end{proof}

Under the smallness of $\delta$ and $\varepsilon_1$, the a priori estimates follows from Lemmas \ref{legj} and \ref{lehs3}, together with Corollary \ref{coro111}.

\section{Proof of Theorem \ref{th1.2}}
In this section, we prove the Theorem \ref{th1.2} by use of the uniform estimates obtained in Section 4 and usual compactness arguments.
Firstly, according to Theorem \ref{th1.1}, we have
\begin{align*}
\sup_{0\le t<+\infty}\|(\Phi^{\tau}, \Psi^{\tau}, \sqrt{\tau}Q_1^{\tau}, \sqrt{\tau}Q_2^{\tau})  (t,\cdot)\|_{H^2(\Omega)}^2+\int_0^{+\infty}\left(\|(\Phi^{\tau}_{x}, \Psi^{\tau}_{x})\|_{H^1(\Omega)}^2+\|(Q_1^{\tau}, Q_2^{\tau})\|_{H^2(\Omega)}^2\right)\dif t
\le C_0E(0),
\end{align*}
where $
E(0)=\|(\Phi^\tau, \Psi^\tau, \sqrt{\tau}Q_1^\tau, \sqrt{\tau}Q_2^\tau)(0, \cdot)\|_{H^2(\Omega)}$, $\Phi^{\tau}=v^{\tau}- (v^s)^{\tau}, \Psi^{\tau}=u^{\tau}- (u^s)^{\tau}, Q_1^{\tau}=\Pi_1^{\tau}- (\Pi^s_1)^{\tau}, Q_2^{\tau}=\Pi_2^{\tau}- (\Pi^s_2)^{\tau}$, $C_0$ is a constant independent of $\tau$ and $ (v^s)^{\tau}, (u^s)^{\tau}, (\Pi^s_1)^{\tau}, (\Pi^s_2)^{\tau}$ are the planar viscous waves of system \eqref{1.6}, respectively. Thus, there exist $(\Phi^0, \Psi^0)\in L^{\infty}((0,\infty);H^3(\Omega))$ and $Q^0\in L^2((0, \infty);H^3(\Omega))$ such that
\begin{equation}\label{5.1}
\begin{aligned}
(\Phi^{\tau}, \Psi^{\tau})\rightharpoonup(\Phi^0, \Psi^0)\qquad weak-* \quad in \quad L^{\infty}((0,\infty);H^3(\Omega)),\\
Q^{\tau}\rightharpoonup Q^0 \qquad weakly- \quad in \quad  L^2((0, \infty);H^3(\Omega)).
\end{aligned}
\end{equation}

Secondly, from Lemma \ref{pvsw}, we get
\[
\begin{aligned}
&\|(v^s)^{\tau}-v_-\|_{L^2(\mathbb{R}^-)}+\|( v^s)^{\tau}-v_+\|_{L^2(\mathbb{R}^+)}
+\|( v^s)^{\tau}_{x}\|_{H^3(\mathbb R)}\le C,\\
&\|(u^s)^{\tau}-u_-\|_{L^2(\mathbb{R}^-)}+\|( u^s)^{\tau}-u_+\|_{L^2(\mathbb{R}^+)}
+\|( u^s)^{\tau}_{x}\|_{H^3(\mathbb R)}\le C,\\
&\|( \Pi^s_1)^{\tau}\|_{H^3(\mathbb R)}\le C,\qquad \|( \Pi^s_2)^{\tau}\|_{H^3(\mathbb R)}\le C,
\end{aligned}
\]
where $C$ independent of $\tau$.

Then, using compactness theorem, for any $T>0$, we have
\[
\begin{aligned}
( v^s)^{\tau}\rightarrow (v^s)^0,\quad( u^s)^{\tau}\rightarrow( u^s)^0,\qquad strongly\quad in \quad C([0,T]; H_{loc}^3(\mathbb R)),\\
( \Pi_1^s)^{\tau}\rightharpoonup( \Pi_1^s)^0, \quad( \Pi_2^s)^{\tau}\rightharpoonup( \Pi_2^s)^0,
\qquad weakly-\quad in \quad L^\infty((0,\infty), H^3(\mathbb R)).
\end{aligned}
\]
In addition, let $\tau\rightarrow0$ in \eqref{2.6}, we have $\tau\sigma_{\ast} (\Pi^s_{11})^\tau\rightarrow 0$, $\tau\sigma_{\ast} (\Pi^s_2)^\tau\rightarrow 0$ in $D^\prime((0,\infty)\times \mathbb R)$ and
\[
\begin{aligned}
( \Pi_1^s)^{\tau}\rightharpoonup\mu\left(\nabla_x (u^s)^0+(\nabla_x (u^s)^0)^T-\frac{2}{3}\ddiv_x(u^s)^0 \mathrm{I}_3\right)
:=( \Pi_1^s)^{0},\qquad  in \quad \mathcal D^\prime((0,\infty)\times \mathbb R),\\
( \Pi_2^s)^{\tau}\rightharpoonup\lambda\ddiv_x(u^s)^0
:=( \Pi^s_2)^{0},\qquad  in \quad \mathcal D^\prime((0,\infty)\times \mathbb R),\\
\end{aligned}
\]
and we know that $( v^s)^0,( u^s)^0$ are the planar viscous wave solutions of 3-D classical Navier-Stokes equations.
Therefore, for any $T>0$, we have
\begin{equation}\label{5.2}
\begin{aligned}
 (v^s)^{\tau}\rightarrow( v^s)^{0}, \quad strongly\quad in \quad C([0, T]; H_{loc}^3(\mathbb R)),\\
  (u^s)^{\tau}\rightarrow( u^s)^{0}, \quad strongly\quad in \quad C([0, T]; H_{loc}^3(\mathbb R)),\\
   (\Pi^s_1)^{\tau}\rightharpoonup(\Pi_1^s)^{0},\qquad weakly- \quad in \quad L^\infty((0, \infty); H_{loc}^3(\mathbb R)),\\
   (\Pi^s_2)^{\tau}\rightharpoonup(\Pi_2^s)^{0},\qquad weakly- \quad in \quad L^\infty((0, \infty); H_{loc}^3(\mathbb R)).
\end{aligned}
\end{equation}

Finally, for any $T>0$, using \eqref{sys1}, we know that $\Phi_t^{\tau}$ and $\Psi_t^{\tau}$ are bounded in $L^2((0,T);H^2(\Omega))$. Furthermore, using compactness theorem, for any $\alpha>0$, $(\Phi^{\tau},\Psi^{\tau})$ are relatively compact in $C([0,T];H_{loc}^{3-\alpha}(\Omega))$. Then, as $\tau\rightarrow0$, we have
\[
(\Phi^{\tau},\Psi^{\tau})\rightarrow(\Phi^0, \Psi^0)\qquad strongly\quad in \quad C([0,T];H_{loc}^{3-\alpha}(\Omega)).
\]
Thus, combining \eqref{5.2}, we have
\begin{align}\label{5.3}
(v^{\tau},u^{\tau})\rightarrow(\Phi^0+ (v^s)^{0}, \Psi^0+ (u^s)^{0})
=:(v^0, u^0),
\qquad strongly\quad in \quad C([0,T];H_{loc}^{3-\alpha}(\Omega)).
\end{align}
On the other hand, noting that $\sqrt{\tau}\Pi_1^{\tau}$ and $\sqrt{\tau}\Pi_2^{\tau}$ are uniformly bounded in $L^{\infty}((0,\infty);H^3(\Omega))$ as $\tau\rightarrow0$, it follows that
\[
 \begin{aligned}
\tau \rho \left(\partial_t\Pi_1+ u\cdot\nabla_x \Pi_1\right)\rightharpoonup 0,\qquad in \quad D^{\prime}((0,\infty)\times \Omega),\\
 \tau \rho \left(\partial_t\Pi_2+ u\cdot\nabla_x \Pi_2\right)\rightharpoonup 0,\qquad in \quad D^{\prime}((0,\infty)\times \Omega).
 \end{aligned}
 \]
Therefore, let $\tau\rightarrow0$ in \eqref{1.6}, we have
\begin{align}\label{5.4}
\Pi_1^{\tau}\rightharpoonup\mu\left(\nabla_xu^0+(\nabla_xu^0)^T-\frac{2}{3}\ddiv_x u^0\mathrm{I}_3\right),\quad
\Pi_1^{\tau}\rightharpoonup\lambda\ddiv_x u^0,
\qquad a.e. \quad (0,\infty)\times\Omega
\end{align}
and we conclude that $v^0,u^0$ are the solutions of 3D classical Navier-Stokes equations.
Then, combining \eqref{5.1}, \eqref{5.2}, \eqref{5.3} and \eqref{5.4}, we get the desired results.


\begin{thebibliography}{aaaaa}

\bibitem{BD} D. Bresch and B. Desjardins, On the construction of approximate solutions for the 2D viscous shallow water model and for compressible Navier-Stokes models, {\it J. Math. Pures Appl.} {\bf 86 }(9) (2006), 362-368.

\bibitem{DJE}{D. Chakraborty and J. E. Sader, Constitutive models for linear compressible viscoelastic flows of
simple liquids at nanometer length scales, {\it Physics of Fluids.}
{\bf 27} (2015), 052002.}



\bibitem{FRE1}{H. Freist\"{u}hler, A Galilei invariant version of Yong's model. arXiv 2012.09059 (2020).}

\bibitem{FRE2}{H. Freist\"{u}hler, Time-Asymptotic Stability for First-Order Symmetric Hyperbolic
Systems of Balance Laws in Dissipative Compressible Fluid Dynamics, {\it Quart. Appl. Math.}
{\bf 80} (2022), 597-606.}


\bibitem{GD} J. Goodman, Nonlinear asymptotic stability of viscous shock profiles for conservation laws, {\it Arch. Ration. Mech. Anal. }{\bf 5} (4) (1986), 325-344.

\bibitem{GHSR} {R. Guan and Y. Hu, Asymptotic stability of composite waves of viscous shocks for relaxed compressible Navier-Stokes equations, arXiv 2024.18480 (2024).}
\bibitem{GHSS} {R. Guan and Y. Hu, Asymptotic stability of composite waves of two viscous shock and rarefaction  for relaxed compressible Navier-Stokes equations, arXiv 2024.17261 (2024).}

\bibitem{SMJ} S. Han, M.-J. Kang and J. Kim, Large-time behavior of composite waves of viscous shocks for barotropic Navier-Stokes equations, {\it SIAM J. Math. Anal.} {\bf  55} (5) (2023), 5526-5574.


\bibitem{hufes2017}{Y. Hu and R. Racke, Compressible Navier-Stokes Equations with revised Maxwell's law, {\it J. Math. Fluid Mech.}
	{\bf 19} (2017), 77-90.}
	
\bibitem{ZWH}{Y. Hu and Z. Wang, Linear stability of viscous shock wave for 1-d compressible Navier-Stokes equations with Maxwell's law, {\it Quart. Appl. Math.}
{\bf 2} (2022), 221-235.}

\bibitem{XFH}{Y. Hu and X. Wang, Asymptotic Stability of Rarefaction Waves for Hyperbolized Compressible Navier-Stokes Equations, {\it J. Math. Fluid Mech.}
{\bf 25} (2023), 90.}


\bibitem{HM} F. Huang and A. Matsumura, Stability of a composite wave of two viscous shock waves for the full compressible Navier-Stokes equation, {\it Comm. Math. Phys.} {\bf 289} (2009), 841-861.




\bibitem{JR} S. Jiang and R. Racke, {\it Evolution equations in thermoelasticity}.  Monographs
Surveys Pure Appl. Math. 112. Chapman and Hall/CRC, Boca Raton (2000).

\bibitem{KV1} M.-J. Kang and A. Vasseur, Criteria on contractions for entropic discontinuities of systems of conservation laws, {\it Arch. Ration. Mech. Anal.} {\bf 222} (1) (2016), 343-391.

\bibitem{KV3} M.-J. Kang and A. Vasseur, Contraction property for large perturbations of shocks of the barotropic Navier-Stokes system, {\it J. Eur. Math. Soc.} {\bf 23} (2) (2021), 585-638.

\bibitem{KV6}  M.-J. Kang and A. Vasseur, $L^2$-contraction for shock waves of scalar viscous conservation laws, {\it Ann. Inst.} Henri Poincar\'{e}, Anal. Non lin\'{e}aire {\bf 34} (1) (2017), 139-156.

\bibitem{KV8}  M.-J. Kang and A. Vasseur, Uniqueness and stability of entropy shocks to the isentropic Euler system in a class of inviscid limits from a large family of Navier-Stokes systems, {\it Invent. Math.} {\bf 224} (1) (2021), 55-146.

\bibitem{KV9}  M.-J. Kang and A. Vasseur, Well-posedness of the Riemann problem with two shocks for the isentropic Euler system in a class of vanishing physical viscosity limits, {\it J. Diff. Eqs.} {\bf 338}  (2022), 128-226.

\bibitem{KV10} M.-J. Kang, A. F. Vasseur, and Y. Wang, $L^2$-contraction for planar shock waves of multi-dimensional scalar viscous conservation laws, {\it J. Diff. Eqs.} {\bf 267}  (2019), 2737-2791.

\bibitem{KV11} M.-J. Kang, A. F. Vasseur, and Y. Wang, Uniqueness of a planar contact discontinuity for 3D comprssible Euler system in a class of zero dissipation limits from Navier-Stokes-Fourier system, {\it Comm. Math. phys.}, {\bf 384}  (2021), 1751-1782.

\bibitem{WY}{M.-J. Kang, A. F. Vasseur, and Y. Wang, Time-asymptotic stability of composite waves of viscous
shock and rarefaction for barotropic Navier-Stokes equations, {\it Adv. Math.}
{\bf 419} (2023), 108963.}


\bibitem{LTP}  T.-P. Liu, Nonlinear stability of shock waves for viscous conservation laws, {\it Mem. Amer. Math. Soc.} {\bf 56}(328)  (1985), 1-108.

\bibitem{LTPZ}  T.-P. Liu and Y. Zeng, Shock waves in conservation laws with physical viscosity, {\it Mem. Amer. Math. Soc.} {\bf 234}(1101)  (2015), 1-168.

\bibitem{LWWR}  L. Li, T. Wang and Y. Wang, Wave phenomena to the three-dimensional fluid-particle model, {\it Arch. Ration. Mech. Anal.} {\bf 243} (2) (2022), 1019-1089.

\bibitem{LWW}  L. Li, T. Wang and Y. Wang, Stability of planar rarefaction wave to 3D full compressible Navier-Stokes equations, {\it Arch. Ration. Mech. Anal.} {\bf 230} (3) (2018), 911-937.

\bibitem{LW} L. Li and Y. Wang, Stability of planar rarefaction wave to two-dimensional compressible Navier-Stokes equations, {\it SIAM J. Math. Anal.} {\bf 50} (5) (2018), 4937-4963.


\bibitem{KV2}  N. Leger and A. Vasseur, Relative entropy and the stability of shocks and contact discontinuities for systems of conservation laws with non-BV perturbations, {\it Arch. Ration. Mech. Anal.} {\bf 201} (1) (2011), 271-302.

\bibitem{MNS}{A. Matsumura and K. Nishihara, On the stability of the traveling wave solutions of a one-dimensional model system for compressible viscous gas, {\it Japan. J. Appl. Math.}
{\bf 2} (1) (1985), 17-25.}

\bibitem{MNR1}{A. Matsumura and K. Nishihara, Asymptotic toward the rarefaction waves of solutions of a one-dimensional model system for compressible viscous gas, {\it Japan. J. Appl. Math.}
{\bf 3} (1986), 1-13.}

\bibitem{MNR2}{A. Matsumura and K. Nishihara, Global stability of the rarefaction waves of a one-dimensional model system for compressible viscous gas, {\it Comm. Math. Phys.}
{\bf 144} (1992), 325-335.}



\bibitem{MV20} A. Mellet and A. Vasseur, Existence and uniqueness of global strong solutions for one dimensional compressible Navier-Stokes equations, {\it SIAM J. Math. Anal.} {\bf  39} (4) (2008), 1344-1365.

\bibitem{GM}{G. Maisano, et al., Evidence of anomalous acoustic behavior from Brillouin scattering in supercooled
water, {\it Phys. Rev. Lett.}
{\bf 52} (1984), 1025-1028.}

\bibitem{MAX}{J. C. Maxwell, On the dynamical theory of gases, {\it Philos. Trans. Roy. Soc. A.}
{\bf 157} (1867), 49-88.}


\bibitem{muller1998} I. M\"uller, T. Ruggeri, {\it Rational Extended Thermodynamics. Second edition.}  Springer Tracts in Natural Philosophy, 37. Springer-Verlag, New York, (1998).

\bibitem{MP}{M. Pelton, et al., Viscoelastic flows in simple liquids generated by vibrating nanostructures, {\it Phys. Rev. Lett.}
{\bf 111} (2013), 244502.}







\bibitem{FS}{F. Sette, et al., Collective dynamics in water by high energy resolution inelastic X-ray scattering, {\it Phys. Rev. Lett.}
{\bf 75} (1995), 850-853.}

\bibitem{SXIN}  A. Szepessy and Z.P. Xin, Nonlinear stability of viscous shocks waves, {\it Arch. Ration. Mech. Anal.} {\bf 122} (1) (1993), 53-93.



\bibitem{TA} M.E. Taylor, {\it Pseudodierential Operators and Nonlinear PDE}, Progress Math., 100, Birkha\"user, Boston, 1991.

\bibitem{V27} A. Vasseur, Recent results on hydrodynamic limits, {\it Handb. Differ. Equ.}, North-Holland, Amsterdam, 2008, 323-376.

\bibitem{V28} A. Vasseur, Relative entropy and contraction for extremal shocks of conservation laws up to a shift, in: Recent Advances in Partial Differential Equations and Applications, in: {\it Contemp. Math. , vol. 666, Amer. Math. Soc.}, Providence, RI, 2016, 385-404.

\bibitem{WWS} T. Wang and Y. Wang, Nonlinear stability of planar viscous shock wave to three-dimensional compressible Navier-Stokes equations, arXiv 2204.09428, to appear in {\it J. Eur. Math. Soc.} 2022.

\bibitem{WW1} T. Wang and Y. Wang, Stability of planar rarefaction wave to the three-dimensional Boltzmann equation, {\it Kinetic and Related Models,} {\bf 12} (3) no.3 (2019) 637-679.

\bibitem{YWA}  W.A. Yong, Newtonian limit of Maxwell fluid flows, {\it Arch. Ration. Mech. Anal.} {\bf 214}  (2014), 913-922.
\end{thebibliography}
\end{document}